\documentclass[amscd]{amsart}

\usepackage[latin1]{inputenc}
\usepackage{amsmath}
\usepackage{amssymb}
\usepackage[all]{xy}
\usepackage{mathrsfs}
\usepackage{amsthm}

\usepackage{xcolor}

\usepackage{comment}

\usepackage{amscd}

\newtheorem{Thm}{Theorem}[section]
\newtheorem{Cor}[Thm]{Corollary}
\newtheorem{Lem}[Thm]{Lemma}
\newtheorem{Prop}[Thm]{Proposition}

\newtheorem{Conj}[Thm]{Conjecture}

\theoremstyle{remark}
\newtheorem{Rem}[Thm]{Remark}
\newtheorem{Def}[Thm]{Definition}

\newtheorem{Example}[Thm]{Example}

\newcommand{\cal}{\mathcal}
\newcommand{\Gr}{{{\cal G}{\frak r} }}

\newcommand{\la}{\lambda}

\newcommand\nc{\newcommand}
\nc{\HC}{{\mathcal{HC}}}
\nc{\on}{\operatorname}
\nc{\BA}{{\mathbb{A}}}
\nc{\BC}{{\mathbb{C}}}
\nc{\BG}{{\mathbb{G}}}
\nc{\BM}{{\mathbb{M}}}
\nc{\BN}{{\mathbb{N}}}
\nc{\BQ}{{\mathbb{Q}}}
\nc{\BP}{{\mathbb{P}}}
\nc{\BR}{{\mathbb{R}}}
\nc{\BZ}{{\mathbb{Z}}}
\nc{\BS}{{\mathbb{S}}}

\nc{\CA}{{\mathcal{A}}}
\nc{\CB}{{\mathcal{B}}}
\nc{\CalC}{{\mathcal C}}
\nc{\CalD}{{\mathcal D}}
\nc{\CE}{{\mathcal{E}}}
\nc{\CF}{{\mathcal{F}}}
\nc{\CG}{{\mathcal{G}}}
\nc{\CH}{{\mathcal{H}}}
\nc{\CL}{{\mathcal{L}}}
\nc{\CM}{{\mathcal{M}}}
\nc{\CMM}{{\mathcal{M}^{\operatorname{gen}}_\hbar(-\rho)}}
\nc{\CN}{{\mathcal{N}}}
\nc{\CO}{{\mathcal{O}}}
\nc{\CP}{{\mathcal{P}}}
\nc{\CQ}{{\mathcal{Q}}}
\nc{\CR}{{\mathcal{R}}}
\nc{\CS}{{\mathcal{S}}}
\nc{\CT}{{\mathcal{T}}}
\nc{\CU}{{\mathcal{U}}}
\nc{\CV}{{\mathcal{V}}}
\nc{\CW}{{\mathcal{W}}}
\nc{\CX}{{\mathcal{X}}}
\nc{\CY}{{\mathcal{Y}}}
\nc{\CZ}{{\mathcal{Z}}}

\nc{\gen}{{\operatorname{gen}}}
\nc{\cM}{{\check{\mathcal M}}{}}
\nc{\csM}{{\check{\mathcal A}}{}}
\nc{\obM}{{\overset{\circ}{\mathbf M}}{}}
\nc{\oCA}{{\overset{\circ}{\mathcal A}}{}}
\nc{\obA}{{\overset{\circ}{\mathbf A}}{}}
\nc{\ooM}{{\overset{\circ}{M}}{}}
\nc{\osM}{{\overset{\circ}{\mathsf M}}{}}
\nc{\vM}{{\overset{\bullet}{\mathcal M}}{}}
\nc{\nM}{{\underset{\bullet}{\mathcal M}}{}}
\nc{\obD}{{\overset{\circ}{\mathbf D}}{}}
\nc{\cp}{{\overset{\circ}{\mathbf p}}{}}
\nc{\ofZ}{{\overset{\circ}{\mathfrak Z}}{}}

\nc{\fb}{{\mathfrak{b}}}
\nc{\fg}{{\mathfrak{g}}}
\nc{\fgl}{{\mathfrak{gl}}}
\nc{\fh}{{\mathfrak{h}}}
\nc{\fj}{{\mathfrak{j}}}
\nc{\fn}{{\mathfrak{n}}}
\nc{\fu}{{\mathfrak{u}}}
\nc{\frr}{{\mathfrak{r}}}
\nc{\fs}{{\mathfrak{s}}}
\nc{\fT}{{\mathfrak{T}}}
\nc{\ofT}{{\overline{\mathfrak T}}}
\nc{\ofS}{{\overline{\mathfrak S}}}
\nc{\fsl}{{\mathfrak{sl}}}
\nc{\hsl}{{\widehat{\mathfrak{sl}}}}
\nc{\hgl}{{\widehat{\mathfrak{gl}}}}
\nc{\hg}{{\widehat{\mathfrak{g}}}}
\nc{\chg}{{\widehat{\mathfrak{g}}}{}^\vee}
\nc{\hn}{{\widehat{\mathfrak{n}}}}
\nc{\chn}{{\widehat{\mathfrak{n}}}{}^\vee}

\nc{\fA}{{\mathfrak{A}}}
\nc{\fB}{{\mathfrak{B}}}
\nc{\fD}{{\mathfrak{D}}}
\nc{\fE}{{\mathfrak{E}}}
\nc{\fF}{{\mathfrak{F}}}
\nc{\fG}{{\mathfrak{G}}}
\nc{\fI}{{\mathfrak{I}}}
\nc{\fJ}{{\mathfrak{J}}}
\nc{\fK}{{\mathfrak{K}}}
\nc{\fL}{{\mathfrak{L}}}
\nc{\fM}{{\mathfrak{M}}}
\nc{\fN}{{\mathfrak{N}}}
\nc{\frP}{{\mathfrak{P}}}
\nc{\fQ}{{\mathfrak Q}}
\nc{\fS}{{\mathfrak S}}
\nc{\fU}{{\mathfrak{U}}}
\nc{\fZ}{{\mathfrak{Z}}}

\nc{\ba}{{\mathbf{a}}}
\nc{\bb}{{\mathbf{b}}}
\nc{\bc}{{\mathbf{c}}}
\nc{\bd}{{\mathbf{d}}}
\nc{\be}{{\mathbf{e}}}
\nc{\bi}{{\mathbf{i}}}
\nc{\bj}{{\mathbf{j}}}
\nc{\bn}{{\mathbf{n}}}
\nc{\bp}{{\mathbf{p}}}
\nc{\br}{{\mathbf{r}}}
\nc{\bv}{{\mathbf{v}}}
\nc{\bx}{{\mathbf{x}}}
\nc{\by}{{\mathbf{y}}}
\nc{\bw}{{\mathbf{w}}}
\nc{\bA}{{\mathbf{A}}}
\nc{\bB}{{\mathbf{B}}}
\nc{\bC}{{\mathbf{C}}}
\nc{\bD}{{\mathbf{D}}}
\nc{\bE}{{\mathbf{E}}}
\nc{\bK}{{\mathbf{K}}}
\nc{\bH}{{\mathbf{H}}}
\nc{\bM}{{\mathbf{M}}}
\nc{\bN}{{\mathbf{N}}}
\nc{\bO}{{\mathbf{O}}}
\nc{\bQ}{{\mathbf Q}}
\nc{\bS}{{\mathbf{S}}}
\nc{\bV}{{\mathbf{V}}}
\nc{\bW}{{\mathbf{W}}}
\nc{\bX}{{\mathbf{X}}}
\nc{\bZ}{{\mathbf{Z}}}

\nc{\sA}{{\mathsf{A}}}
\nc{\sB}{{\mathsf{B}}}
\nc{\sC}{{\mathsf{C}}}
\nc{\sD}{{\mathsf{D}}}
\nc{\sF}{{\mathsf{F}}}
\nc{\sK}{{\mathsf{K}}}
\nc{\sM}{{\mathsf{M}}}
\nc{\sO}{{\mathsf{O}}}
\nc{\sQ}{{\mathsf{Q}}}
\nc{\sP}{{\mathsf{P}}}
\nc{\sV}{{\mathsf{V}}}
\nc{\sW}{{\mathsf{W}}}
\nc{\sZ}{{\mathsf{Z}}}
\nc{\sfp}{{\mathsf{p}}}
\nc{\sr}{{\mathsf{r}}}
\nc{\st}{{\mathsf{t}}}
\nc{\sfb}{{\mathsf{b}}}
\nc{\sfc}{{\mathsf{c}}}
\nc{\sd}{{\mathsf{d}}}
\nc{\sg}{{\mathsf{g}}}
\nc{\sk}{{\mathsf{k}}}
\nc{\sfl}{{\mathsf{l}}}

\nc{\tA}{{\widetilde{\mathbf{A}}}}
\nc{\tB}{{\widetilde{\mathcal{B}}}}
\nc{\tg}{{\widetilde{\mathfrak{g}}}}
\nc{\tG}{{\widetilde{G}}}
\nc{\TM}{{\widetilde{\mathbb{M}}}{}}
\nc{\tN}{{\widetilde{\mathcal{N}}}{}}
\nc{\tO}{{\widetilde{\mathsf{O}}}{}}
\nc{\tU}{{\widetilde{\mathfrak{U}}}{}}
\nc{\TZ}{{\tilde{Z}}}
\nc{\tZ}{\widetilde{Z}{}}
\nc{\tx}{{\tilde{x}}}
\nc{\tbv}{{\tilde{\bv}}}
\nc{\tfP}{{\widetilde{\mathfrak{P}}}{}}
\nc{\tz}{{\tilde{\zeta}}}
\nc{\tmu}{{\tilde{\mu}}}

\nc{\td}{\ddot{\underline{d}}{}}
\nc{\tzeta}{\widetilde{\zeta}{}}
\nc{\hd}{{\widehat{\underline{d}}}}
\nc{\hG}{{\widehat{G}}}
\nc{\hBP}{\widehat{\mathbb P}{}}
\nc{\hQ}{{\widehat{Q}}}
\nc{\hsM}{\widehat{\mathsf M}{}}
\nc{\hfM}{\widehat{\mathfrak M}{}}
\nc{\hCP}{\widehat{\mathcal P}{}}
\nc{\hCR}{\widehat{\mathcal R}{}}
\nc{\hCS}{{\widehat{\mathcal S}}}
\nc{\hfZ}{\widehat{\mathfrak Z}{}}
\nc{\hZ}{\widehat{Z}{}}

\nc{\urho}{\underline{\rho}}
\nc{\uB}{\underline{B}}
\nc{\uC}{{\underline{\mathbb{C}}}}
\nc{\ui}{\underline{i}}
\nc{\ofP}{{\overline{\mathfrak{P}}}}

\nc{\hrho}{{\hat{\rho}}}

\nc{\unl}{\underline}
\nc{\ol}{\overline}
\nc{\one}{{\mathbf{1}}}
\nc{\two}{{\mathbf{t}}}

\nc{\Sym}{{\mathop{\operatorname{Sym}}}}
\nc{\Tot}{{\mathop{\operatorname{\normalshape Tot}}}}
\nc{\Hilb}{{\mathop{\operatorname{\normalshape Hilb}}}}
\nc{\Hom}{{\mathop{\operatorname{Hom}}}}
\nc{\CHom}{{\mathop{\operatorname{{\mathcal{H}}\it om}}}}
\nc{\defi}{{\mathop{\operatorname{\normalshape def}}}}
\nc{\length}{{\mathop{\operatorname{\normalshape length}}}}

\nc{\Cliff}{{\mathsf{Cliff}}}
\nc{\Fib}{{\mathsf{Fib}}}
\nc{\Coh}{{\mathsf{Coh}}}
\nc{\FCoh}{{\mathsf{FCoh}}}

\nc{\reg}{{\text{\normalshape reg}}}
\nc{\res}{{\operatorname{res}}}
\nc{\cplus}{{\mathbf{C}_+}}
\nc{\cminus}{{\mathbf{C}_-}}
\nc{\cthree}{{\mathbf{C}_*}}
\nc{\Qbar}{{\bar{Q}}}

\nc{\bh}{{\bar{h}}}
\nc{\bOmega}{{\overline{\Omega}}}
\nc\tGr{\widetilde{\Gr}}

\nc{\seq}[1]{\stackrel{#1}{\sim}}
\nc\ogu{\overline{G/U}}
\nc\chlam{\check{\lam}}

\nc{\oZ}{{\overset{\circ}{Z}}}
\nc{\tF}{\widetilde{\mathcal F}}

\nc\uS{\underline{S}}
\nc\QM{\mathcal{QM}}

\nc{\chmu}{\check{\mu}}
\newcommand\iso{\,\vphantom{j^{X^2}}\smash{\overset{\sim}{\vphantom{\rule{0pt}{0.20em}}\smash{\longrightarrow}}}\,}
\nc{\ul}{\underline}
\nc{\Mvd}{\mathfrak{M}(\underline{v},\underline{d})}
\nc{\MvdT}{\mathfrak{M}(\underline{v}^{\dagger},\underline{d}^{\dagger})}
\nc{\MVD}{\mathfrak{M}(V,D)}
\nc{\mt}{\mapsto}
\nc{\sm}{\setminus}
\nc{\ra}{\rightarrow}
\nc{\lar}{\leftarrow}
\nc{\hr}{\hookrightarrow}
\nc{\Lap}{\Lambda^{+}}
\nc{\oZal}{\overset{\circ}{Z^{\alpha}}}
\nc{\sig}{\sigma}
\nc{\al}{\alpha}
\nc{\is}{\simeq}
\nc{\ip}{\iota^{+}_{\la, \mu}}
\nc{\im}{\iota^{-}_{\la, \mu}}
\nc{\jp}{j^{+}_{\la, \mu}}
\nc{\jm}{j^{-}_{\la, \mu}}
\nc{\pip}{\pi^{+}_{\la, \mu}}
\nc{\pim}{\pi^{-}_{\la, \mu}}
\nc{\s}{\star}
\nc{\fpt}{[A^{\la},B^{\la},\gamma^{\la},\delta^{\la}]}
\nc{\ulfpt}{[\ul{A}^{\la},\ul{B}^{\la},\ul{\gamma}^{\la},\ul{\delta}^{\la}]}
\nc{\lvee}{\!\scriptscriptstyle\vee}
\nc{\rra}{\twoheadrightarrow}

\nc{\End}{\on{End}}

\nc{\RHom}{R\mathcal{H}om}

\nc{\yg}{Y(\fg)}
\nc{\yvg}{Y_V(\fg)}
\nc{\CAg}{\CA_{\fg}}
\nc{\Ag}{A_{\fg}}
\nc{\Sgt}{S^{\bullet}(\fg[t])}
\nc{\Sg}{S^{\bullet}(\fg)}
\nc{\Ugt}{U(\fg[t])}

\nc{\Spec}{\operatorname{Spec}}

\title[Subregular orbits and explicit character formulas]{Subregular nilpotent orbits and explicit character formulas for modules over affine Lie algebras}

\author{Roman Bezrukavnikov}
\address{Department of Mathematics
Massachusetts Institute of Technology
\newline
77 Massachusetts Avenue,
Cambridge, MA 02139,
USA
}
\email{bezrukav@math.mit.edu}

\author{Victor Kac}
\address{Department of Mathematics
Massachusetts Institute of Technology
\newline
77 Massachusetts Avenue,
Cambridge, MA 02139,
USA
}
\email{kac@math.mit.edu}

\author{Vasily Krylov}
\address{Department of Mathematics
Massachusetts Institute of Technology
\newline
77 Massachusetts Avenue,
Cambridge, MA 02139,
USA;
\newline National Research University Higher School of Economics, Russian Federation\newline
Department of Mathematics, 6 Usacheva st., Moscow 119048;
}
\email{krvas@mit.edu, krylovasya@gmail.com}

\begin{document}
	\begin{abstract}
	Let $\mathfrak{g}$ be a simple finite dimensional complex Lie algebra
 and let $\widehat{\mathfrak{g}}$ be the corresponding affine Lie algebra. Kac and Wakimoto observed 
	that in some cases the coefficients in the character formula for a simple highest weight $\widehat{\mathfrak{g}}$-module are either bounded or are given by a linear function of the weight. We explain and generalize this observation using Kazhdan-Lusztig theory, by
	 computing values at $q=1$ of certain  (parabolic) affine inverse Kazhdan-Lusztig polynomials.
In particular, we obtain explicit character formulas for some $\widehat{\mathfrak{g}}$-modules  of negative integer level $k$ when $\mathfrak g$ is of type $D_n$, $E_6$, $E_7$, $E_8$ and $k \geqslant -2, -3, -4, -6$ respectively, as conjectured by Kac and Wakimoto.
  
  The calculation relies on the explicit description of the canonical basis in the cell quotient
	 of the anti-spherical module over the affine Hecke algebra  corresponding to the subregular cell. We also present an explicit description of the corresponding objects in the derived category of equivariant coherent sheaves on the Springer resolution, they 
	 correspond to irreducible objects in the heart of a certain $t$-structure related to the so called non-commutative Springer resolution. 
	\end{abstract}

\maketitle

\centerline{{\em To Corrado De Concini, with admiration }}

\section{Introduction}

\subsection{}

Let $\mathfrak{g}$ be a simple finite dimensional Lie algebra over
${\mathbb{C}}$ and let
$\widehat{\mathfrak{g}}=\mathfrak{g}[t^{\pm 1}] \oplus {\mathbb{C}}K
\oplus {\mathbb{C}}d$ be the corresponding affine Lie algebra.

Computing characters of highest weight $\widehat{\mathfrak{g}}$-modules
is a classical problem of representation theory. For example, for integrable
modules the answer is given by the Kac character formula which can be viewed
as a direct generalization of the Weyl character formula for characters
of finite dimensional representations of $\mathfrak{g}$. In the general
case, the character formula involves the affine Kazhdan-Lusztig polynomials
(or rather their evaluation at $q=1$). While conceptually deep and algorithmically
computable, this answer is much more complicated than the explicit expression
appearing in the Kac character formula. It is unlikely that an essential
simplification is possible in general, however, it is interesting to explore
special cases when a simple character formula exists. In particular, in
\cite{KaW} the second author and Wakimoto identified (partly conjecturally)
cases when the coefficients in the sum appearing in the character formula
either take values $0$, $\pm 1$ (for ${\frak g}$ of type $A_{n}$,
$n\geqslant 2$) or depend {\em linearly} on the indexing weight (for
${\frak g}$ of type $D$ or $E$). In the present paper we partly prove their
conjecture and extend their results, while connecting it to the Kazhdan-Lusztig
theory. Instead of working with the (parabolic) Kazhdan-Lusztig polynomials
directly, we analyze them using their relation to the Grothendieck group
of equivariant coherent sheaves on the Springer resolution constructed
from the Langlands dual group. Elementary geometric properties of the Springer
resolution provide a transparent explanation for algebraic properties of
Kazhdan-Lusztig polynomials and allow one to compute some of them effectively.

We should mention that most information about the canonical bases we use
is already contained in Lusztig's work \cite{lu_subreg,lu_notes_aff}. The new result in this direction we provide is the
realization of the basis elements as classes of explicit objects in the
derived category of coherent sheaves, which arise as irreducible objects
in the heart of a certain $t$-structure.

In order to present the content of this work in more detail we introduce
further notation. Let $\mathfrak{g}^{\vee}$ be the Langlands dual Lie algebra.
Let ${\mathcal{N}}\subset \mathfrak{g}^{\vee}$ be the nilpotent cone. Let
$\mathfrak{h} \subset \mathfrak{g}$ be a Cartan subalgebra,
$Q^{\vee} \subset \mathfrak{h}$ the coroot lattice, and
$W \subset \operatorname{End}(\mathfrak{h})$ the Weyl group of
$\mathfrak{g}$. Let $G^{\vee}$ be the adjoint group with Lie algebra
$\mathfrak{g}^{\vee}$.

Let $\widehat{\mathfrak{h}} \subset \widehat{\mathfrak{g}}$ be the Cartan
subalgebra of $\widehat{\mathfrak{g}}$, containing $\mathfrak{h}$. Let
$\widehat{W}=W \ltimes Q^{\vee}$ be the affine Weyl group. For
$\gamma \in Q^{\vee}$ we denote by $t_{\gamma}$ the corresponding element
of $\widehat{W}$ and by $w_{\gamma} \in \widehat{W}$ the shortest element
of the coset $t_{\gamma}W \subset \widehat{W}$.

\subsection{Characters of certain irreducible $\widehat{\mathfrak{g}}$-modules}
\label{appl_repres_th}

Recall the notion of two-sided cells in $\widehat{W}$, which are certain
subsets in $\widehat{W}$ (the main reference is \cite{Lu1}, see also
\cite{hum} for a short exposition). There exists a canonical bijection
between the set of two-sided cells in $\widehat{W}$ and $G^{\vee}$-orbits
on ${\cal N}$ (see \cite{Lu}). We denote by
${\mathbb{O}}_{c} \subset {\mathcal{N}}$ the nilpotent orbit, corresponding
to a cell $c \subset \widehat{W}$.

Let $c=c_{\mathrm{subreg}} \subset \widehat{W}$ be the cell, corresponding
to the subregular nilpotent orbit. For $\nu \in Q^{\vee}$ let
$c_{\nu }\subset \widehat{W}$ be the two-sided cell that contains
$w_{\nu}$.

We can now describe the main results of this paper. Pick
$\Lambda \in \widehat{\mathfrak{h}}^{*}$. Let $L(\Lambda )$ be the irreducible
$\widehat{\mathfrak{g}}$-module with highest weight $\Lambda $. Assume
that the level of $L(\Lambda )$ is greater than $-h^{\vee}$ and that
$\Lambda +\widehat{\rho}$ is integral quasi-dominant (these notions are
defined in Section~\ref{recall_irred_hw_po_level} and Definition~\ref{def_la_reg_int_dom_quasid}). Let $w \in \widehat{W}$ be the longest
element such that $w(\Lambda +\widehat{\rho})$ is dominant.

It follows from \cite[Section 0.3]{KT} (see Theorem~\ref{canon_in_Hecke} and Equation (\ref{char_our_L_via_repres_short_long}))
that the character of $L(\Lambda )$ can be expressed in terms of values
at $q=1$ of affine inverse Kazhdan-Lusztig polynomials
${\mathbf{m}}^{w_{\gamma}}_{w}(q)$ for $\gamma \in Q^{\vee}$ (see
\cite[Section 2]{kl01} or Appendix~\ref{basic_KL} for the definitions).
Using this observation, we derive explicit formulas for characters of
$L(\Lambda )$ such that the corresponding $w$ lies in $c$ and is equal
to $w_{\nu}$ for some $\nu \in Q^{\vee}$ (see Theorems~\ref{main_th_formulation},~\ref{main_th_formulation_A}). We describe such
$\Lambda $ explicitly and compare formulas that we obtain with the results
of \cite{KaW}, partly proving \cite[Conjecture 3.2]{KaW} (see Propositions~\ref{sing_our_char},~\ref{reg_our_char},~\ref{sing_our_char_A},~\ref{reg_our_char_A}). Let us now describe the approach that we use to
compute the values
${\bf{m}}^{w_{\gamma}}_{w_{\nu}}(1)={\bf{m}}^{w_{\gamma}}_{w_{\nu}}$.

Consider the $\widehat{W}$-module
$ M:={\mathbb Z}\widehat{W} \otimes _{{\mathbb Z}W} \mathbb{Z}_{
\mathrm{sign}} $, called the anti-spherical $\widehat{W}$-module. This
module contains a {\textit{standard}} basis
$\{T_{\gamma}\mid  \gamma \in Q^{\vee}\}$ and a {\textit{canonical}} basis
$\{C_{\gamma}\mid  \gamma \in Q^{\vee}\}$ in the sense of Kazhdan-Lusztig
(see Appendix~\ref{basic_KL} for details). By definition,
%
\begin{equation}
\label{detrm_m_first}
t_{\gamma }\cdot 1=T_{\gamma }= \sum _{\nu \in Q^{\vee}} {\bf{m}}_{w_{
\nu}}^{w_{\gamma}} C_{\nu}.
\end{equation}
We explicitly describe a certain quotient of the module $M$ and then consider
the image of (\ref{detrm_m_first}) in this quotient to determine the numbers
${\bf{m}}_{w_{\nu}}^{w_{\gamma}}$ for $w_{\nu }\in c$ (see the next section
for more details).

\subsection{Modules over $\widehat{W}$ via Springer theory}

Let us recall the ``coherent'' realization of the module $M$ and then describe
the approach that we use to compute
${\bf{m}}_{w_{\nu}}^{w_{\gamma}}$ for $w_{\nu }\in c$.

Let $\pi \colon \widetilde{{\mathcal{N}}} \rightarrow {\mathcal{N}}$ be
the Springer resolution. For $\gamma \in Q^{\vee}$ we denote by
${\mathcal{O}}_{{\cal B}}(\gamma )$ the corresponding line bundle on the
flag variety ${\cal B}$ of $G^{\vee}$ and by
${\mathcal{O}}_{\widetilde{{\mathcal{N}}}}(\gamma )$ its pull back to
$\widetilde{{\mathcal{N}}}$. For every $G^{\vee}$-invariant locally closed
subvariety $X \subset {\mathcal{N}}$ there is a natural action
$\widehat{W} \curvearrowright K^{G^{\vee}}(\pi ^{-1}(X))$ (see
\cite{lu_bases_K}), where by $ K^{G^{\vee}}(\pi ^{-1}(X))$ we denote the
Grothendieck group of $G^{\vee}$-equivariant coherent sheaves on
$\pi ^{-1}(X)$.

\begin{Rem}
This action comes from the identification
${\mathbb{Z}}\widehat{W} \iso K^{G^{\vee}}(\widetilde{{\mathcal{N}}}
\times _{{\mathcal{N}}} \widetilde{{\mathcal{N}}})$ (see
\cite[Sections 7, 8]{lu_bases_K} or \cite{CG,kl}), and the algebra
structure on
$K^{G^{\vee}}(\widetilde{{\mathcal{N}}} \times _{{\mathcal{N}}}
\widetilde{{\mathcal{N}}})$ is given by convolution. The algebra
$K^{G^{\vee}}(\widetilde{{\mathcal{N}}} \times _{{\mathcal{N}}}
\widetilde{{\mathcal{N}}})$ acts naturally on
$K^{G^{\vee}}(\pi ^{-1}(X))$.
\end{Rem}

It is known (see, for example, \cite{CG}) that
$K^{G^{\vee}}(\widetilde{{\mathcal{N}}})$ is isomorphic to the anti-spherical
$\widehat{W}$-module $M$. The standard basis of
$M \simeq K^{G^{\vee}}(\widetilde{{\mathcal{N}}})$ can be described explicitly:
it consists of classes of line bundles
${\mathcal{O}}_{\widetilde{{\mathcal{N}}}}(\gamma )$,
$\gamma \in Q^{\vee}$. The canonical basis does not have any explicit description.
It can be shown that the canonical basis consists of classes of irreducible
objects of the heart of the ``exotic'' $t$-structure on the derived category
$D^{b}(\operatorname{Coh}^{G^\vee}(\widetilde{{\mathcal{N}}}))$ (see
\cite{BHum} and \cite{BHe} for details).

Let $U \subset {\mathcal{N}}$ be an open $G^{\vee}$-invariant subvariety.
Set $\widetilde{U}:=\pi ^{-1}(U)$. It follows from
\cite[\S 11.3]{BHe} that the kernel of the restriction homomorphism
$K^{G^{\vee}}(\widetilde{{\mathcal{N}}}) \twoheadrightarrow K^{G^{
\vee}}(\widetilde{U})$ is freely generated (as a module over
${\mathbb{Z}}$) by elements $C_{\nu}$ such that
${\mathbb{O}}_{c_{\nu}} \not \subset U$. In particular,
$K^{G^{\vee}}(\widetilde{U})$ admits a canonical basis parametrized by
$\{\nu \in Q^{\vee}\mid  {\mathbb{O}}_{c_{\nu}} \subset U\}$.

Recall now that our goal is to compute the numbers
${\bf{m}}^{w_{\gamma}}_{w_{\nu}}$ for $w_{\nu }\in c$, where
$c \subset \widehat{W}$ is the cell, corresponding to the subregular nilpotent.
The numbers ${\bf{m}}^{w_{\gamma}}_{w_{\nu}}$ are determined by equation
(\ref{detrm_m_first}) as follows.

Consider
$U={\mathbb{O}}_{e} \cup {\mathbb{O}}^{\mathrm{reg}} \subset {
\mathcal{N}}$, where ${\mathbb{O}}_{e} \subset {\mathcal{N}}$ is the
$G^{\vee}$-orbit of a subregular nilpotent element
$e \in {\mathcal{N}}$. It follows from the above that the canonical basis
in $K^{G^{\vee}}(\widetilde{U})$ is parametrized by
$\{1\} \cup \{\nu \in Q^{\vee}\mid  w_{\nu} \in c\}$. For
$\gamma , \nu \in Q^{\vee}$ such that $w_{\nu }\in c \cup \{1\}$ let
$\bar{T}_{\gamma}$ and $\bar{C}_{\nu}$ be the images of $T_{\gamma}$ and
$C_{\nu}$ in $K^{G^{\vee}}(\widetilde{U})$. Taking the image of (\ref{detrm_m_first})
in $K^{G^{\vee}}(\widetilde{U})$, we conclude that
\begin{equation*}
t_{\gamma }\cdot \bar{1}=\bar{T}_{\gamma}=\sum _{\nu \in Q^{\vee},\, w_{
\nu }\in c \cup \{1\}} {\bf{m}}^{w_{\gamma}}_{w_{\nu}} \bar{C}_{\nu}.
\end{equation*}
So, to compute the coefficients ${\bf{m}}^{w_{\gamma}}_{w_{\nu}}$ as above,
it is enough to describe the $\widehat{W}$-module structure on
$K^{G^{\vee}}(\widetilde{U})$ and the canonical basis
$\{\bar{C}_{\nu}\mid  w_{\nu }\in c \cup \{1\}\}$ in it. Let us describe
the answer.

The cell $c$ has an explicit description (see
\cite[Proposition 3.8]{lu_subreg_crit}): it consists of elements
$w \in \widehat{W}$ with {\textit{unique}} reduced decomposition. The subset
$\{w \in c\mid  w=w_{\nu}~\text{for some}~\nu \in Q^{\vee}\}$ consists of
elements $w \in c$ such that the reduced decomposition of $w$ ends by
$s_{0}$ (simple reflection, corresponding to the $0$th vertex of the Dynkin
diagram of $\widehat{\mathfrak{g}}$). This set can be described explicitly
(see Corollary~\ref{descr_nu_i_subreg} and Lemma~\ref{descr_subreg_cell_type_A}). It turns out that for
$\mathfrak{g}$ of type $D$, $E$ it is parametrized by the set
$\widehat{I}$ of vertices of the Dynkin diagram of
$\widehat{\mathfrak{g}}$, and for $\mathfrak{sl}_{n}$ ($n \geqslant 3$)
it is parametrized by ${\mathbb{Z}}$ (that should be considered as the
set of vertices of the Dynkin diagram $A_{\infty}$).

We prove (see Proposition~\ref{ident_K_U_DE}) that in types $D$, $E$ the
$\widehat{W}$-module\break
$K^{G^{\vee}}(\widetilde{U}) \otimes {\mathbb{Z}}_{\mathrm{sign}}$ is isomorphic
to the integral form $\widehat{\mathfrak{h}}_{{\mathbb{Z}}}$ of the Cartan
subalgebra $\widehat{\mathfrak{h}} \subset \widehat{\mathfrak{g}}$, i.e.
that
\begin{equation*}
K^{G^{\vee}}(\widetilde{U}) \simeq ({\mathbb{Z}}Q^{\vee }\oplus {
\mathbb{Z}}K \oplus {\mathbb{Z}}d) \otimes {\mathbb{Z}}_{
\mathrm{sign}}
\end{equation*}
as $\widehat{W}$-modules. After this identification the canonical basis
consists of $d$ and minus simple coroots of $\widehat{\mathfrak{g}}$.

In type $A$ we identify $K^{\operatorname{PGL}_{n}}({\cal B}_{e})$ with
$(\mathfrak{h}_{\infty ,{\mathbb{Z}}} \oplus {\mathbb{Z}}d) \otimes {
\mathbb{Z}}_{\mathrm{sign}}$ and describe explicitly the
$\widehat{W}$-action on the latter (see Section~\ref{ext_h_infty} and Proposition~\ref{descr_K_U_A}). After this identification the canonical basis consists
of $d$ and minus simple coroots of $\mathfrak{sl}_{\infty}$.

\subsection{Structure of the paper}

The paper is organized as follows. In
Section~\ref{sect_affine_Lie_alg_rep_theory} we recall the structure theory
and representation theory of affine Lie algebras (see
Section~\ref{prelim_affine_lie}), we then give character formulas for certain
$L(\Lambda )$ (see Theorem~\ref{char_KW_thm} for types $D$, $E$ and
Theorem~\ref{main_th_formulation_A} for type $A$) and rewrite them in more
explicit terms (see Propositions~\ref{sing_our_char},~\ref{reg_our_char} for
types $D$, $E$ and Propositions~\ref{sing_our_char_A},~\ref{reg_our_char_A}
for type~$A$). In Section~\ref{sect_cat_O_and_KL} we recall categories
${\mathcal{O}}$ for $\widehat{\mathfrak{g}}$ and describe characters of
irreducible $\widehat{\mathfrak{g}}$-modules via values at $q=1$ of (affine)
inverse Kazhdan-Lusztig polynomials. In Section~\ref{sect_geom_springer_KL}
we recall the Springer resolution and the geometric realization of the
anti-spherical $\widehat{W}$-module $M$, we also recall some information
about the canonical basis, in particular, we describe explicitly the
canonical basis of $K^{G^{\vee}}(\widetilde{U})$. In
Section~\ref{subreg_section} we describe $\widehat{W}$-module $K^{G^{\vee}}(\widetilde{U})$
explicitly for $\mathfrak{g}$ of type $D$, $E$ (see
Proposition~\ref{ident_K_U_DE}). We then compute
${\bf{m}}^{w_{\gamma}}_{w_{\nu}}$ for $w_{\nu }\in c$ and derive
Theorem~\ref{main_th_formulation} (see Section~\ref{comp_KL_values_DE}). In
Section~\ref{sect_subreg_A} we describe $\widehat{W}$-module 
$K^{\operatorname{PGL}_{n}}(\widetilde{U})$ explicitly (see
Proposition~\ref{descr_K_U_A}). We then compute
${\bf{m}}^{w_{\gamma}}_{w_{\nu}}$ for $w_{\nu }\in c$ and derive
Theorem~\ref{main_th_formulation_A} (see Section~\ref{m_th_A}). In
Section~\ref{poss_gener} we discuss possible generalizations.
Appendix~\ref{basic_KL} contains the information about Kazhdan-Lusztig
bases that we use.

\section{Affine Lie algebras and their representation theory}
\label{sect_affine_Lie_alg_rep_theory}

\subsection{Affine Lie algebras: structure theory and irreducible highest weight modules \cite{ka}}
\label{prelim_affine_lie}

\subsubsection{Simple Lie algebra $\mathfrak{g}$: notations}

Let $\mathfrak{g}$ be a simple finite dimensional Lie algebra over
${\mathbb{C}}$. We fix a Catan subalgebra
$\mathfrak{h} \subset \mathfrak{g}$ and denote by $\Delta $ the set of
roots of $(\mathfrak{h},\mathfrak{g})$ and by $W$ the Weyl group of
$(\mathfrak{h},\mathfrak{g})$. Let $Q$ be the root lattice of
$\mathfrak{g}$. Let $\alpha _{1},\ldots ,\alpha _{r} \in \Delta $ be a
set of simple roots, and $\theta \in \Delta $ be the highest root. We denote
by $\Delta _{+} \subset \Delta $ the set of positive roots and set
$\rho :=\frac{1}{2}\sum _{\alpha \in \Delta _{+}}\alpha $. We fix the nondegenerate
invariant symmetric bilinear form $(\,,\,)$ on $\mathfrak{g}$ normalized
by $(\theta ,\theta )=2$. Let
$\alpha ^{\vee}_{1},\ldots ,\alpha ^{\vee}_{r} \in \mathfrak{h}$ be the
simple coroots defined by
\begin{equation*}
\langle \alpha _{j},\alpha ^{\vee}_{i} \rangle = a_{ij},
\end{equation*}
where $A=(a_{ij})_{i,j=1,\ldots ,r}$ is the Cartan matrix of
$\mathfrak{g}$. Let $\Delta ^{\vee}$ be the $W$-orbit of
$\{\alpha ^{\vee}_{1},\ldots ,\alpha ^{\vee}_{r}\}$. We also denote by
$\theta ^{\vee} \in \mathfrak{h}$ the highest coroot of
$\Delta ^{\vee}$.

\subsubsection{Affine Lie algebra $\widehat{\mathfrak{g}}$}

We denote by $\widehat{\mathfrak{g}}$ the affine Lie algebra, corresponding
to $\mathfrak{g}$. Recall that
\begin{equation*}
\widehat{\mathfrak{g}}:=\mathfrak{g}[t^{\pm 1}] \oplus {\mathbb{C}}K
\oplus {\mathbb{C}}d
\end{equation*}
with the bracket defined as follows ($a,b \in \mathfrak{g}$,
$m,n \in {\mathbb{Z}}$):
\begin{equation*}
[at^{m},bt^{n}]=[a,b]t^{m+n}+m(a,b)\delta _{m,-n}K,\, [d,at^{n}]=nat^{n},
\, [K,\widehat{\mathfrak{g}}]=0.
\end{equation*}

The Lie algebra $\widehat{\mathfrak{g}}$ has a nondegenerate invariant
symmetric bilinear form $(\,,\,)$ defined by
\begin{gather*}
(a t^{m},b t^{n})=\delta _{m,-n}(a,b),\, ({\mathbb{C}}K \oplus {
\mathbb{C}}d, \mathfrak{g}[t^{\pm 1}])=0,\,%
\\
(K,K)=(d,d)=0,\, (K,d)=1.%
\end{gather*}
This bilinear form restricts to a nondegenerate bilinear form on the Cartan
subalgebra of $\widehat{\mathfrak{g}}$:
\begin{equation*}
\widehat{\mathfrak{h}}:=\mathfrak{h} \oplus {\mathbb{C}}K \oplus {
\mathbb{C}}d.
\end{equation*}

We extend every $\gamma \in \mathfrak{h}^{*}$ to the linear function on
$\widehat{\mathfrak{h}}$ by setting
$\langle \gamma , {\mathbb{C}}K \oplus {\mathbb{C}}d \rangle =0$. Let
$\delta \in \widehat{\mathfrak{h}}^{*}$ be the linear function given by
$\langle \delta ,\mathfrak{h} \oplus {\mathbb{C}}K \rangle =0$,
$\langle \delta ,d \rangle = 1$. Set
$\alpha _{0}:=\delta -\theta \in \mathfrak{h}^{*}$,
$\alpha ^{\vee}_{0}:=K-\theta ^{\vee} \in \mathfrak{h}$. Then
$\{\alpha _{0},\alpha _{1},\ldots ,\alpha _{r}\}$ are simple roots of
$\widehat{\mathfrak{g}}$ and
$\{\alpha _{0}^{\vee},\alpha _{1}^{\vee},\ldots ,\alpha _{r}^{\vee}\}$
are simple coroots. Define the fundamental weights
$\Lambda _{i} \in \widehat{\mathfrak{h}}^{*}$ by
\begin{equation*}
\langle \Lambda _{i}, \alpha _{j}^{\vee }\rangle := \delta _{i,j},\, i,j=0,1,
\ldots ,r.
\end{equation*}

We denote by
$\eta \colon \widehat{\mathfrak{h}} \iso \widehat{\mathfrak{h}}^{*}$ the
identification induced by the bilinear form $(\,,\,)$.

From now on, we assume for simplicity that $\mathfrak{g}$ is of type
$A, D, E$. So we have
\begin{equation*}
\eta (\alpha _{i}^{\vee})=\alpha _{i}, \,\eta (\theta ^{\vee})=
\theta ,\,\eta (\alpha ^{\vee}_{0})=\alpha _{0}, \,\eta (K)=\delta ,
\, \eta (d)=\Lambda _{0},\, i=1,\ldots ,r.
\end{equation*}

\subsubsection{Root system of $\widehat{\mathfrak{g}}$}

Let $\widehat{\Delta}$ be the root system of
$\widehat{\mathfrak{g}}$. Recall that $\widehat{\Delta}$ can be decomposed
as a disjoint union of the sets of {\textit{real}} and {\textit{imaginary}} roots:
\begin{equation*}
\widehat{\Delta}=\widehat{\Delta}^{\mathrm{re}} \cup \widehat{\Delta}^{
\mathrm{im}},
\end{equation*}
where
\begin{equation*}
\widehat{\Delta}^{\mathrm{re}}=\{\alpha +n\delta \mid  \alpha \in
\Delta ,\, n \in {\mathbb{Z}}\},\, \widehat{\Delta}^{\mathrm{im}}=\{n
\delta \mid  n \in {\mathbb{Z}}_{\neq 0}\}.
\end{equation*}

A root $\alpha \in \widehat{\Delta}$ is called {\textit{positive}} if it can
be obtained as a {\textit{nonnegative}} linear combination of simple roots
$\alpha _{i} \in \widehat{\Delta}$, $i=0,1,\ldots ,r$. The subset of
$\widehat{\Delta}$, consisting of positive roots, will be denoted
$\widehat{\Delta}_{+} \subset \widehat{\Delta}$ and can be described as
follows:
\begin{equation*}
\widehat{\Delta}_{+}=\widehat{\Delta}^{\mathrm{re}}_{+} \cup
\widehat{\Delta}^{\mathrm{im}}_{+},
\end{equation*}
where
\begin{equation*}
\widehat{\Delta}^{\mathrm{re}}_{+}=\Delta _{+} \cup \{\alpha +n
\delta \mid  \alpha \in \Delta ,\, n \in {\mathbb{Z}}_{\geqslant 1}\},
\, \widehat{\Delta}^{\mathrm{im}}_{+}=\{n\delta \mid  n \in {
\mathbb{Z}}_{\geqslant 1}\}.
\end{equation*}

To every $\gamma \in \widehat{\Delta}^{\mathrm{re}}$ we associate the corresponding
coroot $\gamma ^{\vee} \in \widehat{\mathfrak{h}}$, defined by
$\gamma ^{\vee}:=\alpha ^{\vee}+nK \in \widehat{\mathfrak{h}}$ if
$\gamma =\alpha +n\delta \in \widehat{\mathfrak{h}}^{*}$.

\subsubsection{Weyl group of $\widehat{\mathfrak{g}}$}
\label{aff_Weyl_grp_g}

Let $\widehat{W}$ be the Weyl group of $\widehat{\mathfrak{g}}$. The group
$\widehat{W}$ is the subgroup of
$\operatorname{Aut}(\widehat{\mathfrak{h}})$ generated by the reflections
$s_{i}$ defined by
%
\begin{equation}
\label{act_on_cartan}
s_{i}(x):=x-\langle \alpha _{i},x\rangle \alpha _{i}^{\vee},\, x \in
\widehat{\mathfrak{h}},\, i=0,1,\ldots ,r.
\end{equation}
For $\gamma \in {Q}^{\vee}$ define the operator
$t_{\gamma} \in \operatorname{Aut}(\widehat{\mathfrak{h}})$ by
\begin{equation*}
t_{\gamma}(x)=x+\langle \delta ,x \rangle \gamma - ((x,\gamma )+
\frac{1}{2}|{\gamma}|^{2}\langle \delta ,x\rangle )K,
\end{equation*}
where $|\gamma |^{2}=(\gamma ,\gamma )$.

The group $\widehat{W}$ is generated by $s_{i}$, $i=1,\ldots ,r$, and
$t_{\gamma}$, $\gamma \in Q^{\vee}$, so that we obtain the identification
\begin{equation*}
\widehat{W} = Q^{\vee} \rtimes W.
\end{equation*}

\begin{Rem}
Note that $\widehat{\mathfrak{h}}$ contains ${\mathbb{C}}K$ as the trivial
subrepresentation of $\widehat{W}$. Note also that
${\mathbb{C}}K \oplus \mathfrak{h} \subset \widehat{\mathfrak{h}}$ is a
subrepresentation: for $x \in {\mathbb{C}}K \oplus \mathfrak{h}$ the action
of $t_{\gamma}$ is given by
\begin{equation*}
t_{\gamma}(x)=x-(x,\gamma )K,\, \gamma \in Q^{\vee}.
\end{equation*}
\end{Rem}

The group $\widehat{W}$ is a Coxeter group generated by reflections
$s_{0},s_{1},\ldots ,s_{r}$ and so is equipped with the {\textit{length}} function
\begin{equation*}
\ell \colon \widehat{W} \rightarrow {\mathbb{Z}}_{\geqslant 0},
\end{equation*}
where $\ell (w)$ is the length of a shortest expression of $w$ in terms
of the $s_{i}$.

Recall that
%
\begin{equation}
\label{length_lattice_even}
\varepsilon (w):=(-1)^{\ell (w)}=\operatorname{det}_{
\widehat{\mathfrak{h}}}w~\text{and} ~\varepsilon (t_{\gamma})=1~
\text{for all}~\gamma \in Q^{\vee}.
\end{equation}

\subsubsection{Irreducible highest weight representations of $\widehat{\mathfrak{g}}$ and their characters}
\label{recall_irred_hw_po_level}

Let
$\mathfrak{b}=\mathfrak{h} \oplus \mathfrak{n}_{+} \subset
\mathfrak{g}$ be the Borel subalgebra, corresponding to our choice of simple
roots $\alpha _{1},\ldots ,\alpha _{r}$. One defines the corresponding
Borel subalgebra of $\widehat{\mathfrak{g}}$:
\begin{equation*}
\widehat{\mathfrak{b}}:=\widehat{\mathfrak{h}} \oplus \mathfrak{n}_{+}
\oplus \bigoplus _{n>0}\mathfrak{g} t^{n}.
\end{equation*}

Given $\Lambda \in \widehat{\mathfrak{h}}^{*}$ one extends it to the character
of $\widehat{\mathfrak{b}}$ by zero on all other summands. Then there exists
a unique irreducible $\widehat{\mathfrak{g}}$-module $L(\Lambda )$ with
highest weight $\Lambda $. Let us recall the construction of
$L(\Lambda )$. Consider the Verma module
\begin{equation*}
M(\Lambda ):=U(\widehat{\mathfrak{g}}) \otimes _{U(
\widehat{\mathfrak{b}})} {\mathbb{C}}_{\Lambda },
\end{equation*}
where ${\mathbb{C}}_{\Lambda }$ is the one dimensional representation of
$\widehat{\mathfrak{b}}$ given by the character $\Lambda $. Then
$L(\Lambda )$ is the unique (nonzero) irreducible quotient of the module
$M(\Lambda )$. Let $\kappa =\Lambda (K)$ be the scalar by which
$K \in \widehat{\mathfrak{g}}$ acts on $L(\Lambda )$ (and
$M(\Lambda )$). This scalar is called the {\textit{level}} of
$L(\Lambda )$ (and the level of $\Lambda $), and is denoted by
$\kappa (\Lambda )$.

For $\mu \in \widehat{\mathfrak{h}}^{*}$ and a
$\widehat{\mathfrak{g}}$-module $M$ let $M_{\mu} \subset M$ be the (generalized)
weight space of $M$ with weight $\mu $. The characters of
$M=L(\Lambda )$ or $M(\Lambda )$ are defined as the following (formal)
series:
\begin{equation*}
\operatorname{ch}M:=\sum _{\mu \in \widehat{\mathfrak{h}}^{*}}(
\operatorname{dim}M_{\mu}) \cdot e^{\mu},
\end{equation*}
here $e^{\mu}$ are formal exponentials such that
$e^{\mu _{1}} \cdot e^{\mu _{2}}=e^{\mu _{1}+\mu _{2}}$ (note that
$\operatorname{dim}M_{\mu }< \infty $).

\begin{Rem}
Note that one can consider $e^{\mu}$ as a function on
$\widehat{\mathfrak{h}}$ whose value on $h \in
\widehat{\mathfrak{h}}$ is equal to $e^{\mu (h)}$. Then
$\operatorname{ch}M$ can be considered as the series $(\operatorname{ch}
M)(h)=\operatorname{tr}_{M} e^{h}$. This series is convergent in the domain
$\{h \in \widehat{\mathfrak{h}}\mid\break  {\operatorname{Re}}(\alpha _{i}(h))>0,\, i=0,1,\ldots ,r\}$.
\end{Rem}

We set
\begin{equation*}
\widehat{\rho}:=\rho +h^{\vee}\Lambda _{0},
\end{equation*}
where $h^{\vee}$ is the dual Coxeter number ($= \frac{1}{2}$ the eigenvalue
on $\mathfrak{g}$ of the Casimir element). Recall that
$\langle \widehat{\rho}, \alpha ^{\vee}_{i} \rangle =1$ for every
$i=0,1,\ldots ,r$, and $\widehat{\rho}=\sum _{i=0}^{r} \Lambda _{i}$.

Set
\begin{equation*}
\widehat{R}:=e^{\widehat{\rho}} \prod _{\alpha \in \widehat{\Delta}_{+}}
(1-e^{-\alpha })^{\operatorname{mult}(\alpha )}.
\end{equation*}
Then
\begin{equation*}
\widehat{R}\operatorname{ch}M(\Lambda )=e^{\Lambda +\widehat{\rho}}.
\end{equation*}

The main result of this note are explicit formulas for characters of modules
$L(\Lambda )$ for certain $\Lambda $ of integer level
$\kappa (\Lambda )>-h^{\vee}$.

\begin{Rem}
Note that the condition that the level of $\Lambda $ is greater than
$-h^{\vee}$ is equivalent to the fact that the level of
$\Lambda +\widehat{\rho}$ is {\textit{positive}}.
\end{Rem}

\subsection{Motivation and main result}

\subsubsection{Motivation and Kac-Wakimoto conjecture}

We start with the following definition.
%
\begin{Def}%
\label{def_la_reg_int_dom_quasid}
An element $\Lambda \in \widehat{\mathfrak{h}}^{*}$ is called {\textit{regular}}
if it has trivial stabilizer w.r.t.
$\widehat{W} \curvearrowright \widehat{\mathfrak{h}}^{*}$; this condition
is equivalent to $\langle \Lambda , \alpha ^{\vee} \rangle \neq 0$ for
all $\alpha \in \widehat{\Delta}^{\mathrm{re}}_{+}$. Element
$\Lambda $ is called {\textit{singular}} if it is not regular (i.e. has nontrivial
stabilizer in $\widehat{W}$). An element
$\Lambda \in \widehat{\mathfrak{h}}^{*}$ is called {\textit{integral}} if
$\langle \Lambda ,\alpha _{i}^{\vee} \rangle \in {\mathbb{Z}}$,
$i=0,1,\ldots ,r$; $\Lambda $ is called {\textit{dominant}} (resp. {\textit{quasi-dominant}})
if
$\langle \lambda ,\alpha _{i}^{\vee} \rangle \in {\mathbb{Z}}_{
\geqslant 0}$ for $i=0,1,\ldots ,r$ (resp. for $i=1,\ldots ,r$).
\end{Def}

Consider the following (shifted) action of $\widehat{W}$ on
$\widehat{\mathfrak{h}}^{*}$:
\begin{equation*}
w \circ \Lambda :=w(\Lambda +\widehat{\rho})-\widehat{\rho}.
\end{equation*}
For $\lambda \in \widehat{\mathfrak{h}}^{*}$ we denote by
$\widehat{W}_{\lambda }\subset \widehat{W}$ the stabilizer of
$\lambda $ w.r.t. the shifted action.

Recall that by \cite{KaKa} for any
$\Lambda \in \widehat{\mathfrak{h}}^{*}$ of level
$\kappa (\Lambda ) > -h^{\vee}$ one has

\begin{equation}
\label{gen_char_integr_comb}
\widehat{R} \operatorname{ch}L(\Lambda ) = \sum _{w \in \widehat{W}} c(w)
e^{w(\Lambda +\widehat{\rho})}~\text{for some}~c(w) \in {\mathbb{Z}}.
\end{equation}

In addition, if $\Lambda $ is quasi-dominant integral, one has (use
$W$-invariance of $\operatorname{ch}L(\Lambda )$)
\begin{equation*}
c(ut_{\gamma})=\varepsilon (u)c(t_{\gamma}),~u \in W, ~\gamma \in Q^{
\vee}.
\end{equation*}

The motivation for this work is the following theorem proven in
\cite[Section 3]{KaW}.

\begin{Thm}%
\label{char_KW_thm}
Let $\mathfrak{g}$ be a simple Lie algebra of type $D_{n}$ ($n
\geqslant 4$) or $E_{6}$, $E_{7}$, $E_{8}$ and $\Lambda $ be a weight of
$\widehat{\mathfrak{g}}$ of negative integral level $\kappa $ such that
the following conditions hold:
\begin{enumerate}
\item[$(i)$] $\Lambda $ is quasi-dominant integral,

\item[$(ii)$] there exists a root $\alpha \in \Delta _{+}$, such that
$(\Lambda +\widehat{\rho},\delta -\alpha )=0$, and if
$\beta \in \widehat{\Delta}_{+}$ is orthogonal to
$\Lambda +\widehat{\rho}$, then $\beta =\delta -\alpha $,

\item[$(iii)$] (extra hypothesis) in (\ref{gen_char_integr_comb}) one has:
$c(t_{\gamma})$ is a linear function in $\gamma \in Q^{\vee}$ plus constant.
\end{enumerate}
Then
%
\begin{equation}
\label{character_spec_La}
\widehat{R} \operatorname{ch}L(\Lambda )=\frac{1}{2}\sum _{u \in W}
\varepsilon (u)\Big(\sum _{\gamma \in Q^{\vee}}(\langle \alpha ,
\gamma \rangle +1)e^{ut_{\gamma}(\Lambda +\widehat{\rho})}\Big).
\end{equation}
\end{Thm}

\begin{Example}%
\label{ex_type_D_4_good_and_bad}
Let us give examples of $\Lambda $, satisfying conditions $(i)$,
$(ii)$ of Theorem~\ref{char_KW_thm} for $\mathfrak{g}$ of type
$D_{4}$. To $\Lambda \in \widehat{\mathfrak{h}}^{*}$ we associate the element
$w \in \widehat{W}$ that is the longest element such that
$\lambda + \widehat{\rho} :=w(\Lambda+\widehat{\rho})$ is dominant. Then the following is a {\textit{partial}}
list of $\Lambda $ of level $-1$, satisfying conditions $(i)$ and
$(ii)$ of Theorem~\ref{char_KW_thm} (we label the branching node of the
Dynkin diagram of $D_{4}$ by $2$), and the corresponding
$\alpha \in \Delta _{+}$ and $w \in \widehat{W}$,
$\lambda \in \widehat{\mathfrak{h}}^{*}$:
\begin{enumerate}
\item[1)] $\Lambda =-\Lambda _{0},\, \alpha =\theta =\alpha _{1}+2\alpha _{2}+
\alpha _{3}+\alpha _{4},\, w=s_{0}, \, \lambda =-\Lambda _{0} $,

\item[2)] $\Lambda =-2\Lambda _{0}+\Lambda _{k},\, \alpha =\theta -\alpha _{2},\, w=s_{2}s_{0},\, \lambda =\Lambda _{k}-\Lambda _{2} $,

\item[3)] $\Lambda =-3\Lambda _{0}+\Lambda _{2},\, \alpha =\theta -\alpha _{2}
,\, w=s_{2}s_{0},\, \lambda =\Lambda _{0}-\Lambda _{2} $,

\item[4)] $\Lambda =-3\Lambda _{0}+\Lambda _{k}+\Lambda _{l},\, \alpha =
\theta -\alpha _{2}-\alpha _{p} ,\, w=s_{p}s_{2}s_{0},\, \lambda =-
\Lambda _{p} $,

\item[5)] $\Lambda =-4\Lambda _{0}+2\Lambda _{k}+\Lambda _{l},\, \alpha =
\theta -\alpha _{2}-\alpha _{l} ,\, w=s_{l} s_{p}s_{2}s_{0},\,
\lambda =-\Lambda _{l} $,

\item[6)] $\Lambda =-5\Lambda _{0}+2\Lambda _{k}+\Lambda _{l}+\Lambda _{p},
\, \alpha =\theta -\alpha _{2}-\alpha _{k} ,\, w=s_{k}s_{p}s_{l}s_{2}s_{0},
\, \lambda =-\Lambda _{k} $,

\item[7)] $\Lambda =-4\Lambda _{0}+3\Lambda _{k},\,\alpha =\theta -\alpha _{2}-
\alpha _{p}-\alpha _{l} ,\,w=s_{2}s_{l}s_{p}s_{2}s_{0} ,\,\lambda =
\Lambda _{k}-\Lambda _{2} $,

where $k,l$ are distinct elements of the set $\{1,3,4\}$ and
$p \in \{1,3,4\} \setminus \{k,l\}$.
\end{enumerate}
Examples $1)$--$4)$ are given in \cite[Example 3.5]{KaW}.
\end{Example}

\begin{Example}%
\label{ex_low_level}
Let $\mathfrak{g}$ be of type $D_{n}$ ($n \geqslant 4$), $E_{6}$,
$E_{7}$, $E_{8}$ and let $b$ be the maximal among the coefficients
$a_{i}$ in $\theta =\sum _{i=1}^{r} a_{i}\alpha _{i}$. Then there is no
$\Lambda $, satisfying conditions $(i)$ and $(ii)$ of Theorem~\ref{char_KW_thm} if $\kappa (\Lambda )<-b$, and there is only one,
$\Lambda =-b\Lambda _{0}$, if $\kappa (\Lambda )=-b$. Furthermore,
$\Lambda =-\kappa \Lambda _{0}$ for $\kappa \in {\mathbb{Z}}$,
$1 \leqslant \kappa \leqslant b$ always satisfies conditions $(i)$ and
$(ii)$ of Theorem~\ref{char_KW_thm}, and then
$\alpha =\theta -\sum _{i=1}^{\kappa -1}\alpha _{i}$,
$w=s_{\kappa -1}\ldots s_{0}$, where we label vertices of the Dynkin diagram
of $\widehat{\mathfrak{g}}$ in such a way that the affine vertex has label
$0$ and $i$, $i+1$ are adjacent for $i \in \{0,1,\ldots ,b-2\}$ (so, in
particular, the label of the branching point is $b-1$).
\end{Example}

It is conjectured in \cite[Conjecture 3.2]{KaW} that if $\mathfrak{g}$ is of
type $D_{4}, E_{6}, E_{7}, E_{8}$, then the extra hypothesis $(iii)$ holds
(so we obtain the character formula (\ref{character_spec_La}) for $L(\Lambda
)$ such that $\Lambda $ satisfies conditions $(i), (ii)$ of
Theorem~\ref{char_KW_thm}). The main motivation of this work was to prove
this conjecture in some cases. For example, in the case of
$\mathfrak{g}=D_{4}$ we prove the conjecture in
Examples~\ref{ex_type_D_4_good_and_bad} $1)$--$4)$ (but not for $5)$--$7)$). We
don't know how to prove the conjecture in general.

We will actually compute (under some conditions) the numbers
$c(t_{\gamma})$ explicitly that will allow us to compute characters of
$L(\Lambda )$ for certain $\Lambda $ that appear in Theorem~\ref{char_KW_thm} and also of certain other $\Lambda $ (see Theorem~\ref{main_th_formulation} and Propositions~\ref{sing_our_char},~\ref{reg_our_char} below).

\subsubsection{Main result: types $D$ and $E$}

Let $\mathfrak{g}$ be of type $D_{n}$ ($n \geqslant 4$), $E_{6}$,
$E_{7}$, $E_{8}$. Before formulating the main result for types $D$ and
$E$ we need to introduce some notation.

Let $I$ be the set of vertices of the Dynkin diagram of
$\mathfrak{g}$. We fix a labeling of $I$ by the numbers
$1,\ldots ,r$. Let $\widehat{I}=I \cup \{0\}$ be the set of vertices of
the Dynkin diagram of $\widehat{\mathfrak{g}}$. For
$i \in \widehat{I}$ consider the unique segment in $\widehat{I}$, connecting
$i$ and $0$. Let $l$ be the length of this segment (i.e. this segment consists
of $l+1$ vertices). Let
\begin{equation*}
0=j_{0},\,j_{1},\ldots ,\,j_{l-1},\,j_{l}=i
\end{equation*}
be the set of vertices that form the segment above. We set
\begin{equation*}
w_{i}:=s_{i}s_{j_{l-1}}\ldots s_{j_{1}}s_{0}.
\end{equation*}

\begin{Rem}
Note that for $i=0$ we have $w_{0}=s_{0}$.
\end{Rem}

We are now ready to describe the main result. The following theorem holds
(see Section~\ref{comp_KL_values_DE} for the proof).

\begin{Thm}%
\label{main_th_formulation}
Let $\mathfrak{g}$ be of type $D_{n}$ ($n \geqslant 4$) or $E_{6}$,
$E_{7}$, $E_{8}$. Pick $i \in \{0,1,\ldots ,r\}$ and let
$\lambda \in \widehat{\mathfrak{h}}^{*}$ be an integral weight, such that
$\lambda +\widehat{\rho}$ is dominant,
$\Lambda =w_{i}^{-1} \circ \lambda $ is quasi-dominant and $w_{i}$ is the
longest element in the coset $\widehat{W}_{\lambda }w_{i}$. Then
%
\begin{equation}
\label{main_char_form_eq}
\widehat{R} \operatorname{ch}L(\Lambda )= \sum _{u \in W}\varepsilon (uw_{i})
\sum _{\gamma \in Q^{\vee}} \biggl\langle \Lambda _{i}, \gamma + \frac{|\gamma |^{2}}{2} K\biggr\rangle e^{ut_{\gamma }w_{i}(\Lambda +
\widehat{\rho})}.
\end{equation}
\end{Thm}

\begin{Rem}
Note that we are not assuming in Theorem~\ref{main_th_formulation} that
the level of $\Lambda $ is negative. Actually, Lemma~\ref{class_our_La} below implies that
$\kappa (\Lambda ) \geqslant 0$ in case $(a)$ and
$\kappa (\Lambda ) \geqslant -b$ in case $(b)$ of this lemma (here
$b$ is as in Example~\ref{ex_low_level} above).
\end{Rem}

The following lemma describes explicitly all the pairs $\lambda $,
$i$, satisfying the conditions of Theorem~\ref{main_th_formulation}.
%
\begin{Lem}%
\label{class_our_La}
Let $\mathfrak{g}$ be of type $D_{n}$ ($n \geqslant 4$) or $E_{6}$,
$E_{7}$, $E_{8}$. Elements $\lambda \in \widehat{\mathfrak{h}}^{*}$,
$i \in \{0,1,\ldots ,r\}$ as in Theorem~\ref{main_th_formulation} are described
as follows. There are two possibilities:
\begin{enumerate}
\item[$(a)$] $\lambda +\widehat{\rho}$ is regular dominant integral and $i$ is
an arbitrary element of $\widehat{I}$,

\item[$(b)$]
$\lambda =-\Lambda _{i}+\sum _{k \neq i}m_{k}\Lambda _{k}+x\delta $ for
some $i \in \{0,1,\ldots ,r\}$, $x \in {\mathbb{C}}$, and
$m_{k} \in {\mathbb{Z}}_{\geqslant 0}$.
\end{enumerate}
\end{Lem}
\begin{proof}
Let us first of all show that if $\lambda $, $i$ are as in Theorem~\ref{main_th_formulation}, then either $\lambda +\widehat{\rho}$ is regular
or
$\lambda =-\Lambda _{i}+\sum _{k \neq i}m_{k}\Lambda _{k}+x\delta $.

Indeed, assume that $\lambda +\widehat{\rho}$ is not regular. Recall that
$\lambda +\widehat{\rho}$ is dominant, so its stabilizer in
$\widehat{W}$ (with respect to the standard non shifted action of
$\widehat{W}$) is generated by some simple reflections
$s_{i_{1}},\ldots ,s_{i_{t}}$. Our goal is to show that $t=1$ and
$i_{1}=i$. Otherwise there exists $d \in \{1,\ldots ,t\}$ such that
$i_{d} \neq i$. Consider the element
\begin{equation*}
w':=s_{i_{d}}w_{i}=s_{i_{d}}s_{i}s_{j_{l-1}}\ldots s_{j_{1}}s_{0}
\end{equation*}
and note that ${w'}^{-1} \circ \lambda =w_{i}^{-1} \circ \lambda $ is quasi-dominant
(by our assumptions). We claim that $\ell (w')=\ell (w_{i})+1$ or equivalently
$\ell ({w'}^{-1})=\ell (w_{i}^{-1})+1$. To see this we need to show that
$w_{i}^{-1}(\alpha _{i_{d}})$ is a positive root. If $i_{d}$ does not lie
in the set $\{0,j_{1},\ldots ,j_{l-1}\}$ and $i_{d}$ is not adjacent to
$i$, then $w_{i}^{-1}(\alpha _{i_{d}})=\alpha _{i_{d}}$ is clearly positive.
If $\alpha _{i_{d}}$ is adjacent to $i=j_{l}$ and is not equal to
$j_{l-1}$, then
$w_{i}^{-1}(\alpha _{i_{d}})=\alpha _{i_{d}}+\alpha _{i}$ is positive.
Finally, if $i_{d}=j_{p}$ for some $p \in \{0,1,\ldots ,l-1\}$, then we
have $w_{i}^{-1}(\alpha _{i_{d}})=\alpha _{j_{p+1}}$ is positive.

So we have shown that $\ell (w')=\ell (w_{i})+1$,
$w_{i}^{-1} \circ \lambda $ is quasi-dominant and
$w'=s_{i_{d}}w_{i} \in \widehat{W}_{\lambda }w_{i}$. This contradicts to
our assumptions (we assumed that $w_{i}$ is the longest element in
$\widehat{W}_{\lambda }w_{i}$). So we conclude that
$\widehat{W}_{\lambda }=\{1,s_{i}\}$ and $\lambda $ must be of the form
$-\Lambda _{i}+\sum _{k \neq i}m_{k} \Lambda _{k}$ with $m_{k}$ being nonnegative.

It remains to show that if $\lambda $, $i$ satisfy assumption $(a)$ or
$(b)$ of Lemma~\ref{class_our_La}, then they satisfy the assumptions of
Theorem~\ref{main_th_formulation}. The only nontrivial part is to check
that
$w_{i}^{-1} \circ \lambda =w^{-1}_{i}(\lambda +\widehat{\rho})-
\widehat{\rho}$ is quasi-dominant.

Let us decompose $\lambda =\sum _{k}m_{k} \Lambda _{k}+x\delta $. We have
$\lambda +\widehat{\rho}=\sum _{k}(m_{k}+1) \Lambda _{k}+x\delta $. We
need to apply $w_{i}^{-1}=s_{0}s_{j_{1}}\ldots s_{j_{l-1}}s_{i}$ to
$\lambda +\widehat{\rho}$ and show that after substituting
$\widehat{\rho}$ we get quasi-dominant element. Indeed, recall that if
$j \in \widehat{I}$, then the action of $s_{j}$ on some
$\Lambda _{k}$ is equal to $\Lambda _{k}$ if $k \neq j$, the action of
$s_{j}$ on $\Lambda _{j}$ is equal to $-\Lambda _{j}$ plus the sum of
$\Lambda _{j'}$, where $j' \in \widehat{I}$ runs through all vertices that
are adjacent to $j$. We easily conclude that the coefficient of
$w_{i}^{-1}(\lambda +\widehat{\rho})$ in front of some
$\Lambda _{j_{p}}$ ($p \in \{1,\ldots ,l\}$) is equal to
$m_{j_{p-1}}+1>0$, the coefficient in front of $\Lambda _{0}$ is equal
to $-(m_{0}+m_{j_{1}}+\cdots +m_{j_{l-1}}+m_{i}+l+1)$ and coefficients
in front of other $\Lambda _{k}$ are at least $m_{k}+1$, so are positive.
We have shown that $w_{i}^{-1} \circ \lambda $ is quasi-dominant.
\end{proof}

\begin{Rem}
The last paragraph of the proof of Lemma~\ref{class_our_La} can be omitted
since it follows from Lemma~\ref{class_irr_R_gen} below.
\end{Rem}

The following proposition shows how to deduce formulas for characters of
(certain) modules appearing in \cite[Conjecture 3.2]{KaW} from Theorem~\ref{main_th_formulation} (note that we get the same character formula
as in \cite[Theorem 3.1]{KaW}).

\begin{Prop}%
\label{sing_our_char}
Let $\mathfrak{g}$ be of type $D_{n}$ ($n \geqslant 4$) or $E_{6}$,
$E_{7}$, $E_{8}$. Assume that
$\lambda =-\Lambda _{i}+\sum _{k \neq i}m_{k}\Lambda _{k}+x\delta $ for
some $i \in \{0,1,\ldots ,r\}$, $x \in {\mathbb{C}}$ and
$m_{k} \in {\mathbb{Z}}_{\geqslant 0}$ for $k \neq i$. Let
$\Lambda =w_{i}^{-1} \circ \lambda $ and
$\alpha =w_{i}^{-1}(\alpha _{i})+\delta$. Then we have
%
\begin{equation}
\label{our_form_KW}
\widehat{R}\operatorname{ch}L(\Lambda )=\frac{1}{2} \sum _{u \in W}
\varepsilon (u) \bigg(\sum _{\gamma \in Q^{\vee}}(\langle \alpha ,
\gamma \rangle +1)e^{ut_{\gamma}(\Lambda +\widehat{\rho})}\bigg).
\end{equation}
\end{Prop}
\begin{proof}
It follows from Lemma~\ref{class_our_La} that $\lambda $, $i$ satisfy
conditions of Theorem~\ref{main_th_formulation}. So the character of
$L(\Lambda )$ is computed using the formula (\ref{main_char_form_eq}). Pick
$u \in W$, $\gamma \in Q^{\vee}$ and let us compute the coefficient of
(\ref{main_char_form_eq}) in front of $e^{ut_{\gamma}(\lambda
+\widehat{\rho})}$. Assume first that $i=0$. It follows that $\lambda
=\Lambda $, $i=0$ and $w_{i}=s_{0}$. Since $s_{0}=s_{\theta }t_{-\theta}$ and
$s_{0}(\Lambda +\widehat{\rho})=\Lambda +\widehat{\rho}$, we conclude that
$us_{\theta}t_{s_{\theta}(\gamma )-\theta}(\Lambda +\widehat{\rho})=ut_{
\gamma}(\Lambda +\widehat{\rho})$, so the coefficient in front of
$e^{ut_{\gamma}(\Lambda +\widehat{\rho})}$ in (\ref{main_char_form_eq}) is
equal to
\begin{align*}
&-\varepsilon (u)\biggl\langle \Lambda _{0}, \gamma +\frac{|\gamma |^{2}}{2} K
\biggr\rangle +\varepsilon (u)\biggl\langle \Lambda _{0},s_{\theta}(\gamma )-
\theta + \frac{|s_{\theta}(\gamma )-\theta |^{2}}{2} K\biggr\rangle
\\
&\quad = \varepsilon (u)\bigg(\frac{|s_{\theta}(\gamma )-\theta |^{2}}{2}-
\frac{|\gamma |^{2}}{2}\bigg)=\varepsilon (u)(\langle \theta ,\gamma
\rangle +1).
\end{align*}

Assume now that $i \neq 0$. Note that
$\lambda +\widehat{\rho}=s_{i}(\lambda +\widehat{\rho})$, hence,
$us_{i}t_{s_{i}(\gamma )}(\lambda +\widehat{\rho})=ut_{\gamma}(
\lambda +\widehat{\rho})$, so the coefficient in front of
$e^{ut_{\gamma}(\lambda +\widehat{\rho})}$ is equal to
\begin{align*}
&\varepsilon (uw_{i})\biggl\langle \Lambda _{i}, \gamma + \frac{|\gamma |^{2}}{2} K \biggr\rangle - \varepsilon (uw_{i})
\biggl\langle \Lambda _{i}, s_{i}(\gamma )+\frac{|\gamma |^{2}}{2} K \biggr\rangle
\\
&\quad =\varepsilon (uw_{i})\langle \Lambda _{i},\alpha _{i}^{\vee}\rangle
\langle \alpha _{i},\gamma \rangle =\varepsilon (uw_{i})\langle w_{i}(
\alpha -\delta ),\gamma \rangle =\varepsilon (uw_{i})\langle \alpha ,w_{i}^{-1}
\gamma \rangle .
\end{align*}

Consider the element
$w_{i}':=w_{i}t_{\theta}=s_{i}s_{j_{l-1}}\ldots s_{j_{1}}s_{\theta}$. We
have shown that the coefficient in front of
\begin{equation*}
e^{ut_{\gamma}(\lambda +\widehat{\rho})}=e^{ut_{\gamma}w_{i}(\Lambda +
\widehat{\rho})}=e^{uw_{i}' t_{{w_{i}'}^{-1}(\gamma )-\theta}(
\Lambda +\widehat{\rho})}
\end{equation*}
is equal to
\begin{align*}
\varepsilon (uw_{i})\langle \alpha ,w_{i}^{-1}\gamma \rangle
&= \varepsilon (uw_{i}')\langle \alpha ,t_{\theta}({w_{i}'}^{-1}\gamma ) \rangle
\\
&= \varepsilon (uw_{i}')\langle \alpha , {w_{i}'}^{-1}\gamma -(\theta ,{w_{i}'}^{-1}
\gamma )K\rangle =\varepsilon (uw_{i}')\langle \alpha ,{w_{i}'}^{-1}
\gamma \rangle
\\
&= \varepsilon (uw_{i}')(\langle \alpha ,{w_{i}'}^{-1}\gamma -\theta \rangle +\langle \alpha ,\theta \rangle ),
\end{align*}
so it remains to check that $(\alpha ,\theta )=1$. Indeed, recall that
$\alpha =w_{i}^{-1}(\alpha _{i})+\delta $, so
\begin{align*}
(\alpha ,\theta )
&=(w_{i}^{-1}(\alpha _{i}),\theta )=(\alpha _{i},w_{i}( \theta ))=-(\alpha _{i},w_{i}(\delta -\theta ))
\\
&=-(\alpha _{i},w_{i}(\alpha _{0}))=(\alpha _{i},s_{i}s_{j_{l-1}} \ldots s_{j_{1}}\alpha _{0})
\\
&=(\alpha _{i},\alpha _{0}+\alpha _{j_{1}}+ \cdots +\alpha _{j_{l-1}}+\alpha _{i})=1.
\end{align*}\vskip-22pt
\end{proof}

\begin{Rem}
In order to identify formulas (\ref{character_spec_La}) and (\ref{our_form_KW}),
we need to show that if $\alpha =w_{i}^{-1}(\alpha _{i})+\delta $, then
$(\delta -\alpha ,\Lambda +\widehat{\rho})=0$ and
$\alpha \in \Delta _{+}$. Note that
$(\alpha _{i},\lambda +\widehat{\rho})=0$ and
$\lambda +\widehat{\rho}=w_{i}(\Lambda +\widehat{\rho})$. It follows that
$(w_{i}^{-1}(\alpha _{i}),\Lambda +\widehat{\rho})=0$, so indeed
$(\delta -\alpha ,\Lambda +\widehat{\rho})=0$. Note also that
\begin{align*}
w_{i}^{-1}(\alpha _{i})
&=s_{0}s_{j_{1}}\ldots s_{j_{l-1}}s_{i}(\alpha _{i})
\\
&=-\alpha _{i}-\alpha _{j_{l-1}}-\cdots -\alpha _{j_{1}}-\alpha _{0}=-
\delta +\theta -\alpha _{j_{1}}-\cdots -\alpha _{j_{l-1}}-\alpha _{i},
\end{align*}
so
$\alpha =\theta -\alpha _{j_{1}}-\cdots -\alpha _{j_{l-1}}-\alpha _{i}
\in \Delta _{+}$.
\end{Rem}

The following proposition gives formulas for characters of certain
$L(\Lambda )$ for regular $\Lambda +\widehat{\rho}$ of nonnegative level.
%
\begin{Prop}%
\label{reg_our_char}
Let $\mathfrak{g}$ be of type $D_{n}$ ($n \geqslant 4$) or $E_{6}$,
$E_{7}$, $E_{8}$. Pick $i \in \{0,1,\ldots ,r\}$ and let
$\lambda \in \widehat{\mathfrak{h}}^{*}$ be integral such that
$\lambda +\widehat{\rho}$ is regular dominant. Set
$\Lambda =w_{i}^{-1} \circ \lambda $, then
\begin{equation*}
\widehat{R} \operatorname{ch}L(\Lambda )= \sum _{u \in W}\varepsilon (uw_{i})
\sum _{\gamma \in Q^{\vee}} \biggl\langle \Lambda _{i},\gamma +
\frac{|\gamma |^{2}}{2}K \biggr\rangle e^{ut_{\gamma }w_{i}(\Lambda +
\widehat{\rho})}.
\end{equation*}
\end{Prop}
\begin{proof}
Follows from Theorem~\ref{main_th_formulation} together with Lemma~\ref{class_our_La}.
\end{proof}

\subsubsection{Main result: type $A$}

Assume that $\mathfrak{g}=\mathfrak{sl}_{n}$. Recall that $Q^{\vee}$ is
the sublattice of ${\mathbb{Z}}^{\oplus n}$, consisting of
$(a_{1},\ldots ,a_{n})$ such that $\sum _{k=1}^{n} a_{k}=0$. Let
$\epsilon _{1},\ldots ,\epsilon _{n}$ be the standard basis of
${\mathbb{Z}}^{\oplus n}$. For $i \in {\mathbb{Z}}$ let
$[i] \in \{0,1,\ldots ,n-1\}$ be the class of $i$ modulo $n$. We set
\begin{equation*}
w_{i}=
\begin{cases}
s_{[i]}s_{[i-1]}\ldots s_{1}s_{0}& \text{for}~i > 0,
\\
s_{0}&\text{for}~i=0,
\\
s_{[-i]}s_{[-i+1]}\ldots s_{[-1]}s_{0}&\text{for}~i < 0,
\end{cases}
\end{equation*}
and for $i \in {\mathbb{Z}}$, $k=1,\ldots ,n$, and $a \in {\mathbb{Z}}$ we set
\begin{equation*}
z_{i}(a\epsilon _{k}):=
\begin{cases}
|{\mathbb{Z}}_{\leqslant i} \cap [k,k+{(a-1)n}] \cap (k+n{\mathbb{Z}})|&
\text{for}~ a \in {\mathbb{Z}}_{\geqslant 0},
\\
-|{\mathbb{Z}}_{\leqslant i} \cap [k+an,k-n] \cap (k+n{\mathbb{Z}})|&
\text{for}~a \in {\mathbb{Z}}_{\leqslant 0}.
\end{cases}
\end{equation*}

We are now ready to describe the main result for type $A$. The following
theorem holds (see Section~\ref{m_th_A} for the proof).

\begin{Thm}%
\label{main_th_formulation_A}
Let $\mathfrak{g}$ be $\mathfrak{sl}_{n}$ ($n \geqslant 3$). Pick
$i \in {\mathbb{Z}}$ and let
$\lambda \in \widehat{\mathfrak{h}}^{*}$ be an integral weight, such that
$\lambda +\widehat{\rho}$ is dominant,
$\Lambda =w_{i}^{-1} \circ \lambda $ is quasi-dominant and $w_{i}$ is the
longest element in the coset $\widehat{W}_{\lambda }w_{i}$. Then
\begin{equation*}
\widehat{R} \operatorname{ch}L(\Lambda )=-\sum _{u \in W}\varepsilon (uw_{i})
\sum _{\gamma \in Q^{\vee}} \Big( \sum _{k=1}^{n} z_{i}(-\langle
\epsilon _{k},\gamma \rangle \epsilon _{k})\Big) e^{ut_{\gamma }w_{i}(
\Lambda +\widehat{\rho})}.
\end{equation*}
\end{Thm}

Similarly to the $D$, $E$ case (see Lemma~\ref{class_our_La}) the pairs
$\lambda $, $i$, satisfying the conditions of Theorem~\ref{main_th_formulation_A} can be described explicitly.

\begin{Lem}%
\label{class_our_La_A}
Let $\mathfrak{g}$ be $\mathfrak{sl}_{n}$ ($n \geqslant 3$). Elements
$\lambda \in \widehat{\mathfrak{h}}^{*}$, $i \in {\mathbb{Z}}$ as in Theorem~\ref{main_th_formulation_A} are described as follows. There are two possibilities:

$(a)$ $\lambda +\widehat{\rho}$ is regular dominant and $i$ is an arbitrary
element of ${\mathbb{Z}}$,

$(b)$
$\lambda =-\Lambda _{[i]}+\sum _{k \neq [i]} m_{k}\Lambda _{k}+x
\delta $ for some $i \in {\mathbb{Z}}$, $x \in {\mathbb{C}}$, and
$m_{k} \in {\mathbb{Z}}_{\geqslant 0}$.
\end{Lem}
\begin{proof}
Same as the proof of Lemma~\ref{class_our_La}.
\end{proof}

The following lemma will be useful.
%
\begin{Lem}%
\label{diff_spec_z}
For $i \in {\mathbb{Z}}\setminus n{\mathbb{Z}}$,
$a \in {\mathbb{Z}}$ we have
\begin{align*}
z_{i}(-a\epsilon _{[i]})-z_{i}(-a\epsilon _{[i]+1})
&=
\begin{cases}
0&\text{for}~a \geqslant 0,~i \notin [[i]-an,[i]-n]
\\
-1&\text{for}~a \geqslant 0,~ i \in [[i]-an,[i]-n]
\\
0&\text{for}~a \leqslant 0,~ i \notin [[i],[i]+(-a-1)n]
\\
1&\text{for}~a \leqslant 0,~i \in [[i],[i]+(-a-1)n]
\end{cases}
\\
&=
\begin{cases}
1&\text{for}~i \in [[i],[i]+(-a-1)n]
\\
0&\text{for}~i \notin [[i]-an,[i]+(-a-1)n]
\\
-1&\text{for}~i \in [[i]-an,[i]-n].
\end{cases}
\end{align*}
\end{Lem}
\begin{proof}
Straightforward.
\end{proof}

\begin{Prop}%
\label{sing_our_char_A}
Let $\mathfrak{g}=\mathfrak{sl}_{n}$, $n \geqslant 3$. Pick
$i \in {\mathbb{Z}}$ and assume that
$\lambda =-\Lambda _{[i]}+\sum _{k \neq [i]}m_{k}\Lambda _{k}+x
\delta $ for some $m_{k} \in {\mathbb{Z}}_{\geqslant 0}$,
$x \in {\mathbb{C}}$. Set $\Lambda =w_{i}^{-1} \circ \lambda $. Then we
have
%
\begin{gather}
\label{form_+_type_A_KW}
\widehat{R}\operatorname{ch}L(\Lambda )= \sum _{u \in W}\varepsilon (u)
\sum _{\gamma \in Q^{\vee},\, \langle \overline{\Lambda }_{n-1},
\gamma \rangle \geqslant 0}e^{u t_{\gamma}(\Lambda )}~\text{for}~i
\geqslant 0,%
\\
\label{form_-_type_A_KW}
\widehat{R}\operatorname{ch}L(\Lambda ) = \sum _{u \in W}\varepsilon (u)
\sum _{\gamma \in Q^{\vee},\, \langle \overline{\Lambda }_{1},\gamma
\rangle \geqslant 0}e^{u t_{\gamma}(\Lambda )}~\text{for}~i \leqslant 0,%
\end{gather}
where
$\overline{\Lambda }_{1},\, \overline{\Lambda }_{n-1} \in
\mathfrak{h}^{*}$ are the corresponding fundamental weights of
$\mathfrak{sl}_{n}$.
\end{Prop}

\begin{proof}
We have
%
\begin{equation}
\label{our_char_form_A}
\widehat{R}\operatorname{ch}L(\Lambda )=\sum _{u \in W}-\varepsilon (uw_{i})
\sum _{\gamma \in Q^{\vee}} \Big(\sum _{k=1}^{n} z_{i}(-\langle
\epsilon _{k},\gamma \rangle \epsilon _{k})\Big) e^{ut_{\gamma }w_{i}(
\Lambda +\widehat{\rho})}.
\end{equation}
Recall that $\lambda +\widehat{\rho}=w_{i}(\Lambda +\widehat{\rho})$ and
$s_{[i]}(\lambda +\widehat{\rho})=\lambda +\widehat{\rho}$. Assume first
that $i \notin n{\mathbb{Z}}$. Decompose
$\gamma =\sum _{k=1}^{n} a_{k}\epsilon _{k}$. We have
$e^{ut_{\gamma }(\lambda +\widehat{\rho})}=e^{us_{[i]}t_{s_{[i]}(
\gamma )}(\lambda +\widehat{\rho})}$, so the coefficient in front of
$e^{ut_{\gamma}(\lambda +\widehat{\rho})}$ is equal to the sum
\begin{align*}
&{-}\varepsilon (uw_{i})\Big(\sum _{k=1}^{n}z_{i}(-\langle \epsilon _{k},
\gamma \rangle \epsilon _{k})\Big)+\varepsilon (uw_{i})\Big(\sum _{k=1}^{n}z_{i}(-
\langle \epsilon _{k},s_{[i]}(\gamma ) \rangle \epsilon _{k})\Big)
\\
&\quad =-\varepsilon (uw_{i})\Big(z_{i}(-a_{[i]}\epsilon _{[i]})+z_{i}(-a_{[i]+1}
\epsilon _{[i]+1})-z_{i}(-a_{[i]+1}\epsilon _{[i]})\!-\!z_{i}(-a_{[i]}\epsilon _{[i]+1})\Big)
\\
&\quad =\varepsilon (uw_{i})\Big((z_{i}(-a_{[i]+1}\epsilon _{[i]})\!-\!z_{i}(-a_{[i]+1}
\epsilon _{[i]+1}))\!-\!(z_{i}(-a_{[i]}\epsilon _{[i]})\!-\!z_{i}(-a_{[i]}
\epsilon _{[i]+1}))\Big).
\end{align*}
Decompose $i=kn+[i]$, $k \in {\mathbb{Z}}_{\geqslant 0}$. Using Lemma~\ref{diff_spec_z}, we conclude that
%
\begin{align}
\label{diff_1}
z_{i}(-a_{[i]}\epsilon _{[i]})-z_{i}(-a_{[i]}\epsilon _{[i]+1})
&=
\begin{cases}
1&\text{for}~i \in [[i],[i]+(-a_{[i]}-1)n]
\\
0&\text{for}~ i \notin [[i],[i]+(-a_{[i]}-1)n]
\end{cases}
\\
&=
\begin{cases}
1&\text{for}~k \in [0,-a_{[i]}-1]
\\
0&\text{for}~ k \notin [0,-a_{[i]}-1]
\end{cases}
\nonumber \\
&=
\begin{cases}
1&\text{for}~{-}a_{[i]}-k-1 \geqslant 0
\\
0&\text{for}~ {-}a_{[i]}-k-1 < 0,
\end{cases}
\nonumber \\
\label{diff_2}
z_{i}(-a_{[i]+1}\epsilon _{[i]})-z_{i}(-a_{[i]+1}\epsilon _{[i]+1})
&=
\begin{cases}
1&\text{for}~i \in [[i],[i]+(-a_{[i]+1}-1)n]
\\
0&\text{for}~ i \notin [[i],[i]+(-a_{[i]+1}-1)n]
\end{cases}
\\
&=
\begin{cases}
1&\text{for}~ {-}a_{[i]+1}-k-1 \geqslant 0
\\
0&\text{for}~ {-}a_{[i]+1}-k-1 < 0.
\end{cases}
\nonumber
\end{align}

It follows that the coefficient in front of
$e^{ut_{\gamma}(\lambda +\widehat{\rho})}$ is equal to
$\varepsilon (uw_{i})$ times the difference of (\ref{diff_2}) and (\ref{diff_1}).

Let us now compute the coefficient in front of
$e^{ut_{\gamma }(\lambda +\widehat{\rho})}$ in (\ref{form_+_type_A_KW}).
We first should write $ut_{\gamma }w_{i} = u' t_{\gamma '}$,
$u' \in S_{n}$, $\gamma ' \in Q^{\vee}$. Let us compute
$u', \gamma '$. Recall that
$w_{i}=s_{[i]}s_{[i-1]}\ldots s_{1}s_{0}$ and let $w_{i}'$ be the element,
obtained from $w_{i}$ by replacing $s_{0}$ with $s_{\theta}$ in the decomposition
above. It is clear that $u'=uw_{i}'$. Let us now compute $\gamma '$. Recall
that $i=kn+[i]$, $k \in {\mathbb{Z}}_{\geqslant 0}$ and decompose
$k+1=l(n-1)+p$, $l \in {\mathbb{Z}}_{\geqslant 0}$,
$1 \leqslant p \leqslant n-1$.

We have
\begin{equation*}
\gamma '={w_{i}'}^{-1}(\gamma )+p\epsilon _{n}-\sum _{j=1}^{p}
\epsilon _{j}+l\Big((n-1)\epsilon _{n}-\sum _{j=1}^{n-1}\epsilon _{j}
\Big),
\end{equation*}
so
\begin{equation*}
\langle \overline{\Lambda }_{n-1},\gamma ' \rangle = \langle
\overline{\Lambda }_{n-1},{w_{i}'}^{-1}(\gamma ) \rangle - k-1=
\langle w_{i}'(\overline{\Lambda }_{n-1}),\gamma \rangle -k-1=-a_{[i]+1}-k-1.
\end{equation*}

Now $e^{u't_{\gamma '}(\Lambda +\widehat{\rho})}$ gives
%
\begin{equation}
\label{KW_form_in_1}
\varepsilon (u w_{i}) \cdot
\begin{cases}
1&\text{for}~{-}a_{[i]+1}-k-1 \geqslant 0
\\
0&\text{for}~{-}a_{[i]+1}-k-1 < 0
\end{cases}
\end{equation}
to the coefficient in front of
$e^{u't_{\gamma '}(\Lambda +\widehat{\rho})}$ in (\ref{form_+_type_A_KW}).
Another part comes from
$e^{us_{[i]}t_{s_{[i]}(\gamma )}w_{i}(\Lambda +\widehat{\rho})}$, it is
easy to see that the corresponding coefficient is equal to
%
\begin{equation}
\label{KW_form_in_2}
-\varepsilon (u w_{i}) \cdot
\begin{cases}
1&\text{for}~{-}a_{[i]}-k-1 \geqslant 0
\\
0&\text{for}~{-}a_{[i]}-k-1 < 0,
\end{cases}
\end{equation}
so the total coefficient of (\ref{form_+_type_A_KW}) in front of
$e^{u' t_{\gamma '}(\Lambda +\widehat{\rho})}$ is equal to the sum of (\ref{KW_form_in_1})
and (\ref{KW_form_in_2}). This sum is clearly equal to the difference of
(\ref{diff_2}) and (\ref{diff_1}) times $\varepsilon (uw_{i})$.

Assume now that $i \in n{\mathbb{Z}}$. Let us identify (\ref{our_char_form_A})
and (\ref{form_+_type_A_KW}) in this case. We have
$e^{ut_{\gamma }(\lambda +\widehat{\rho})}=e^{us_{\theta}t_{s_{\theta}(
\gamma )-\theta}(\lambda +\widehat{\rho})}$, so the coefficient in front
of $e^{ut_{\gamma }(\lambda +\widehat{\rho})}$ is equal to
\begin{align*}
&-\varepsilon (uw_{i})\Big(\sum _{k=1}^{n} z_{i}(-\langle \epsilon _{k},
\gamma \rangle \epsilon _{k})\Big)+\varepsilon (uw_{i})\Big(\sum _{k=1}^{n}
z_{i}(-\langle \epsilon _{k},s_{\theta}(\gamma )-\theta \rangle
\epsilon _{k})\Big)
\\
&\quad = -\varepsilon (uw_{i})\Big( z_{i}(-a_{1}\epsilon _{1})+z_{i}(-a_{n}
\epsilon _{n})-z_{i}(-(a_{n}-1)\epsilon _{1})-z_{i}(-(a_{1}+1)
\epsilon _{n}) \Big)
\\
&\quad =\varepsilon (uw_{i})\Big((z_{i}(-(a_{n}-1)\epsilon _{1})-z_{i}(-a_{n}
\epsilon _{n}))-(z_{i}(-a_{1}\epsilon _{1})\!-\!z_{i}(-(a_{1}+1)\epsilon _{n}))\Big).
\end{align*}

Decompose $i=kn$, $k \in {\mathbb{Z}}_{\geqslant 0}$. We have
\begin{gather*}
z_{i}(-(a_{n}-1)\epsilon _{1})-z_{i}(-a_{n}\epsilon _{n}) =
\begin{cases}
1,&\text{if}~a_{n}+k-1 \geqslant 0
\\
0,&\text{if}~a_{n}+k-1<0,
\end{cases}
\\
z_{i}(-a_{1}\epsilon _{1})-z_{i}(-(a_{1}+1)\epsilon _{n})=
\begin{cases}
1,&\text{if}~k+a_{1} \geqslant 0
\\
0,&\text{if}~k+a_{1}<0.
\end{cases}
\end{gather*}

Let us now compute the (total) coefficient in front of
$e^{ut_{\gamma }(\lambda +\widehat{\rho})}$ in (\ref{form_+_type_A_KW}).
Decompose $k+1=l(n-1)+p$, $l \in {\mathbb{Z}}_{\geqslant 0}$,
$1 \leqslant p \leqslant n-1$. We have
$ut_{\gamma }w_{i} = u't_{\gamma '}$, where $u'=uw_{i}'$,
\begin{equation*}
\gamma '={w_{i}'}^{-1}(\gamma )+p\epsilon _{n}-\sum _{j=1}^{p}
\epsilon _{j} + l\Big((n-1)\epsilon _{n}-\sum _{j=1}^{n-1}\epsilon _{j}
\Big).
\end{equation*}
It follows that
$\langle \overline{\Lambda }_{n-1},\gamma ' \rangle = -a_{1}-k-1$, so
$e^{u't_{\gamma '}(\Lambda +\widehat{\rho})}$ gives
\begin{equation*}
\varepsilon (uw_{i}) \cdot
\begin{cases}
1&\text{for}~{-}a_{1}-k-1 \geqslant 0
\\
0&\text{for}~{-}a_{1}-k-1<0.
\end{cases}
\end{equation*}
Recall also that
$e^{ut_{\gamma }(\lambda +\widehat{\rho})}=e^{us_{\theta }t_{s_{
\theta}(\gamma )-\theta}(\lambda +\widehat{\rho})}$, so the corresponding
coefficient is
\begin{equation*}
-\varepsilon (uw_{i}) \cdot
\begin{cases}
1&\text{for}~{-}a_{n}-k \geqslant 0
\\
0&\text{for}~{-}a_{n}-k < 0.
\end{cases}
\end{equation*}

We see that the (total) coefficient of both (\ref{our_char_form_A}) and
(\ref{form_+_type_A_KW}) in front of
$e^{ut_{\gamma}(\lambda +\widehat{\rho})}$ is equal to
\begin{equation*}
\varepsilon (uw_{i}) \cdot
\begin{cases}
1&\text{for}~a_{n}+k>0>a_{1}+k
\\
0&\text{for}~a_{n}+k \leqslant 0,\, a_{1}+k<0~\text{or}~a_{n}+k>0,\,a_{1}+k
\geqslant 0
\\
-1&\text{for}~a_{1}+k \geqslant 0 \geqslant a_{n}+k.
\end{cases}
\end{equation*}

Formula (\ref{form_-_type_A_KW}) follows from (\ref{form_+_type_A_KW}),
using the involution of the Dynkin diagram of
$\widehat{\mathfrak{sl}}_{n}$, which keeps the $0$th node fixed.
\end{proof}

Restricting attention to the negative level cases in Theorem~\ref{main_th_formulation_A} yields the following corollary, proven in
\cite[Theorem 1.1]{KaW} by different methods.
%
\begin{Cor}
For $\Lambda =-(1+i)\Lambda _{0}+i\Lambda _{n-1}$,
$i \in {\mathbb{Z}}_{\geqslant 0}$, we have
\begin{equation*}
\widehat{R}\operatorname{ch}L(\Lambda )=\sum _{u \in W}\varepsilon (u)
\sum _{\gamma \in Q^{\vee},\, \langle \overline{\Lambda }_{n-1},
\gamma \rangle \geqslant 0}e^{ut_{\gamma}(\Lambda )}.
\end{equation*}

For $\Lambda =-(1+i)\Lambda _{0}+i\Lambda _{1}$,
$i \in {\mathbb{Z}}_{\geqslant 0}$ we have
\begin{equation*}
\widehat{R}\operatorname{ch}L(\Lambda )=\sum _{u \in W}\varepsilon (u)
\sum _{\gamma \in Q^{\vee},\, \langle \overline{\Lambda }_{1},\gamma
\rangle \geqslant 0}e^{ut_{\gamma}(\Lambda )}.
\end{equation*}
\end{Cor}

\begin{Prop}%
\label{reg_our_char_A}
Let $\mathfrak{g}=\mathfrak{sl}_{n}$ ($n \geqslant 3$). Pick
$i \in {\mathbb{Z}}$ and let
$\lambda \in \widehat{\mathfrak{h}}^{*}$ be integral such that
$\lambda +\widehat{\rho}$ is regular dominant. Set
$\Lambda :=w_{i}^{-1} \circ \lambda $, then
\begin{equation*}
\widehat{R}\operatorname{ch}L(\Lambda )=-\sum _{u \in W}\varepsilon (uw_{i})
\sum _{\gamma \in Q^{\vee}}\sum _{k=1}^{n} z_{i}(-\langle \epsilon _{k},
\gamma \rangle \epsilon _{k}) e^{ut_{\gamma }w_{i}(\Lambda +
\widehat{\rho})}.
\end{equation*}
\end{Prop}
\begin{proof}
Follows from Theorem~\ref{main_th_formulation_A} together with Lemma~\ref{class_our_La_A}.
\end{proof}

\subsubsection{Main steps of the proof of Theorem~\ref{main_th_formulation}}

Let us describe the main steps of the proof of Theorem~\ref{main_th_formulation}. Proof of Theorem~\ref{main_th_formulation_A} is similar. We use notations of Theorem~\ref{main_th_formulation}.

The first observation is that
%
\begin{equation}
\label{form_our_La_KL_type}
\widehat{R}\operatorname{ch}L(\Lambda )=\sum _{w \in \widehat{W}}
\varepsilon (ww_{i}^{-1}){\bf{m}}_{w_{i}}^{w}e^{w^{-1}(\lambda +
\widehat{\rho})},
\end{equation}
where ${\bf{m}}_{w_{i}}^{w}={\bf{m}}_{w_{i}}^{w}(1)$ are values at
$q=1$ of inverse Kazhdan-Lusztig polynomials
${\bf{m}}_{w_{i}}^{w}(q)$ for $\widehat{W}$ (see
\cite[Section 0.3]{KT} or Theorem~\ref{char_dom_via_inverse} below).

Using that $\Lambda $ is integral quasi-dominant, we conclude from (\ref{form_our_La_KL_type})
that
\begin{align*}
\widehat{R}\operatorname{ch}L(\Lambda )
&=\sum _{\gamma \in Q^{\vee}}
\varepsilon (w_{\gamma }w_{i}^{-1}){\bf{m}}_{w_{i}}^{w_{\gamma}}\sum _{u
\in W}\varepsilon (u)e^{uw_{\gamma}^{-1}(\lambda +\widehat{\rho})}
\\
&=\varepsilon (w_{i})\sum _{\gamma \in Q^{\vee}} \sum _{u \in W}
\varepsilon (w_{\gamma }u) {\bf{m}}^{w_{\gamma}}_{w_{i}} e^{uw_{
\gamma}^{-1}(\lambda +\widehat{\rho})}
\\
&=\varepsilon (w_{i})\sum _{\gamma \in Q^{\vee}}\sum _{u \in W}
\varepsilon (u){\bf{m}}^{w_{\gamma}}_{w_{i}} e^{ut_{-\gamma}(\lambda +
\widehat{\rho})},
\end{align*}
where $w_{\gamma} \in t_{\gamma}W$ is the shortest element of the coset
$t_{\gamma}W$.

By the definitions together with Proposition~\ref{eq_m_m_prime} below,
the values ${\bf{m}}^{w_{\gamma}}_{w_{i}}(1)$ are computed as follows.
Consider the group algebra ${\mathbb{Z}}\widehat{W}$. This algebra admits
two bases $H_{w}$ and $C_{w}$, $w \in \widehat{W}$, called, respectively,
the standard and the canonical basis (see Appendix~\ref{basic_KL} for the
definitions). Recall the anti-spherical $\widehat{W}$-module
\begin{equation*}
M={\mathbb Z}\widehat{W} \otimes _{{\mathbb Z}W} \mathbb{Z}_{\mathrm{sign}}.
\end{equation*}
Taking images of $H_{w}$, $C_{w}$ under the natural surjection
$\widehat{W} \twoheadrightarrow M$, we obtain standard and canonical bases
in $M$ to be denoted $H_{w}'$, $C'_{w}$ respectively.

Then we have
%
\begin{equation}
H_{w_{\gamma}}'= \sum _{\nu \in Q^{\vee}}\varepsilon (w_{\gamma }w_{
\nu}^{-1}){\bf{m}}_{w_{\nu}}^{w_{\gamma}}C'_{w_{\nu}},
\end{equation}
and our goal is to compute the numbers
${\bf{m}}_{w_{i}}^{w_{\gamma}}$. Set
\begin{equation*}
T_{\gamma}:=\varepsilon (w_{\gamma})H'_{w_{\gamma}},~ C_{\nu}:=
\varepsilon (w_{\nu})C'_{w_{\nu}}.
\end{equation*}
We see that
%
\begin{equation}
\label{decomp_that_we_use_VTEX1}
T_{\gamma }= \sum _{\nu \in Q^{\vee}} {\bf{m}}_{w_{\nu}}^{w_{\gamma}} C_{
\nu}.
\end{equation}

The module $M$ has a ``coherent'' realization as the equivariant $K$-theory
$K^{G^{\vee}}(\widetilde{{\mathcal{N}}})$ of the Springer resolution for
the Langlands dual group $G^{\vee}$ (see \cite{CG}). It follows from
\cite{AB,BHe} that after the identification
$M \simeq K^{G^{\vee}}(\widetilde{{\mathcal{N}}})$ elements
$T_{\gamma}$ correspond to the classes of natural line bundles
${\mathcal{O}}_{\widetilde{{\mathcal{N}}}}(\gamma )$ on
$\widetilde{{\mathcal{N}}}$. Elements of the canonical basis
$C_{\nu}$ correspond to classes of irreducible objects in the ``exotic''
$t$-structure on
$D^{b}(\operatorname{Coh}^{G}(\widetilde{{\mathcal{N}}}))$ defined and
studied in \cite{BHum,BHe}.

Recall that our goal is to compute the numbers
${\bf{m}}_{w_{i}}^{w_{\gamma}}$. It turns out that all these numbers are
already ``contained'' in a certain quotient of the $\widehat{W}$-module
$K^{G^{\vee}}(\widetilde{{\mathcal{N}}})$. Let us describe this quotient
in our situation.

Let $e \in {\mathcal{N}}$ be a subregular nilpotent element of
$\mathfrak{g}^{\vee}$ and recall that
${\mathbb{O}}_{e} \subset {\mathcal{N}}$ is the corresponding nilpotent
orbit. Recall also that
${\mathbb{O}}^{\mathrm{reg}} \subset {\mathcal{N}}$ is the regular nilpotent
orbit, so that $U={\mathbb{O}}^{\mathrm{reg}} \cup {\mathbb{O}}_{e}$ is
an open $G^{\vee}$-invariant subvariety of ${\mathcal{N}}$. Recall that
$\widetilde{U} \subset \widetilde{{\mathcal{N}}}$ is the preimage of
$U$. We have the natural surjection of $\widehat{W}$-modules
$K^{G^{\vee}}(\widetilde{{\mathcal{N}}}) \twoheadrightarrow K^{G^{
\vee}}(\widetilde{U})$. It follows from the results of the first author
(see \cite[\S 11.3]{BHe}) that the kernel of this surjection is spanned over $\mathbb{Z}$
by $C_{\nu}$ for $w_{\nu }\notin \{1,w_{0},w_{1},\ldots ,w_{r}\}$.

Taking the image of the equality (\ref{decomp_that_we_use_VTEX1}) in
$K^{G^{\vee}}(\widetilde{U})$, we conclude that
%
\begin{equation}
\label{eq_in_U}
\bar{T}_{\gamma}=\bar{C}_{0}+\sum _{i=0,1,\ldots ,r}{\bf{m}}^{w_{
\gamma}}_{w_{i}}\bar{C}_{\nu _{i}},
\end{equation}
where by $\bar{x}$ we mean the image of the element
$x \in M \simeq K^{G^{\vee}}(\widetilde{{\mathcal{N}}})$ in
$K^{G^{\vee}}(\widetilde{U})$ and $\nu _{i}$ is the image of $w_{i}$ in
$Q^{\vee }\simeq \widehat{W}/W$.

It turns out (see Proposition~\ref{ident_K_U_DE}) that the
$\widehat{W}$-module $K^{G^{\vee}}(\widetilde{U})$ is isomorphic to
$\widehat{\mathfrak{h}}_{{\mathbb{Z}}} \otimes {\mathbb{Z}}_{
\mathrm{sign}}$ (with the $\widehat{W}$-action on
$\widehat{\mathfrak{h}}_{{\mathbb{Z}}}$ given by (\ref{act_on_cartan})).
The canonical basis elements $\bar{C}_{0}$ and $\bar{C}_{\nu _{i}}$ are
$d \otimes 1$ and $-\alpha ^{\vee}_{i} \otimes 1$. Using the equation (\ref{eq_in_U}),
this allows us to compute the numbers
${\bf{m}}^{w_{\gamma}}_{w_{i}}$ explicitly.

\section{Categories ${\mathcal{O}}$ for $\widehat{\mathfrak{g}}$ and characters of irreducible modules via Kazhdan-Lusztig polymomials}
\label{sect_cat_O_and_KL}

\subsection{Category ${\mathcal{O}}_{\kappa}$ for $\widehat{\mathfrak{g}}$ and its decomposition into blocks}

Consider a module $L(\Lambda )$ of integer level
$\kappa (\Lambda )>-h^{\vee}$,
$\Lambda \in \widehat{\mathfrak{h}}^{*}$. The module $L(\Lambda )$ is an
object of the (affine) category ${\mathcal{O}}$, denoted by
${\mathcal{O}}_{\kappa}$ and defined as follows (see, for example,
\cite[Section 3]{KT0}). Set
$\widehat{Q}_{+}:=\sum _{i \in \hat{I}}{\mathbb{Z}}_{\geqslant 0}
\alpha _{i}$.

\begin{Def}
The category ${\mathcal{O}}_{\kappa}$ is the full subcategory of the category
of $\widehat{\mathfrak{g}}$-modules of level $\kappa $, consisting of
$\widehat{\mathfrak{g}}$-modules $N$ such that
\begin{enumerate}
\item[$(a)$] $N=\bigoplus _{\mu \in \widehat{\mathfrak{h}}^{*}} N_{\mu}$, where
$N_{\mu}$ is a generalized $\mu $-weight space for
$\widehat{\mathfrak{h}}$,

\item[$(b)$] $\operatorname{dim}N_{\mu }< \infty $,

\item[$(c)$] for any $\mu \in \widehat{\mathfrak{h}}^{*}$ there exists only finitely
many $\beta \in \mu +\widehat{Q}^{+}$ such that $N_{\beta }\neq 0$.
\end{enumerate}
\end{Def}

Let us recall the block decomposition of the category
${\mathcal{O}}_{\kappa}$. Irreducible objects of this category are precisely
$L(\Lambda )$ such that $\kappa =\kappa (\Lambda )$. The category
${\mathcal{O}}_{\kappa}$ can be decomposed into blocks
${\mathcal{O}}_{\kappa ,\xi}$ as follows. Let
$\widehat{\mathfrak{h}}^{*}/_{\circ} \widehat{W}$ be the set of equivalence
classes with respect to the shifted $\widehat{W}$-action. For
$\Lambda \in \widehat{\mathfrak{h}}^{*}$ let
$\overline{\Lambda } \in \widehat{\mathfrak{h}}^{*}/_{\circ}
\widehat{W}$ be the class of $\Lambda $.

\begin{Def}
For $\xi \in \widehat{\mathfrak{h}}^{*}/_{\circ} \widehat{W}$ let
${\mathcal{O}}_{\kappa ,\xi} \subset {\mathcal{O}}_{\kappa}$ be the full
subcategory of ${\mathcal{O}}_{\kappa}$, consisting of modules, all of
whose composition factors are $L(\Lambda )$ with
$\overline{\Lambda }=\xi $.
\end{Def}

\begin{Prop}
The category ${\mathcal{O}}_{\kappa}$ decomposes as a direct sum as follows:
\begin{equation*}
{\mathcal{O}}_{\kappa}=\bigoplus _{\xi \in \widehat{\mathfrak{h}}^{*}/_{
\circ} \widehat{W}} {\mathcal{O}}_{\kappa ,\xi}.
\end{equation*}
\end{Prop}
\begin{proof}
Follows from \cite[Theorem 5.7]{DGK}.
\end{proof}

The subcategory ${\mathcal{O}}_{\kappa ,\overline{\Lambda }}$, corresponding
to an integral $\Lambda $, is called an {\textit{integral block}}. Let
${\mathcal{O}}_{\kappa ,\xi}$ be an integral block; then there exists the
unique $\lambda \in \xi $ such that $\lambda +\widehat{\rho}$ is dominant
integral (since $\kappa (\Lambda )>-h^{\vee}$).

Recall that $\widehat{W}_{\lambda }\subset \widehat{W}$ is the stabilizer
of $\lambda $ (w.r.t. the $\circ $ action). Fixing an integral
$\lambda \in \widehat{\mathfrak{h}}^{*}$, such that
$\lambda +\widehat{\rho}$ is dominant, we obtain the identification
\begin{equation*}
\widehat{W}_{\lambda }\backslash \widehat{W} \iso \operatorname{Irr}({
\mathcal{O}}_{\kappa ,\xi}),\, \overline{w} \mapsto L(w^{-1} \circ
\lambda ).
\end{equation*}
For $\overline{w} \in \widehat{W}_{\lambda }\backslash \widehat{W}$ let
$L_{\overline{w}}$ and
$M_{\overline{w}} \in {\mathcal{O}}_{\kappa ,\xi}$ be simple and Verma
modules, corresponding to the coset $\overline{w}$ i.e.:
\begin{equation*}
L_{\overline{w}}:=L(w^{-1} \circ \lambda ),\, M_{\overline{w}} := M(w^{-1}
\circ \lambda ).
\end{equation*}

\begin{Def}
For an integral $\lambda \in \widehat{\mathfrak{h}}^{*}$ such that
$\lambda +\widehat{\rho}$ is dominant let ${}^{\lambda }\widehat{W}$ be
the subset of $\widehat{W}$, consisting of {\textit{maximal}} length representatives
of {\textit{left}} $\widehat{W}_{\lambda }$-cosets
$\widehat{W}_{\lambda }\backslash \widehat{W}$.
\end{Def}

We have the bijection
${}^{\lambda }\widehat{W} \iso \widehat{W}_{\lambda }\backslash
\widehat{W}$ that sends $w$ to ${\widehat{W}}_{\lambda }w$. For
$w \in {}^{\lambda }\widehat{W}$ let
$L_{w},\,M_{w} \in {\mathcal{O}}_{\kappa ,\xi}$ be simple and Verma modules,
corresponding to the coset ${\widehat{W}}_{\lambda }w$:
\begin{equation*}
L_{w}:=L_{\overline{w}},\, M_{w} := M_{\overline{w}}.
\end{equation*}

\subsection{Classes of irreducible objects of a regular integral block ${\mathcal{O}}_{\kappa ,\xi}$}

Let ${\mathcal{O}}_{\kappa ,\xi}$ be a regular integral block. Consider
the Grothendieck $K$-group $K_{0}({\mathcal{O}}_{\kappa ,\xi})$ of the
category ${\mathcal{O}}_{\kappa ,\xi}$. For an object
$N \in {\mathcal{O}}_{\kappa ,\xi}$ we denote by
$[N] \in K_{0}({\mathcal{O}}_{\kappa ,\xi})$ the corresponding class. For
$v\in \widehat{W}$ and $w\in \widehat{W}$ let ${\bf{m} }_{v}^{w}$ be determined
by:
%
\begin{equation}
\label{mult_pos}
[L_{v}]=\sum _{w \in \widehat{W}} \varepsilon (wv^{-1}){\bf{m} }_{v}^{w}
[M_{w}].
\end{equation}

\begin{Rem}
The numbers ${\bf{m} }_{v}^{w}$ are given by the affine inverse Kazhdan-Lusztig
polynomials ${\bf{m}}^{w}_{v}(q)$ evaluated at $q=1$ (see Theorem~\ref{char_dom_via_inverse} or \cite[Section 0.3]{KT}).
\end{Rem}

\subsection{Classes of irreducible objects of integral blocks ${\mathcal{O}}_{\kappa ,\xi}$}

Let ${\mathcal{O}}_{\kappa ,\xi}$ be an integral block of level
$\kappa > -h^{\vee}$ (we are not assuming that $\xi $ is regular).

\begin{Prop}%
\label{char_sing}
Pick $v \in {}^{\lambda }\widehat{W}$. We have
\begin{equation*}
[L_{v}]=\sum _{w \in \widehat{W}} \varepsilon (wv^{-1}){\bf{m}}^{w}_{v}
[M_{\overline{w}}]=\sum _{w \in {}^{\lambda }\widehat{W}}
\varepsilon (wv^{-1})\Big(\sum _{u \in \widehat{W}_{\lambda }}
\varepsilon (u){\bf{m} }^{u w}_{v} \Big)[M_{w}].
\end{equation*}
\end{Prop}
\begin{proof}
The claim can be deduced from \cite[Theorem 1.1]{KT0}. It also follows
from the equality (\ref{mult_pos}) by using translation functors (see
\cite[Theorem 5.13]{DGK} or \cite[Section 3]{KT0}).
\end{proof}

\subsection{Subcategory ${\mathcal{R}}_{\kappa }\subset {\mathcal{O}}_{\kappa}$ and irreducible objects in integral blocks ${\mathcal{R}}_{\kappa ,\xi}$}

We will be interested in integral quasi-dominant $\Lambda $ (see Definition~\ref{def_la_reg_int_dom_quasid}). The last condition corresponds to the
fact that $L(\Lambda )$ lies in the subcategory
${\mathcal{R}}_{\kappa }\subset {\mathcal{O}}_{\kappa}$, defined as follows.

\begin{Def}
The category ${\mathcal{R}}_{\kappa}$ is the full subcategory of
${\mathcal{O}}_{\kappa}$, consisting of modules
$N \in {\mathcal{O}}_{\kappa}$ such that the action of
$\mathfrak{g}[t]$ on $N$ is locally finite. For a block
${\mathcal{O}}_{\kappa ,\xi}$ we denote by
${\mathcal{R}}_{\kappa ,\xi} \subset {\mathcal{O}}_{\kappa ,\xi}$ the full
subcategory of ${\mathcal{O}}_{\kappa ,\xi}$ 
by objects in
${\mathcal{R}}_{\kappa}$.
\end{Def}

Let us now describe irreducible objects of the category
${\mathcal{R}}_{\kappa ,\xi}$. We start from the case when
$\xi \in \widehat{\mathfrak{h}}^{*}/_{\circ} \widehat{W}$ is regular (see
Definition~\ref{def_la_reg_int_dom_quasid}).

\begin{Def}
Let $\widehat{W}^{f}$ be the set of {\textit{minimal}} length representatives
of {\textit{right}} $W$-cosets in $\widehat{W}$. We have
$\widehat{W}^{f}\iso \widehat{W}/W\simeq Q^{\vee}$ and let
$\nu \mapsto w_{\nu}$ be the inverse bijection.
\end{Def}

\begin{Lem}%
\label{R_in_O_for_reg}
For integral regular $\xi $ the category
${\mathcal{R}}_{\kappa ,\xi}$ is the Serre subcategory of
${\mathcal{O}}_{\kappa ,\xi}$ whose irreducible objects are $L_{v}$,
$v \in \widehat{W}^{f}$.
\end{Lem}
\begin{proof}
It is clear that
${\mathcal{R}}_{\kappa ,\xi} \subset {\mathcal{O}}_{\kappa ,\xi}$ is the
Serre subcategory. It remains to show that
$L_{v} \in {\mathcal{R}}_{\kappa ,\xi}$ iff $v \in \widehat{W}^{f}$. Indeed,
let us first of all note that $v \in \widehat{W}^{f}$ iff
$v^{-1}(\lambda +\widehat{\rho})$ is quasi-dominant (indeed, if
$v \in \widehat{W}^{f}$ and $v^{-1}(\lambda +\widehat{\rho})$ is not quasi-dominant
then there exists $i \in \{1,\ldots ,r\}$ such that
$\langle v^{-1}(\lambda +\widehat{\rho}),\alpha _{i}^{\vee} \rangle <0$,
so
$\langle \lambda +\widehat{\rho},v(\alpha _{i}^{\vee})\rangle < 0$ i.e.
$v(\alpha _{i}^{\vee})$ is negative, hence,
$\ell (vs_{i})=\ell (v)-1$ that contradicts to
$v \in \widehat{W}^{f}$, similarly if
$v^{-1}(\lambda +\widehat{\rho})$ is quasi-dominant but
$v \notin \widehat{W}^{f}$ then there exists $i \in \{1,\ldots ,r\}$ such
that $v(\alpha _{i}^{\vee})$ is negative that contradicts to
$\langle v^{-1}(\lambda +\widehat{\rho}),\alpha _{i}^{\vee }\rangle > 0$).

It remains to show that $L(\Lambda ) \in {\mathcal{R}}_{\kappa}$ iff
$\Lambda $ is quasi-dominant. Assume that
$L(\Lambda ) \in {\mathcal{R}}_{\kappa}$ and consider the
$\mathfrak{g}$-submodule of $L(\Lambda )$ generated by the highest weight
vector of $L(\Lambda )$. This is a finite dimensional module with highest
weight~$\Lambda |_{\mathfrak{h}}$. It follows that
$\Lambda |_{\mathfrak{h}}$ is dominant i.e. $\Lambda $ is quasi-dominant.

Assume now that $\Lambda $ is quasi-dominant. Let $V(\Lambda )$ be the
irreducible (finite dimensional) representation of $\mathfrak{g}$ with
highest weight $\Lambda |_{\mathfrak{h}}$. Consider $V(\Lambda )$ as a
module over
$\widehat{\mathfrak{g}}_{+}:=\mathfrak{g}[t] \oplus {\mathbb{C}}K
\oplus {\mathbb{C}}d$, letting $t\mathfrak{g}[t]$ act via zero, $K$ act
via the multiplication by $\Lambda (K)$, and $d$ act via the multiplication
by $\Lambda (d)$. Consider the induced module
$\operatorname{Ind}^{\widehat{\mathfrak{g}}}_{\widehat{\mathfrak{g}}_{+}}V(
\Lambda ):=U(\widehat{\mathfrak{g}}) \otimes _{U(
\widehat{\mathfrak{g}}_{+})} V(\Lambda )$. It is easy to see that
$\operatorname{Ind}^{\widehat{\mathfrak{g}}}_{\widehat{\mathfrak{g}}_{+}}V(
\Lambda ) \in {\mathcal{R}}_{\kappa}$. Since $L(\Lambda )$ is a quotient
of
$\operatorname{Ind}^{\widehat{\mathfrak{g}}}_{\widehat{\mathfrak{g}}_{+}}V(
\Lambda )$, we conclude that
$L(\Lambda ) \in {\mathcal{R}}_{\kappa}$.
\end{proof}

In general (for singular $\xi $) irreducible objects of
${\mathcal{R}}_{\kappa ,\xi} \subset {\mathcal{O}}_{\kappa ,\xi}$ can be
described as follows. Recall that irreducible objects of
${\mathcal{O}}_{\kappa ,\xi}$ are in bijection with~${}^{\lambda }\widehat{W}$.

\begin{Def}
Let ${}^{\lambda }\widehat{W}^{f}$ be the intersection
${}^{\lambda }\widehat{W} \cap \widehat{W}^{f} \subset \widehat{W}$. Using
the identification $\widehat{W}^{f} \iso Q^{\vee}$, we can identify
${}^{\lambda }\widehat{W} \cap \widehat{W}^{f}$ with the subset of
$Q^{\vee}$ to be denoted ${}^{\lambda }Q^{\vee}$.
\end{Def}

\begin{Lem}%
\label{class_irr_R_gen}
For integral $\xi $ the category ${\mathcal{R}}_{\kappa ,\xi}$ is the Serre
subcategory of ${\mathcal{O}}_{\kappa ,\xi}$ whose irreducible objects
are $L_{v}$, $v \in {}^{\lambda }{\widehat{W}}^{f}$, where
$\lambda \in \xi $ is such that $\lambda +\widehat{\rho}$ is dominant integral.
\end{Lem}
\begin{proof}
Clearly ${\mathcal{R}}_{\kappa ,\xi}$ is a Serre subcategory of
${\mathcal{O}}_{\kappa ,\xi}$. It follows from the proof of Lemma~\ref{R_in_O_for_reg} that an irreducible object
$L_{v} \in {\mathcal{O}}_{\kappa ,\xi}$ ($v \in {}^{\lambda }
\widehat{W}$) lies in ${\mathcal{R}}_{\kappa ,\xi}$ iff the element
$v^{-1} \circ \lambda $ is quasi-dominant.

It remains to show that $v \in {}^{\lambda }\widehat{W}^{f}$ iff
$v^{-1} \circ \lambda $ is quasi-dominant. Indeed, assume that
$v \in {}^{\lambda }\widehat{W}$ is such that
$\Lambda :=v^{-1} \circ \lambda $ is quasi-dominant. It follows that
$\Lambda +\widehat{\rho}=v^{-1}(\lambda +\widehat{\rho})$ pairs with all
$\alpha _{i}^{\vee}$, $i=1,2,\ldots ,r$, by {\textit{positive}} numbers. If
$v \notin W^{f}$, then there exists $i \in \{1,2,\ldots ,r\}$ such that
$\ell (vs_{i})=\ell (v)-1$. This is equivalent to
$v(\alpha _{i}) \in \widehat{\Delta}_{-}$. On the other hand we have
%
\begin{equation}
\label{ineq_pairing_al_i}
\langle \lambda +\widehat{\rho},v(\alpha _{i}^{\vee}) \rangle =
\langle v^{-1}(\lambda +\widehat{\rho}),\alpha _{i}^{\vee} \rangle =
\langle \Lambda +\widehat{\rho}, \alpha _{i}^{\vee}\rangle >0.
\end{equation}
Recall now that $\lambda +\widehat{\rho}$ is quasi-dominant, so the pairing
$\langle \lambda +\widehat{\rho},v(\alpha _{i}^{\vee}) \rangle $ must be
nonpositive, contradicting to (\ref{ineq_pairing_al_i}).

Assume now that $v \in {}^{\lambda }\widehat{W}^{f}$. Then
$v \in \widehat{W}^{f}$, and since $\lambda +\widehat{\rho}$ is dominant,
it is clear that
$\Lambda +\widehat{\rho}=v^{-1}( \lambda +\widehat{\rho})$ is quasi-dominant
(same argument as in the proof of Lemma~\ref{R_in_O_for_reg}). It remains
to show that there is no $i \in \{1,\ldots ,r\}$ such that
$\langle \Lambda +\widehat{\rho},\alpha _{i}^{\vee} \rangle = 0$. Assume
that such $i$ exists. Then
$s_{i}(\Lambda +\widehat{\rho})=\Lambda +\widehat{\rho}$, i.e.
$(vs_{i})^{-1} \circ \lambda =v^{-1} \circ \lambda $, hence,
$vs_{i} \in \widehat{W}_{\lambda }v$. Recall that $v$ is the longest element
of $\widehat{W}_{\lambda }v$, so $\ell (vs_{i})=\ell (v)-1$, and that contradicts
the fact that $v$ is the shortest element of $vW$ (since
$vs_{i} \in vW$ and $vs_{i}$ is shorter than $v$).
\end{proof}

\subsection{Classes of irreducibles of regular integral block ${\mathcal{R}}_{\kappa ,\xi}$}

Assume that
$\xi \in \widehat{\mathfrak{h}}^{*}/_{\circ} \widehat{W}$ is regular integral
and consider the corresponding category
${\mathcal{R}}_{\kappa ,\xi}$.

\begin{Prop}%
\label{char_int}
Suppose that $v \in \widehat{W}^{f}$. Then we have
%
\begin{equation}
\label{class_simple_in_R_reg}
[L_{v}]=\sum _{w \in \widehat{W}^{f}}\varepsilon (wv^{-1})\sum _{u
\in W}\varepsilon (u){\bf{m}}^{w}_{v}[M_{wu}].
\end{equation}
\end{Prop}

\begin{proof}
This follows from $W$-invariance of $\operatorname{ch}L_{v}$.
\end{proof}

Recall now that we have the bijection
$\widehat{W}^{f} \iso Q^{\vee}$ and the inverse bijection sends
$\nu $ to $w_{\nu}$. Then the equality (\ref{class_simple_in_R_reg}) can
be rewritten as follows: for $v=w_{\nu}$ we have
\begin{equation*}
[L_{v}]=\sum _{\gamma \in Q^{\vee}}\varepsilon (w_{\gamma }v^{-1})
\sum _{u \in W}\varepsilon (u){\bf{m}}^{w_{\gamma}}_{v}[M_{w_{\gamma }u}].
\end{equation*}

\subsection{Classes of irreducibles of integral block ${\mathcal{R}}_{\kappa ,\xi}$}

Assume now that ${\mathcal{R}}_{\kappa ,\xi}$ is an arbitrary (possibly
singular) integral block.

\begin{Thm}
For $v \in {}^{\lambda }\widehat{W}^{f}$ we have
%
\begin{align}
\label{irreg_mult}
[L_{v}]
&=\sum _{w \in \widehat{W}^{f}} \varepsilon (wv^{-1})\sum _{u
\in W}\varepsilon (u){\bf{m}}^{w}_{v} [M_{\overline{wu}}]
\\
&=\sum _{w \in {}^{\lambda }\widehat{W}^{f}}\varepsilon (wv^{-1})\sum _{u
\in W}\varepsilon (u)\Big(\sum _{\sigma \in \widehat{W}_{\lambda }}
\varepsilon (\sigma ){\bf{m}}_{v}^{\sigma w}\Big)[M_{wu}].
\nonumber
\end{align}
\end{Thm}
\begin{proof}
Follows from Proposition~\ref{char_sing} and the $W$-invariance of
$\operatorname{ch}L_{v}$.
\end{proof}

Equivalently the equality (\ref{irreg_mult}) can be rewritten as follows.
Pick $v \in {}^{\lambda }\widehat{W}^{f}$, then
%
\begin{align}
\label{char_r_sing_mu_nu}
[L_{v}]
&=\sum _{\gamma \in Q^{\vee}}\varepsilon (w_{\gamma }v^{-1}) \sum _{u \in W}\varepsilon (u){\bf{m}}_{v}^{w_{\gamma}}[M_{w_{\gamma }u}]
\\
&=\sum _{\gamma \in {}^{\lambda }Q^{\vee}} \varepsilon (w_{\gamma }v^{-1})
\sum _{u \in W} \varepsilon (u)\Big(\sum _{\sigma \in \widehat{W}_{
\lambda }}\varepsilon (\sigma ) {\bf{m}}_{v}^{\sigma w_{\gamma}}\Big)[M_{w_{\gamma }u}].
\nonumber
\end{align}

\begin{Rem}
Note that if $w \in {}^{\lambda }\widehat{W}^{f}$ and
$\sigma \in \widehat{W}_{\lambda }$, then
$\sigma w \in \widehat{W}^{f}$.
\end{Rem}

\subsection{Characters of $L(\Lambda ) \in {\mathcal{R}}_{\kappa ,\xi}$}

Recall that we are assuming that $\Lambda $ is integral and quasi-dominant.
This corresponds to the fact that
$L(\Lambda ) \in {\mathcal{R}}_{\kappa ,\xi}$, where
$\xi =\widehat{W} \circ \Lambda $. Let $\lambda \in \xi $ be the (unique)
element such that $\lambda +\widehat{\rho}$ is dominant. Let
$v \in {}^{\lambda }\widehat{W}^{f}$ be the element such that
$\Lambda =v \circ \lambda $.

The equality (\ref{char_r_sing_mu_nu}) can be obviously rewritten in the
following way:
\begin{equation*}
[L(\Lambda )]=\sum _{\gamma \in Q^{\vee}}\varepsilon (w_{\gamma }v^{-1})
\sum _{u \in W}\varepsilon (u){\bf{m}}^{w_{\gamma}}_{v}[M((u^{-1}w_{
\gamma}^{-1})\circ \lambda )].
\end{equation*}
Hence the formula for the character of $L(\Lambda )$ is
\begin{equation*}
\label{char_our_L_via_repres_short_long}
\widehat{R}\operatorname{ch}L(\Lambda )=\sum _{\gamma \in Q^{\vee}}
\varepsilon (w_{\gamma }v^{-1})\sum _{u \in W}\varepsilon (u){\bf{m}}^{w_{
\gamma}}_{v}e^{uw_{\gamma}^{-1}(\lambda +\widehat{\rho})}.
\end{equation*}

Let $\nu \in Q^{\vee}$ be such that $v=w_{\nu}$. We conclude that
%
\begin{align}
\label{our_kl_form_convenient}
\widehat{R} \operatorname{ch}L(\Lambda )
&=\sum _{\gamma \in Q^{\vee}}
\sum _{u \in W}\varepsilon (uw_{\gamma }w_{\nu}) {\bf{m}}_{w_{\nu}}^{w_{
\gamma}} e^{u(w_{\gamma}^{-1}t_{\gamma})t_{-\gamma}(\lambda +
\widehat{\rho})}
\\
&= \sum _{\gamma \in Q^{\vee}} \sum _{u \in W} \varepsilon (u w_{\nu}) {
\bf{m}}_{w_{\nu}}^{w_{\gamma}} e^{ut_{-\gamma}(\lambda +
\widehat{\rho})}.
\nonumber
\end{align}
So all the information about the character of $L(\Lambda )$ is contained
in the numbers ${\bf{m}}_{w_{\nu}}^{w_{\gamma}}$ for
$\gamma \in Q^{\vee}$.

\subsection{Description of ${\bf{m}}^{w_{\gamma}}_{w_{\nu}}$ via anti-spherical module $M$}

In this section we recall some results of Appendix~\ref{basic_KL}. Consider
the group algebra ${\mathbb{Z}}\widehat{W}$. Then
${\mathbb{Z}}\widehat{W}$ admits two bases $H_{w}$ and $C_{w}$, indexed
by $\widehat{W}$, and called, respectively, the standard and the canonical
basis (whose definition involves deformation of
${\mathbb Z}\widehat{W}$ to the Hecke algebra of $\widehat{W}$, see Appendix~\ref{basic_KL} for details). Recall the anti-spherical module
$M={\mathbb Z}\widehat{W}\otimes _{{\mathbb Z}W} \mathbb{Z}_{\mathrm{sign}}$.
Module $M$ admits the standard basis $H'_{w}$ and canonical basis
$C_{w}'$ indexed by $w \in \widehat{W}^{f}$ and defined as the image of
$H_{w}$ and $C_{w}$ under the natural surjection
${\mathbb{Z}}\widehat{W} \twoheadrightarrow M$.

By Theorem~\ref{char_dom_via_inverse} we have
%
\begin{equation}
\label{mult_neg}
H'_{w_{\gamma}}=\sum _{\nu \in Q^{\vee}}\varepsilon (w_{\gamma }w_{
\nu}^{-1}) {\bf{m} }_{w_{\nu}}^{w_{\gamma}} C'_{w_{\nu}}.
\end{equation}

We set
\begin{equation*}
T_{\gamma}:=\varepsilon (w_{\gamma})H_{w_{\gamma}}',~C_{\nu}=
\varepsilon (w_{\nu})C_{w_{\nu}}'.
\end{equation*}
For $\nu , \gamma \in Q^{\vee}$ we see that
%
\begin{equation}
\label{mult_neg_t}
T_{\gamma}=\sum _{\nu \in Q^{\vee}} {\bf{m}}_{w_{\nu}}^{w_{\gamma}} C_{
\nu}.
\end{equation}

\begin{Rem}%
\label{normal_coeff}
Note that ${\bf{m}}^{w_{\nu}}_{w_{\nu}}=1$.
\end{Rem}

\section{Geometry of Springer resolution and realization of $M$}
\label{sect_geom_springer_KL}

We will analyze ${\bf{m} }_{w_{\nu}}^{w_{\gamma}}$ based on the ``coherent''
realization of $M$ (as the equivariant $K$-theory of the Springer resolution
for the Langlands dual group $G^{\vee}$). Let us first of all recall basic
things about the Springer resolution.

\subsection{Springer resolution}

Recall that $G^{\vee}$ is the adjoint group with Lie algebra
$\mathfrak{g}^{\vee}$, $\mathcal{B}$ is the flag variety of
$\mathfrak{g}^{\vee}$ and ${\cal N}\subset \mathfrak{g}^{\vee}$ is the
variety of nilpotent elements. Recall also that
$\widetilde{{\mathcal{N}}}=T^{*}{\mathcal{B}}$ and
$\pi \colon \widetilde{{\mathcal{N}}}\to {\cal N}$ is the projection (Springer)
map.

The lattice $Q^{\vee}$ is the root lattice of $G^{\vee}$ that identifies
with the weight lattice
$\operatorname{Hom}(T,{\mathbb{C}}^{\times})$ of characters of a maximal
torus $T \subset G^{\vee}$ (here we use that $G^{\vee}$ is adjoint). For
$\gamma \in Q^{\vee}$ we denote by
${\cal O}_{{\mathcal{B}}}(\gamma ):=G^{\vee} \times ^{B} {\mathbb{C}}_{-
\gamma}$ the corresponding $G^{\vee}$-equivariant line bundle on
${\mathcal{B}}$, and ${\cal O}_{\widetilde{{\mathcal{N}}}}(\gamma )$ is
the pull back of ${\cal O}_{{\mathcal{B}}}(\gamma )$ to
$\widetilde{{\cal N}}$.

The variety $\widetilde{{\mathcal{N}}}$ contains an open $G^{\vee}$-orbit
${\mathbb{O}}^{\mathrm{reg}}$ (we identify
${\mathbb{O}}^{\mathrm{reg}} \subset {\mathcal{N}}$ with its preimage in
$\widetilde{{\mathcal{N}}}$). The complement
$\widetilde{{\mathcal{N}}} \setminus {\mathbb{O}}^{\mathrm{reg}}$ is the
divisor in $\widetilde{{\mathcal{N}}}$. It is a standard fact that the
irreducible components of this divisor are parametrized by simple coroots
$\alpha ^{\vee}=\alpha _{i}^{\vee}$, $i=1,\ldots ,r$, as follows. Let
$B \in {\mathcal{B}}$ be a Borel subgroup and let
$P_{\alpha ^{\vee}} \supset B$ be the minimal parabolic, corresponding
to $\alpha ^{\vee}$. Let
$\pi _{\alpha ^{\vee}}\colon {\mathcal{B}}=G^{\vee}/B
\twoheadrightarrow G/P_{\alpha ^{\vee}}$ be the projection. Set
$\widetilde{{\mathcal{N}}}_{\alpha ^{\vee}}:=T^{*}(G^{\vee}/P_{
\alpha ^{\vee}}) \times _{G^{\vee}/P_{\alpha ^{\vee}}} {\mathcal{B}}$.
The differential of $\pi _{\alpha ^{\vee}}$ provides the closed embedding
$i_{\alpha ^{\vee}} \colon \widetilde{{\mathcal{N}}}_{\alpha ^{\vee}}
\subset \widetilde{{\mathcal{N}}}$. Subvarieties
$\widetilde{{\mathcal{N}}}_{\alpha ^{\vee}} \subset
\widetilde{{\mathcal{N}}}$ are precisely the irreducible components of
$\widetilde{{\mathcal{N}}} \setminus {\mathbb{O}}^{\mathrm{reg}}$.

The following exact sequence is standard (see, for example,
\cite[Equation~(13)]{BHum} or \cite[Lemma 5.3]{a}).
%
\begin{Lem}%
\label{short_exact_div_on_tilde_N}
There is a canonical (in particular, $G^{\vee}$-equivariant) exact sequence
of (coherent) sheaves on $\widetilde{{\mathcal{N}}}$:
%
\begin{equation}
\label{ex_seq_O_al_N_tilde}
0 \rightarrow {\mathcal{O}}_{\widetilde{{\mathcal{N}}}}(\alpha ^{\vee})
\rightarrow {\mathcal{O}}_{\widetilde{{\mathcal{N}}}} \rightarrow i_{
\alpha ^{\vee}*}{\mathcal{O}}_{\widetilde{{\mathcal{N}}}_{\alpha ^{
\vee}}} \rightarrow 0.
\end{equation}
\end{Lem}

\subsection{Realization of $M$ via the equivariant $K$-theory of $\widetilde{{\mathcal{N}}}$ and canonical basis}

For a group $H$, acting on an algebraic variety $X$, we let
$K^{H}(X)$ denote the Grothendieck group of $H$-equivariant coherent sheaves
on $X$.

\begin{Thm}[{See e.g. \cite{CG}}]%
\label{CGThm}
We have a canonical isomorphism $M \simeq \break K^{G^{\vee}}(\widetilde{{\mathcal{N}}})$, such that the element $T_{\gamma}$
is sent to $[{\cal O}_{\widetilde{{\mathcal{N}}}}(\gamma )]$, the action of
$\gamma \in Q^{\vee}\subset \widehat{W}$ corresponds to the automorphism
induced by the functor ${\cal F}\mapsto {\cal F}\otimes _{{\cal O}_{
\widetilde{{\mathcal{N}}}}} {\cal O}_{\widetilde{{\mathcal{N}}}}( \gamma )$.
\end{Thm}

We will need some information about the image of $C_{\nu}$ in
$K^{G^{\vee}}(\widetilde{{\mathcal{N}}})$. Recall the notion of a two-sided
cell in $\widehat{W}$ (see \cite{Lu1}). These are certain subsets in
$\widehat{W}$, the set of two-sided cells comes equipped with a partial
order. There exists a canonical bijection between the set of two-sided
cells and $G^{\vee}$-orbits on ${\cal N}$ (see \cite{Lu}). The order on
two-sided sets corresponds to the adjunction order on nilpotent orbits
(see \cite[Theorem 4(b)]{BHe0}).

We will write ${\mathbb{O}}_{c}$ for the orbit, corresponding to the two
sided cell of $c$, and ${\mathbb{O}}_{\leqslant c}$ for its closure. Let
$\widetilde{{\mathbb{O}}}_{c}$,
$\widetilde{{\mathbb{O}}}_{\leqslant c}$ be the (reduced) preimages of
${\mathbb{O}}_{c}$, ${\mathbb{O}}_{\leqslant c}$ in
$\widetilde{{\cal N}}$.

Let
$K^{G^{\vee}}_{\leqslant c}(\widetilde{{\cal N}})\subset K^{G^{\vee}}(
\widetilde{{\cal N}})$ denote the subgroup generated by classes of sheaves
supported on $\widetilde{{\mathbb{O}}}_{\leqslant c}$. Let
$K^{G^{\vee}}_{< c}(\widetilde{{\mathcal{N}}})\subset K^{G^{\vee}}(
\widetilde{{\mathcal{N}}})$ denote the subgroup generated by classes of
sheaves supported on
$\widetilde{{\mathbb{O}}}_{\leqslant c}\setminus
\widetilde{\mathbb{O}}_{c}$.

\begin{Thm}[{\cite[\S 11.3]{BHe}}]\label{cells}
The ${\mathbb{Z}}$-module $K^{G^{\vee}}_{\leqslant c}({\widetilde{\cal N}})$
is spanned by the elements of the canonical basis $C_{\nu}$ such that $w_{\nu}\!\in\! c'\!\leqslant\! c$. The quotient $K^{G^{\vee}}_{\leqslant
c}(\widetilde{{\mathcal{N}}})/K^{G^{\vee}}_{< c}(\widetilde{{\mathcal{N}}})$
has a ${\mathbb{Z}}$-basis, consisting of classes of $C_{\nu},\, w_{\nu }\in
c$.
\end{Thm}

\begin{Cor}%
\label{ker_open}
For every open $G^{\vee}$-invariant locally closed subvariety
$U \subset {\mathcal{N}}$ and $\widetilde{U}=\pi ^{-1}(U)$, the kernel
of the surjection
$K^{G^{\vee}}(\widetilde{{\mathcal{N}}}) \twoheadrightarrow K^{G^{
\vee}}(\widetilde{U})$ is spanned over ${\mathbb{Z}}$ by
$\{C_{\nu}\mid  {\mathbb{O}}_{w_{\nu}} \not \subset U\}$.
\end{Cor}
\begin{proof}
Follows from Theorem~\ref{cells}.
\end{proof}

In fact, \cite{AB,BHe} provide more information on the images of
$C_{\nu}$ in $K^{G^{\vee}}({\widetilde{\cal N}})$: these are exactly the
classes of irreducible objects in the heart of a certain $t$-structure
on $D^{b}(\operatorname{Coh}^{G^{\vee}}(\widetilde{{\cal N}}))$, the so
called exotic $t$-structure, introduced in \cite{BHum} (and called ``perversely
exotic'' in \cite{bm}). Recall their explicit description.

\begin{Thm}\label{NC_S}
There exist a $G^{\vee}$-equivariant vector bundle ${\mathcal{E}}$ on
${\widetilde{\cal N}}$ with the following properties.
\begin{enumerate}
\item[0)] The structure sheaf ${\cal O}$ is a direct summand in
${\mathcal{E}}$. Also, ${\mathcal{E}}^{*}$ is globally generated.

\item[1)] Let ${\mathbf{A}}=\operatorname{End}({\mathcal{E}})^{op}$. Then the
functor ${\mathcal{F}}\mapsto RHom({\mathcal{E}},{\mathcal{F}})$ provides
equivalences
\begin{gather*}
D^{b}(\operatorname{Coh}({\widetilde{\cal N}}))\cong D^{b}({
\mathbf{A}}-mod),%
\\
D^{b}(\operatorname{Coh}^{G^{\vee}}({\widetilde{\cal N}}))\cong D^{b}({
\mathbf{A}}-mod^{G^{\vee}}),%
\end{gather*}
where ${\mathbf{A}}-mod$ is the category of finitely generated
${\mathbf{A}}$-modules.

\item[2)] Irreducible objects in the heart of the exotic $t$-structure are in
bijection with pairs $({\mathbb{O}},M)$ where ${\mathbb{O}}$ is a
$G^{\vee}$-orbit in ${\cal N}$ and $M$ is an irreducible equivariant module
for ${\mathbf{A}}|_{\mathbb{O}}$.

More precisely, given $({\mathbb{O}},M)$ as above there exists an object
${\bf L}_{{\mathbb{O}},M}\in D^{b}({\mathbf{A}}-mod^{G^{\vee}})$ uniquely
characterized by the following properties: ${\bf L}_{{\mathbb{O}},M}$ is
supported on the closure of ${\mathbb{O}}$, its restriction to the open
subset ${\mathbb{O}}$ of its support is isomorphic to
$M\Big[-\frac{\operatorname{codim}({\mathbb{O}})}{2}\Big]$; for every orbit
${\mathbb{O}}'\ne {\mathbb{O}}$ the object $i_{{\mathbb{O}}'}^{*}({\bf
L}_{{\mathbb{O}},M})$ is concentrated in cohomological degrees less than
$\frac{\operatorname{codim}({\mathbb{O}}')}{2}$ and the object
$i_{\mathbb{O}'}^{!}({\bf L}_{{\mathbb{O}},M})$ is concentrated in
cohomological degrees greater than
$\frac{\operatorname{codim}({\mathbb{O}}')}{2}$. The object ${\bf
L}_{{\mathbb{O}},M}$ is irreducible in the heart of the exotic
$t$-structure and every such irreducible is isomorphic to ${\bf
L}_{{\mathbb{O}},M}$ for some $({\mathbb{O}},M)$.

\item[3)] The classes of ${\bf L}_{{\mathbb{O}},M}$ form the canonical basis in
$M$, where we identified
$K^{G^{\vee}}({\widetilde{\cal N}}) \simeq M$ using the map
$[{\mathcal{O}}(\gamma )] \mapsto T_{-\gamma}$.
\end{enumerate}
\end{Thm}

\begin{proof}
The vector bundle ${\mathcal{E}}$ is introduced in
\cite[Theorem 1.5.1]{bm}, which asserts that ${\mathcal{E}}$ is a tilting
generator, i.e. statement 1) holds. It contains ${\mathcal{O}}$ as a direct
summand by \cite[Theorem 1.8.2 (a,1)]{bm}. Statements 2), 3) follow from
\cite[Theorem 6.2.1]{bm}.

It remains to show that ${\mathcal{E}}^{*}$ is globally generated. Recall
from \cite[\S 1.8]{bm} a collection of tilting vector bundles on
${\widetilde{\cal N}}$ parametrized by alcoves, here ${\mathcal{E}}$ corresponds
to the fundamental alcove and ${\mathcal{E}}^{*}$ is the tilting bundle
corresponding to the anti-fundamental one. A vector bundle
${\mathcal{V}}$ is globally generated iff for every morphism from
$f\colon {\mathcal{V}}\to k_{x}$, where $k_{x}$ is a skyscraper sheaf, there
exists a morphism $\phi\colon {\mathcal{O}}\to {\mathcal{V}}$ such that
$f\circ \phi \ne 0$. It is easy to see that when ${\mathcal{V}}$ is a dilation
equivariant vector bundle on ${\widetilde{\cal N}}$, it suffices to consider
$x$ in the zero section
$G^{\vee}/B^{\vee }\subset {\widetilde{\cal N}}$. Moreover, it is enough
to prove a similar statement over a field $k$ of a large positive characteristic.
In that case, we can apply localization functor corresponding to the point
$-2\rho $ in the anti-fundamental alcove to translate this statement into
one in the representation theory of the Lie algebra
${\frak g}^{\vee}$ over $k$. The localization equivalence relates the functor
of global sections to translation to the singular central character
$-\rho $, a skyscraper sheaf is identified with a module $M_{0}^{*}$ where
$M_{0}$ is a baby Verma module with highest weight zero, while
${\mathcal{E}}^{*}$ (pulled back to the formal neighborhood of the zero
section) is identified with a projective generator in the corresponding
category of ${\frak g}^{\vee}$-modules. Thus the statement reduces to showing
that translation functor $T_{-2\rho \to -\rho}$ does not kill any nonzero
submodule of $M_{0}^{*}$. This follows from the standard fact that the
adjunction arrow
$M_{0}^{*}\to T_{-\rho \to -2\rho} T_{-2\rho \to \rho}M_{0}^{*}$ is injective.
\end{proof}

\begin{Rem}
Notice that the isomorphisms between
$K^{G^{\vee}}({\widetilde{\cal N}})$ and the anti-spherical module
$M$ in Theorem~\ref{CGThm} and in Theorem~\ref{NC_S} are different: one
sends $[{\cal O}(\gamma )]$ to $T_{\gamma}$ while the other sends it to
$T_{-\gamma}$. Both are natural from some perspective and both appear in
the literature. A related issue is the choice of the isomorphism between
the weight lattice and the Picard group of the flag variety: we use the
one sending a dominant weight to a semi-ample line bundle, thus the weights
of the action of a Borel subgroup on the nilpotent radical of its Lie algebra
correspond to negative roots, while some authors prefer the opposite convention
(see \cite[Section 6.1.11]{CG}). We will work with the isomorphism of Theorem~\ref{CGThm}, see below.
\end{Rem}

For technical reasons we prefer to work with the globally generated tilting
bundle ${\mathcal{E}}^{*}$ rather than with ${\mathcal{E}}$. Thus we set
$A=End({\mathcal{E}}^{*})^{op}={\mathbf{A}}^{op}$ and consider the equivalence
$D^{b}(Coh({\widetilde{\cal N}}))\cong D^{b}(A-mod)$,
${\cal F}\mapsto RHom({\mathcal{E}}^{*},{\cal F})$. In this approach it
is more natural to use the isomorphism between
$K^{G^{\vee}}({\widetilde{\cal N}})$ and the anti-spherical module
$M$ sending $[{\cal O}(\gamma )]$ to $T_{\gamma}$, see Theorem~\ref{CGThm}. With this identification, elements of the canonical basis
correspond to classes ${\mathcal{L}}_{{\mathbb{O}},M}$ where
${\mathbb{O}}$ is a $G^{\vee}$-orbit in ${\cal N}$ and $M$ is an irreducible
equivariant module for $A|_{\mathbb{O}}$ characterized as in Theorem~\ref{NC_S} (2).

For $e\in {\cal N}$ let $A_{e}$ denote the corresponding specialization
of $A$.

We recall an explicit description of (complexes of) coherent sheaves corresponding
to some irreducible $A$-modules. Let ${\mathcal{A}}=A-mod$, the category
of finitely generated $A$-modules. We identify ${\mathcal{A}}$ with the
corresponding full subcategory in
$D^{b}(\operatorname{Coh}({\widetilde{\cal N}}))$. For
$e\in {\cal N}$ let ${\mathcal{A}}_{e}$ be the full subcategory in
${\mathcal{A}}$, consisting of objects set-theoretically supported on
$\pi ^{-1}(e)$

Recall that $e$ is a subregular nilpotent.

Irreducible components of ${\cal B}_{e}$ are parametrized by
$\tilde I$; here $\tilde I = I$ if the Dynkin diagram of ${\frak g}$ is
simply laced (see Sections~\ref{subsec_str_B_e_DE} and~\ref{subsec_str_B_e_A} below) and $\tilde I$ is the set of vertices of
the unfolding of the Dynkin graph in general (see \cite{sl}). Each irreducible
component $\Pi _{i}$, $i \in \tilde I$, is isomorphic to
${\mathbb{P}}^{1}$.
%
\begin{Lem}[{cf. \cite[Example 5.3.3]{bmr0}}]%
\label{MacKay}
 Let $e\in {\cal N}$ be a subregular nilpotent.
The irreducible objects in ${\mathcal{A}}_{e}$ are:
${\mathcal{O}}_{\Pi _{i}}(-1)[1]$, ${\mathcal{O}}_{\pi ^{-1}(e)}$ where
$\pi ^{-1}(e)$ is the schematic fiber of the Springer map $\pi $.
\end{Lem}
\begin{proof}
Property $1)$ in Theorem~\ref{NC_S} shows that ${\mathcal{E}}^*$ is tilting,
i.e. $Ext^{>0}({\mathcal{E}}^*,{\mathcal{E}}^*)=0$. Since ${\cal O}$ is a direct
summand of ${\mathcal{E}}^*$, it follows that
$H^{1}({\mathcal{E}})=H^{1}({\mathcal{E}}^{*})=0$. Using that
${\mathcal{E}}^*$ is globally generated we conclude that
${\mathcal{E}}|_{\Pi _{i}}$ is a sum of copies of
${\mathcal{O}}_{\Pi _{i}}$ and ${\mathcal{O}}_{\Pi _{i}}(1)$ (note that
$R\Gamma $ has homological dimension $1$ so $H^{1}$ is right exact, thus
$H^{1}({\mathcal{E}})=0$ implies $H^{1}(E|_{\Pi _{i}})=0$).

It is then clear that ${\mathcal{O}}_{\Pi _{i}}(-1)[1]$ lies in the heart.

Since ${\mathcal{O}}$ is a direct summand in ${\mathcal{E}}^*$ it follows
that the objects ${\cal F}$ such that $R\Gamma ({\cal F})=0$ form a Serre
subcategory ${\mathcal{A}}_{e}^{0}$ in ${\mathcal{A}}_{e}$.

For an object ${\cal F}$ supported on $\pi ^{-1}(e)$ with
$R\Gamma ({\cal F})=0$, each cohomology sheaf of ${\cal F}$ is an extension
of sheaves of the form ${\cal O}_{\Pi _{i}}(-1)$, see, for example,
\cite[Theorem 2.3]{KV}. It follows that ${\mathcal{A}}_{e}^{0}$ consists
of objects of the form ${\cal F}[1]$ where ${\cal F}$ is an extension of
${\cal O}_{\Pi _{i}}(-1)$ and that ${\cal O}_{\Pi _{i}}(-1)[1]$ are irreducible.

We know that the classes of irreducible objects form a basis in the Grothen\-dieck
group $K(\operatorname{Coh}(\pi ^{-1}(e))$ which is isomorphic to homology
of $\pi ^{-1}(e)$ and has dimension $|\tilde I|+1$. Thus there exists a
unique irreducible $L_{0}$ not isomorphic to
${\cal O}_{\Pi_{i}}(-1)[1]$. It remains to show that
$L_{0}\cong {\mathcal{O}}_{\pi ^{-1}(e)}$. In the usual $t$-structure there
exists a filtration of ${\mathcal{O}}_{\pi ^{-1}(e)}$ starting with
$\BC_{p}$ for some point $p \in {\cal B}_{e}$ with the other subquotients being
${\mathcal{O}}_{\Pi_{i}}(-1)$. Let $F_j$ be the $j$'th quotient with respect to this filtration ($F_0=\BC_p$). Let us prove by the induction on $j$ that ${\mathcal{O}}_{\pi ^{-1}(e)}$ is in the heart, and is a subobject
of $\BC_{p}$.  The base of the induction follows from the fact that  $\BC_{p}$ clearly lies in the
heart of our $t$-structure. To prove the induction step, consider the exact triangle $F_{j+1} \rightarrow F_{j} \rightarrow \mathcal{O}_{\Pi_j}(-1)[1] \rightarrow $. We already know that $F_j,\, \mathcal{O}_{\Pi_j}(-1)[1] \in \mathcal{A}_e$, moreover, we have already proved that $\mathcal{O}_{\Pi_j}(-1)[1]$ is simple. We conclude that (nonzero) morphism $F_{j} \rightarrow \mathcal{O}_{\Pi_j}(-1)[1]$ must be surjective. It follows that $F_{j+1} \in \mathcal{A}_e$ and is a subobject of $F_j$.

It is clear that
$\operatorname{Hom}({\cal O}_{\Pi _{i}}(-1)[1],{\cal O}_{\pi ^{-1}}(e))$
vanishes. We check that Hom vanishing in the other direction also holds.
To this end, consider the Slodowy variety $\widetilde{S}_{e}$, resolving
the Slodowy slice $S_{e}$ to $e \in {\mathcal{N}}$ (see \cite{sl}). We
have an exact sequence
%
\begin{equation}
\label{ex_pi_e}
0 \rightarrow {\mathcal{O}}_{\widetilde{S}^{(1)}_{e}}(-\pi ^{-1}(e))
\rightarrow {\mathcal{O}}_{\widetilde{S}^{(1)}_{e}} \rightarrow {
\mathcal{O}}_{\pi ^{-1}(e)} \rightarrow 0.
\end{equation}
Since $S_{e}$ is affine,
${\mathcal{O}}_{\widetilde{S}^{(1)}_{e}}(-\pi ^{-1}(e))$ is globally generated,
which yields
\begin{equation*}
\operatorname{Hom}({\mathcal{O}}_{\pi ^{-1}(e)}, {\mathcal{O}}_{\Pi _{i}}(-1)[1])=0.
\end{equation*}
It follows that both socle and cosocle of
${\mathcal{O}}_{\pi ^{-1}(e)}$ is the sum of copies of $L_{0}$. If
${\mathcal{O}}_{\pi ^{-1}(e)} \neq L_{0}$, then we get a nonconstant endomorphism
of ${\mathcal{O}}_{\pi ^{-1}(e)}$ but again using the exact sequence (\ref{ex_pi_e})
we see that $\Gamma ({\mathcal{O}}_{\pi ^{-1}(e)})$ is one dimensional.
\end{proof}

\begin{Rem}
According to \cite{bm}, base change of $A$ to a slice to a nilpotent orbit
is derived equivalent to the resolution of the slice, this yields a
$t$-structure on the resolution of the slice. When $e$ is subregular, the
resolution of the slice coincides with the minimal resolution of a rational
(Kleinian) singularity. Comparing Lemma~\ref{MacKay} with
\cite[Theorem 2.3]{KV} one sees that in this case the above $t$-structure
coincides with one arising from the McKay equivalence between the derived
category of the resolution and the derived category of the orbifold.
\end{Rem}

Assume now ${\frak g}$ is of type $D$ or $E$. Then the centralizer
$Z_{G^{\vee}}(e)$ is unipotent. It is clear each of the above irreducibles
carries a unique equivariant $Z_{G^{\vee}}(e)$ structure. For
$\mathfrak{g}=\mathfrak{sl}_{n}$ ($n \geqslant 3$) we have
$Z_{e} \simeq {\mathbb{C}}^{\times}$ (recall that
$Z_{e} \subset Z_{G^{\vee}}(e)$ is the reductive part); for a
${\mathbb{C}}^{\times}$-equivariant sheaf ${\mathcal{F}}$ and
$k \in {\mathbb{Z}}$ we denote by ${\mathcal{F}}\langle k\rangle $ the
same sheaf but with the ${\mathbb{C}}^{\times}$-equivariant structure twisted
by the character $t \mapsto t^{k}$ of ${\mathbb{C}}^{\times}$. It is clear
that each of the above irreducibles carries unique (up to a shifting by
$k \in {\mathbb{Z}}$) $Z_{G^{\vee}}(e)$-equivariant structure.

We will use the same notation for the resulting $Z_{G^{\vee}}(e)$-equivariant
sheaves, as well as for the corresponding $G^{\vee}$-equivariant sheaves
on the schematic preimage of the orbit $\pi ^{-1}({\mathbb{O}}_{e}) \xrightarrow{\iota} \widetilde{\mathcal{N}}$.

Consider $U={\mathbb{O}}_{e} \cup {\mathbb{O}}^{\mathrm{reg}}$. We record
a description of the canonical basis in
$K^{G^{\vee}}(\widetilde{U})$ stemming from Theorem~\ref{NC_S} and Lemma~\ref{MacKay}.

\begin{Prop}%
\label{descr_canon}
For $\mathfrak{g}$ of type $D_{n}$ ($n \geqslant 4$) or $E_{6}$,
$E_{7}$, $E_{8}$ the canonical basis of
$K^{G^{\vee}}(\widetilde{U})$ consists of classes of
\begin{equation*}
{\mathcal{O}}_{\widetilde{U}},\, \iota _{*}{\mathcal{O}}_{\pi ^{-1}({
\mathbb{O}}_{e})}[-1],~\text{and}~\iota _{*} {\mathcal{O}}_{\Pi _{i}}({-1}),~i=1,
\ldots ,r.
\end{equation*}

For $\mathfrak{g}=\mathfrak{sl}_{n}$ ($n \geqslant 3$) the canonical basis
of $K^{G^{\vee}}(\widetilde{U})$ consists of classes of
\begin{equation*}
{\mathcal{O}}_{\widetilde{U}},\, \iota _{*}{\mathcal{O}}_{\pi ^{-1}({
\mathbb{O}}_{e})}[-1]\langle k \rangle ,~\text{and}~\iota _{*} {
\mathcal{O}}_{\Pi _{i}}({-1})\langle k \rangle ,~i=1,\ldots ,n-1,\, k
\in {\mathbb{Z}}.
\end{equation*}
\end{Prop}
\begin{proof}
We use the description of the canonical basis provided by the Theorem~\ref{NC_S}. It is easy to see that
${\mathcal{O}}_{{\widetilde{\cal N}}}$ satisfies the properties in Theorem~\ref{NC_S} (2), so it corresponds to an element of the canonical basis
(this is just the image of the
unit element in the affine Hecke algebra).

Also, Theorem~\ref{NC_S} together with Lemma~\ref{MacKay} show the existence
of irreducible objects in the heart of the exotic $t$-structure whose restriction
to $\widetilde{U}$ coincides with the other objects listed in the Proposition
(note that the objects that appear in Lemma~\ref{MacKay} should be shifted
by $[{-}\frac{\operatorname{codim}{\mathbb{O}}_{e}}{2}]=[-1]$). Thus
the statement follows from Theorem~\ref{NC_S} (3).
\end{proof}

\begin{Rem}
The canonical basis in the corresponding module over the affine Hecke algebra
was computed by other methods in \cite{lu_notes_aff} for type $A$ and in
\cite{lu_subreg} for types $D$ and $E$. It is not hard to check that the
basis in Proposition~\ref{descr_canon} after applying Grothendieck-Serre
duality to it (see \cite[Sections 6.10, 6.11, 6.12]{lu_bases_K}) agrees
with those earlier results.
\end{Rem}

\section{The subregular type $D,\, E$ case}
\label{subreg_section}

In this section we assume that $\mathfrak{g}$ is of type $D$ or $E$. Since
$\mathfrak{g}$ is simply-laced, we have
$\mathfrak{g}=\mathfrak{g}^{\vee}$. Recall that $e \in \mathfrak{g}$ is
the subregular nilpotent element and $c \subset \widehat{W}$ is the corresponding
two-sided cell. Recall the $\widehat{W}$-module
$K^{G^{\vee}}(\widetilde{U})$, and that this module has a (canonical) basis
$\bar{C}_{\nu}$ parametrized by $\nu $ such that the corresponding
$w_{\nu}$ lies in $c \cup \{1\}$. Let us describe such $\nu $,
$w_{\nu}$ explicitly.

The following proposition follows from
\cite[Proposition 3.8]{lu_subreg_crit}, see also
\cite[Proposition 3.6]{xu}.
%
\begin{Prop}%
\label{descr_cell_subreg}
The cell $c$, corresponding to the subregular nilpotent $e$, consists of
all nonidentity elements $w \in \widehat{W}$ that have unique reduced decomposition.
\end{Prop}

\begin{Cor}%
\label{descr_nu_i_subreg}
Let $c$ be as in Proposition~\ref{descr_cell_subreg}. The elements
$\nu \in Q^{\vee}$ such that $w_{\nu }\in c$ can be described as follows.
The set of possible $\nu $ is parametrized by $\widehat{I}$. For
$i \in \widehat{I}$ let us connect $i$ with $0 \in \widehat{I}$ by the
segment:
\begin{equation*}
0=j_{0},\,j_{1},\ldots ,\,j_{l-1},\,j_{l}=i.
\end{equation*}
Then the element $\nu _{i}$ is equal to
%
\begin{equation}
\label{nu_i_expl}
\nu _{i}=s_{i}s_{j_{l-1}}\ldots s_{j_{1}}(\theta ) = \theta -\alpha _{j_{1}}^{
\vee}-\cdots -\alpha _{j_{l-1}}^{\vee}-\alpha _{i}^{\vee}
\end{equation}
and the corresponding $w_{\nu _{i}}$ is
\begin{equation*}
w_{\nu _{i}}=w_{i}=s_{i}s_{j_{l-1}}\ldots s_{j_{1}}s_{0}.
\end{equation*}
\end{Cor}
\begin{proof}
It easily follows from Proposition~\ref{descr_cell_subreg} that the elements
of $c$ of the form $w_{\nu}$ are precisely the elements $w_{i}$,
$i \in \widehat{I}$. Recall now that
\begin{equation*}
w_{i}=s_{i}s_{j_{l-1}}\ldots s_{j_{1}}s_{0}= s_{i}s_{j_{l-1}}\ldots s_{j_{1}}s_{
\theta }t_{-\theta}=t_{s_{i}s_{j_{l-1}}\ldots s_{j_{1}}(\theta )}s_{
\theta }s_{j_{1}}\ldots s_{j_{l-1}}s_{i}.
\end{equation*}
We conclude that $\nu $ that corresponds to $w_{i}$ is equal to
$s_{i}s_{j_{l-1}}\ldots s_{j_{1}}(\theta )$. It remains to note that
\begin{align*}
s_{i}s_{j_{l-1}}\ldots s_{j_{1}}(\theta )
&=s_{i}s_{j_{l-1}}\ldots s_{j_{1}}(\delta -\alpha _{0})
\\
&=\delta -s_{i}s_{j_{l-1}}\ldots s_{j_{1}}(\alpha _{0})=\delta -(\alpha _{0}+\alpha _{j_{1}}+\cdots +\alpha _{j_{l-1}}+\alpha _{i})
\\
&=\theta -\alpha _{j_{1}}-\cdots -\alpha _{j_{l-1}}-\alpha _{i}.
\end{align*}\vskip-22pt
\end{proof}

\begin{Rem}
Note that $\nu _{0}=\theta $, $w_{0}=s_{0}$.
\end{Rem}

Our goal is to identify $\widehat{W}$-module
$K^{G^{\vee}}(\widetilde{U}) \otimes {\mathbb{Z}}_{\mathrm{sign}}$ with
$\widehat{\mathfrak{h}}_{{\mathbb{Z}}}={\mathbb{Z}}Q^{\vee }\oplus {
\mathbb{Z}}K \oplus {\mathbb{Z}}d$.

\subsection{Structure of $K({\cal B}_{e})$}
\label{subsec_str_B_e_DE}

Let us first of all describe the geometry of the variety
${\cal B}_{e}$. Recall that ${\cal B}_{e}$ is the fiber over $e$ of the
Springer resolution
$\pi \colon \widetilde{{\mathcal{N}}} \rightarrow {\mathcal{N}}$. Directly
from the definitions we have
\begin{equation*}
{\mathcal{B}}_{e}=\{\mathfrak{b}' \in {\mathcal{B}}\mid  e \in
\mathfrak{n}_{\mathfrak{b}'}\},
\end{equation*}
where $\mathfrak{n}_{\mathfrak{b'}} \subset \mathfrak{b} $ is the unipotent
radical of the Borel subalgebra $\mathfrak{b}'$.

The following lemma is standard (see \cite{sl}).
%
\begin{Lem}
For every $i=1,\ldots ,r$ there exists the unique parabolic subalgebra
$\mathfrak{p}_{e,i}$ such that
\begin{enumerate}
\item[$(1)$] $\mathfrak{p}_{e,i}$ is conjugate to the standard minimal parabolic
subalgebra, corresponding to $\alpha _{i}^{\vee}$,

\item[$(2)$] the nilradical of $\mathfrak{p}_{e,i}$ contains $e$.
\end{enumerate}
We denote the corresponding parabolic subgroup by
$P_{e,i} \subset G^{\vee}$.
\end{Lem}

For $i \in 1,\ldots ,r$ let ${\mathcal{P}}_{i}=G^{\vee}/P_{e,i}$ be the variety
of parabolic subalgebras of $\mathfrak{g}^\vee$ of type
$\mathfrak{p}_{e,i}$. The following proposition is standard (see
\cite{sl}).
%
\begin{Prop}
The variety $\mathcal{B}_{e}$ has $r$ irreducible components
$\Pi _{i}$, $i=1,\ldots ,r$. The component $\Pi _{i}$, corresponding to
$i \in I$, is the fiber of the morphism
$\mathcal{B} \rightarrow \mathcal{P}_{i}$ over the point
$\mathfrak{p}_{e,i}$ (in particular,
$\Pi _{i} \simeq {\mathbb{P}}^{1}$). For $i \neq j$ components
$\Pi _{i}$, $\Pi _{j}$ intersect iff $(\alpha _{i},\alpha _{j})=-1$. If
this is the case, then $\Pi _{i}$, $\Pi _{j}$ intersect transversally at
one point.
\end{Prop}

We pick any point $p \in {\cal B}_{e}$ and denote by
$[{\mathbb{C}}_{p}]$ the class in $K({\cal B}_{e})$ of the skyscraper sheaf
${\mathbb{C}}_{p}$.

\begin{Rem}
Note that $[{\mathbb{C}}_{p}]$ does not depend on the point $p$. Indeed,
it is enough to check this for ${\mathbb{P}}^{1}$ where every skyscraper
sheaf is equal to
$[{\mathcal{O}}_{{\mathbb{P}}^{1}}]-[{\mathcal{O}}_{{\mathbb{P}}^{1}}(-1)]$.
\end{Rem}

For $m \in {\mathbb{Z}}$ we denote by $O_{i}^{m}$ the class of
${\mathcal{O}}_{\Pi _{i}}(m)$ in $K(\Pi _{i})$. We denote by
$o_{i}^{m} \in K(\mathcal{B}_{i})$ the direct image of $O_{i}^{m}$ under
the closed embedding $\Pi _{i} \subset \mathcal{B}_{i}$.

The following lemma holds by \cite[Section 3.4]{lu_subreg}.
%
\begin{Lem}%
\label{gen_K_ADE_fiber}
The ${\mathbb{Z}}$-module $K(\mathcal{B}_{e})$ is spanned by
$o_{i}^{m}$, $i=1,\ldots ,r$, $m \in {\mathbb{Z}}$, subject to relations
\begin{equation*}
o_{i}^{m}-o_{i}^{m-1}=o_{j}^{m}-o_{j}^{m-1},\,o_{i}^{m+1}-o_{i}^{m}=o_{i}^{m}-o_{i}^{m-1},~i,j=1,
\ldots ,r.
\end{equation*}
\end{Lem}

\begin{Rem}
The relations follow from the standard exact sequences on
${\mathbb{P}}^{1}$ (and their twistings by
${\mathcal{O}}_{{\mathbb{P}}^{1}}(m)$):
\begin{gather}
\begin{gathered}
0 \rightarrow {\mathcal{O}}_{{\mathbb{P}}^{1}}(-1) \rightarrow {
\mathcal{O}}_{{\mathbb{P}}^{1}} \rightarrow {\mathbb{C}}_{p}
\rightarrow 0,%
\\
\label{seq_p_1_eul}
0 \rightarrow {\mathcal{O}}_{{\mathbb{P}}^{1}}(-1) \rightarrow {
\mathcal{O}}_{{\mathbb{P}}^{1}}^{\oplus 2} \rightarrow {\mathcal{O}}_{{
\mathbb{P}}^{1}}(1) \rightarrow 0.%
\end{gathered}
\end{gather}
\end{Rem}

\begin{Cor}[{\cite[Section 3.4]{lu_subreg}}]
The ${\mathbb{Z}}$-module
$K({\cal B}_{e})$ has a basis $o_{i}^{-1}$, $i=1,\ldots ,r$,
$[{\mathbb{C}}_{p}]$.
\end{Cor}
\begin{proof}
Easily follows from Lemma~\ref{gen_K_ADE_fiber}.
\end{proof}

\begin{Rem}
\emph{Note that for every $i \in I$ we have
$[{\mathbb{C}}_{p}]=o_{i}^{0}-o_{i}^{-1}$ and, more generally,
$o_{i}^{m}-o_{i}^{m-1}=[{\mathbb{C}}_{p}]$ for every
$m \in {\mathbb{Z}}$.}
\end{Rem}

The following lemma holds by \cite[Lemma 3.6]{lu_subreg}.
%
\begin{Lem}
For $\gamma \in Q^{\vee}$ we have
\begin{enumerate}
\item[$(a)$] $t_{\gamma}([{\mathbb{C}}_{p}])=[{\mathbb{C}}_{p}]$,

\item[$(b)$]
$t_{\gamma}(o_{i}^{m})=o_{i}^{m+\langle \alpha _{i},\gamma \rangle}$.
\end{enumerate}
\end{Lem}

The following lemma holds by \cite[Lemmas 3.7, 3.8]{lu_subreg}.
%
\begin{Lem}
For $i \in I$ we have
\begin{enumerate}
\item[$(a)$] $s_{i}([{\mathbb{C}}_{p}])=-[{\mathbb{C}}_{p}] $,

\item[$(b)$] $s_{i}(o_{i}^{-1})=o_{i}^{-1}$.
\end{enumerate}
\end{Lem}

The following lemma holds by \cite[Lemma 3.12]{lu_subreg}.
%
\begin{Lem}
If $i,j \in I$ are such that $(\alpha _{i},\alpha _{j})=0$, then
$s_{i}(o_{j}^{-1})=-o_{j}^{-1}$.
\end{Lem}

The following lemma holds by \cite[Lemma 3.11]{lu_subreg}.
%
\begin{Lem}
If $i,j \in I$ are such that $(\alpha _{i},\alpha _{j})=-1$, then
$s_{i}(o_{j}^{-1})=-o_{j}^{-1}-\nobreak o_{i}^{-1}$.
\end{Lem}

Recall that ${\mathbb{Z}}_{\mathrm{sign}}$ is the one dimensional sign
representation of $\widehat{W}$. We finally obtain the following proposition.
%
\begin{Prop}%
\label{descr_of_k_Spring}
There is an isomorphism of $\widehat{W}$-modules
\begin{equation*}
K({{\mathcal{B}}}_{e}) \simeq (\mathfrak{h}_{{\mathbb{Z}}} \oplus {
\mathbb{Z}}K) \otimes {\mathbb{Z}}_{\mathrm{sign}},
\end{equation*}
given by
\begin{equation*}
K \otimes 1 \mapsto [{\mathbb{C}}_{p}],\, \alpha ^{\vee}_{i} \otimes 1
\mapsto -o_{i}^{-1},\,i=1,\ldots ,r.
\end{equation*}
\end{Prop}
\begin{proof}
The only nontrivial part is to compare the action of
$t_{\alpha _{j}^{\vee}}$ on $o_{i}^{-1}$ with the action of
$t_{\alpha _{j}^{\vee}}$ on $\alpha _{i}^{\vee}$. If
$(\alpha _{i},\alpha _{j})=0$, then we have
\begin{equation*}
t_{\alpha _{j}^{\vee}}(o_{i}^{-1})=o_{i}^{-1}
\end{equation*}
as desired. If $(\alpha _{i},\alpha _{j})=-1$, then we have
\begin{equation*}
t_{\alpha _{j}^{\vee}}(o_{i}^{-1})=o_{i}^{-2}=o_{i}^{-1}-[{\mathbb{C}}_{p}].
\end{equation*}
Finally, if $i=j$ then we have
\begin{equation*}
t_{\alpha _{i}^{\vee}}(o_{i}^{-1})=o_{i}^{1}=o_{i}^{0}+{\mathbb{C}}_{p}=o_{i}^{-1}+2[{
\mathbb{C}}_{p}].\qedhere
\end{equation*}
\end{proof}

\subsection{Structure of $K^{G^{\vee}}(\widetilde{U})$}

Let us describe the $\widehat{W}$-module
$K^{G^{\vee}}(\widetilde{U})$ (recall that $\mathfrak{g}$ is of type
$D$, $E$). Recall that we have the closed embedding
$\iota \colon \widetilde{{\mathbb{O}}}_{e} \subset \widetilde{U}$ (where $\widetilde{{\mathbb{O}}}_{e}$ is the schematic preimage of ${\mathbb{O}}_e \subset U$). It induces
the homomorphism
$\iota _{*}\colon K^{G^{\vee}}(\widetilde{{\mathbb{O}}}_{e})
\rightarrow K^{G^{\vee}}(\widetilde{U})$ of $\widehat{W}$-modules. Note
also that we have the natural identification
$K^{G^{\vee}}(\widetilde{{\mathbb{O}}}_{e})=K({\cal B}_{e})$.

Recall that by Theorem~\ref{cells} and Corollary~\ref{ker_open} we have
an exact sequence of $\widehat{W}$-modules
\begin{equation*}
0 \rightarrow K({\cal B}_{e}) \xrightarrow{\iota _{*}} K^{G^{\vee}}(
\widetilde{U}) \xrightarrow{} K^{G^{\vee}}({\mathbb{O}}^{\mathrm{reg}})
\rightarrow 0.
\end{equation*}

\begin{Prop}%
\label{ident_K_U_DE}
There is an isomorphism of $\widehat{W}$-modules
\begin{equation*}
K^{G^{\vee}}(\widetilde{U}) \simeq (\mathfrak{h}_{{\mathbb{Z}}}
\oplus {\mathbb{Z}}K \oplus {\mathbb{Z}}d) \otimes {\mathbb{Z}}_{
\mathrm{sign}} = \widehat{\mathfrak{h}}_{{\mathbb{Z}}} \otimes {
\mathbb{Z}}_{\mathrm{sign}},
\end{equation*}
given by ($i \in \{1,2,\ldots ,r\}$)
\begin{equation*}
K \mapsto \iota _{*}[{\mathbb{C}}_{p}],\,\alpha _{0}^{\vee} \mapsto
\iota _{*}[{\mathcal{O}}_{\pi ^{-1}(e)}],\, \alpha ^{\vee}_{i}
\mapsto -\iota _{*}o_{i}^{-1},\, d \mapsto [{\mathcal{O}}_{
\widetilde{U}}].
\end{equation*}

The images of the elements of the canonical basis under this isomorphism
are as follows:
\begin{equation*}
\bar{C}_{0}=\bar{T}_{0}\mapsto d,\, \bar{C}_{\nu _{i}} \mapsto -
\alpha _{i}^{\vee}, \, i=0,1,\ldots ,r,
\end{equation*}
where the $\nu _{i}$ are defined by (\ref{nu_i_expl}).
\end{Prop}
\begin{proof}
Recall that $K^{G^{\vee}}(\widetilde{U})$ has a basis
\begin{equation*}
\iota _{*}[{\mathbb{C}}_{p}], \iota _{*}o_{1}^{-1}, \ldots , \iota _{*}
o_{r}^{-1}, [{\mathcal{O}}_{\widetilde{U}}].
\end{equation*}
Moreover,
$\iota _{*}[{\mathbb{C}}_{p}], \iota _{*}o_{1}^{-1}, \ldots , \iota _{*}
o_{r}^{-1}$ form the submodule isomorphic to
$K(\mathcal{B}_{e}) \simeq (\mathfrak{h}_{{\mathbb{Z}}} \oplus {
\mathbb{C}}K) \otimes {\mathbb{Z}}_{\mathrm{sign}}$ (see Proposition~\ref{descr_of_k_Spring}).

It remains to compute the action of $\widehat{W}$ on
$[{\mathcal{O}}_{\widetilde{U}}]$. The module
$K^{G^{\vee}}(\widetilde{U})$ is the quotient of
$K^{G^{\vee}}(\widetilde{{\mathcal{N}}})$, where the surjection
$K^{G^{\vee}}(\widetilde{{\mathcal{N}}}) \twoheadrightarrow K^{G^{
\vee}}(\widetilde{U})$ is induced by the restriction
$\widetilde{{\mathcal{N}}} \supset \widetilde{U}$. Moreover, we have the
identification of $\widehat{W}$-modules
\begin{equation*}
K^{G^{\vee}}(\widetilde{{\mathcal{N}}}) \simeq {\mathbb{Z}}
\widehat{W} \otimes _{{\mathbb{Z}}W} {\mathbb{Z}}_{\mathrm{sign}},\, [{
\mathcal{O}}(\gamma )] \mapsto T_{\gamma}.
\end{equation*}
In particular,
$[{\mathcal{O}}_{\widetilde{{\mathcal{N}}}}] \in K^{G^{\vee}}(
\widetilde{{\mathcal{N}}})$ identifies with
$1 \in {\mathbb{Z}}\widehat{W} \otimes _{{\mathbb{Z}}W} {\mathbb{Z}}_{
\mathrm{sign}}$, so $W$ acts on
$[{\mathcal{O}}_{\widetilde{{\mathcal{N}}}}]$ via the sign representation.
It remains to compute the action of $Q^{\vee}$ on
$[{\mathcal{O}}_{\widetilde{U}}]$. In other words we need to compute the
action of the elements $t_{\alpha ^{\vee}_{i}}$, $i=1,\ldots ,r$ on
$[{\mathcal{O}}_{\widetilde{U}}]$. By the definitions and Lemma~\ref{short_exact_div_on_tilde_N} we have
\begin{equation*}
t_{\alpha ^{\vee}_{i}} \cdot [{\mathcal{O}}_{\widetilde{U}}]=[{
\mathcal{O}}_{\widetilde{U}}(\alpha _{i}^{\vee})]=[{\mathcal{O}}_{
\widetilde{U}}]-\iota _{*}o^{0}_{i}=[{\mathcal{O}}_{\widetilde{U}}]-
\iota _{*}o_{i}^{-1}-\iota _{*}[{\mathbb{C}}_{p}].
\end{equation*}
It remains to recall that the action of $t_{\alpha _{i}^{\vee}}$ on
$d$ is given by
\begin{equation*}
t_{\alpha _{i}^{\vee}}(d)=d+\alpha _{i}^{\vee}-K.
\end{equation*}
The isomorphism of $\widehat{W}$-modules
$K^{G^{\vee}}(\widetilde{U}) \simeq \widehat{\mathfrak{h}}_{{
\mathbb{Z}}} \otimes {\mathbb{Z}}_{\mathrm{sign}}$ follows.

The claim about the canonical basis follows from Proposition~\ref{descr_canon}. The fact that the element $C_{\nu _{i}}$ is equal to
$-\alpha _{i}^{\vee}$ but not $-\alpha _{j}^{\vee}$ for some other
$j$ can be easily seen from the equality (use (\ref{nu_i_expl})):
\begin{equation*}
t_{\nu _{i}}(d)=d+\nu _{i}-K=d-\alpha _{0}^{\vee}-\alpha ^{\vee}_{j_{1}}-
\cdots -\alpha ^{\vee}_{i}
\end{equation*}
together with Remark~\ref{normal_coeff}.
\end{proof}

\begin{Rem}
One can avoid the use of Proposition~\ref{descr_canon} in the proof of
Proposition~\ref{ident_K_U_DE} and instead use results of
\cite{lu_subreg} on the canonical basis in
$K^{G^{\vee}}({\cal B}_{e})$ together with \cite[Theorem 5.3.5]{bm}. Recall
that the canonical basis from \cite{lu_subreg} differs from the one in
Proposition~\ref{descr_canon} by applying Grothendieck-Serre duality to
it.
\end{Rem}

\subsection{Computation of ${\bf{m}}^{w_{\gamma}}_{w_{i}}$ and the proof of Theorem~\ref{main_th_formulation}}
\label{comp_KL_values_DE}

Recall now that by Equation (\ref{decomp_that_we_use_VTEX1}) we have
\begin{equation*}
T_{\gamma}=\sum {\bf{m}}^{w_{\gamma}}_{w_{\nu}} C_{\nu}.
\end{equation*}
Taking the image of this equality in
$K^{G^{\vee}}(\widetilde{U}) \simeq \widehat{\mathfrak{h}}_{{
\mathbb{Z}}} \otimes {\mathbb{Z}}_{\mathrm{sign}}$ and using Proposition~\ref{ident_K_U_DE}, we see that
\begin{equation*}
t_{\gamma}(d)=d-\sum _{i=0,1,\ldots ,r} {\bf{m}}^{w_{\gamma}}_{w_{i}}
\alpha _{i}^{\vee}.
\end{equation*}

So we can compute ${\bf{m}}^{w_{\gamma}}_{w_{i}}$. Indeed, this is the
coefficient in front of $-\alpha _{i}^{\vee}$ in
\begin{equation*}
t_{\gamma}(d)-d=\gamma -\frac{1}{2}|\gamma |^{2}K.
\end{equation*}
We conclude that
%
\begin{equation}
\label{comp_m}
{\bf{m}}^{w_{\gamma}}_{{w_{i}}}=\biggl\langle -\Lambda _{i}, \gamma -
\frac{1}{2}|\gamma |^{2}K \biggr\rangle =
\biggl\langle \Lambda _{i},-\gamma +
\frac{|\gamma |^{2}}{2}K \biggr\rangle .
\end{equation}
We are now ready to prove Theorem~\ref{main_th_formulation}.
\begin{proof}[Proof of Theorem~\ref{main_th_formulation}]
Let $i$, $\lambda $, $\Lambda $ be as in Theorem~\ref{main_th_formulation}.
Combining (\ref{our_kl_form_convenient}) and (\ref{comp_m}), we conclude that
\begin{align*}
\widehat{R} \operatorname{ch}L(\Lambda )
&=\sum _{\gamma \in Q^{\vee}}
\sum _{u \in W}\varepsilon (uw_{i}) {\bf{m}}^{w_{-\gamma}}_{w_{i}} e^{ut_{
\gamma}(\lambda +\widehat{\rho})}
\\
&=\sum _{\gamma \in Q^{\vee}}\sum _{u \in W}\varepsilon (uw_{i})
\biggl\langle \Lambda _{i},\gamma +\frac{|\gamma |^{2}}{2}K\biggr\rangle e^{ut_{\gamma}w_{i}(\Lambda +\widehat{\rho})},
\end{align*}
which is precisely the statement of Theorem~\ref{main_th_formulation}.
\end{proof}

\section{The subregular type $A$ case}
\label{sect_subreg_A}

In this section we assume that $\mathfrak{g}=\mathfrak{sl}_{n}$ for some
$n \in {\mathbb{Z}}_{\geqslant 3}$. Recall that our goal is to describe
the $\widehat{W}$-module $K^{\operatorname{PGL}_{n}}(\widetilde{U})$ and
the canonical basis in it.

\subsection{Structure of $\widehat{W}$, its extended version $\widehat{W}^{\mathrm{ext}}$}

Recall that $Q^{\vee}$ is the cocharacter lattice of
$\operatorname{SL}_{n}$ that is equal to the character lattice of
$\operatorname{PGL}_{n}$. We can identify $Q^{\vee}$ with the following
sublattice of ${\mathbb{Z}}^{n}$:
\begin{equation*}
Q^{\vee}=\{(a_{1},\ldots ,a_{n}) \in {\mathbb{Z}}^{n}\mid  a_{1}+ \cdots +a_{n}=0\}.
\end{equation*}
Recall that $\epsilon _{1},\ldots ,\epsilon _{n}$ is the standard basis
of ${\mathbb{Z}}^{n}$. Simple coroots
$\alpha _{i}^{\vee} \in Q^{\vee}$ are
$ \alpha _{i}^{\vee}=\epsilon _{i}-\epsilon _{i+1}$,
$i=1,\ldots ,n-1$, and the Weyl group of $\mathfrak{sl}_{n}$ is
$S_{n}$. The action of $S_{n}$ on $Q^{\vee}$ is the standard action via
permutations. Recall that
\begin{equation*}
\widehat{W} = Q^{\vee} \rtimes S_{n}.
\end{equation*}

The group $\widehat{W}$ can be described as follows. It is the group of
all permutations
$\sigma \colon {\mathbb{Z}}\rightarrow {\mathbb{Z}}$ such that
$\sigma (i+n)=\sigma (i)+n$ and $\sum _{i=1}^{n}(\sigma (i)-i)=0$. The
group of all permutations
$\sigma \colon {\mathbb{Z}}\rightarrow {\mathbb{Z}}$ such that
$\sigma (i+n)=\sigma (i)+n$ is isomorphic to
$S_{n} \ltimes {\mathbb{Z}}^{n}$.

Let us recall the description of the elements
$s_{0},s_{1},\ldots ,s_{n-1}$, $t_{\gamma}$ in these terms. Element
$s_{i}$ is given by
%
\begin{equation}
s_{i}(j)=
\begin{cases}
j+1&\text{for}~j \equiv i ~(\operatorname{mod}~n),
\\
j-1&\text{for}~j \equiv i+1 ~(\operatorname{mod}~n),
\\
j&\text{otherwise}.
\end{cases}
\end{equation}

For a lattice element
$\gamma =(a_{0},\ldots ,a_{n-1}) \in {\mathbb{Z}}^{n}$ the corresponding
element $t_{\gamma}$ of ${\mathbb{Z}}^{n} \rtimes S_{n}$ is given by
\begin{equation*}
t_{\gamma}(k)=k+a_{[k]}n,
\end{equation*}
where $[k] \in {\mathbb{Z}}/n{\mathbb{Z}}\simeq \{0,1,\ldots ,n-1\}$ is
the class of $k$ modulo $n$. So the element $t_{\epsilon _{k}}$ is given
by
\begin{equation*}
t_{\epsilon _{k}}(j)=
\begin{cases}
j+n&\text{for}~j \equiv k~(\text{mod}~n),
\\
j&\text{otherwise}.
\end{cases}
\end{equation*}

Let $P^{\vee}$ be the cocharacter lattice of
$\operatorname{PGL}_{n}$. We can identify $P^{\vee}$ with the following
quotient of ${\mathbb{Z}}^{n}$:
\begin{equation*}
P^{\vee}={\mathbb{Z}}^{\oplus n}/{\mathbb{Z}}(\epsilon _{1}+\cdots +
\epsilon _{n}).
\end{equation*}
We have a natural embedding $Q^{\vee} \subset P^{\vee}$. It will be useful
to consider the the {\textit{extended}} affine Weyl group
\begin{equation*}
\widehat{W}^{\mathrm{ext}}:= P^{\vee} \rtimes S_{n}.
\end{equation*}
For $\gamma \in {\mathbb{Z}}^{n}$ we denote by
$t_{\gamma} \in \widehat{W}^{\mathrm{ext}}$ the corresponding element of
$\widehat{W}^{\mathrm{ext}}$. The group $\widehat{W}^{\mathrm{ext}}$ is
generated by $t_{\epsilon _{1}},\ldots ,t_{\epsilon _{n}}$,
$s_{1},\ldots ,s_{n-1}$.

\subsection{Parametrization of the canonical basis of $K^{\operatorname{PGL}_{n}}(\widetilde{U})$}

Recall that $e \in \mathfrak{sl}_{n}$ is the subregular nilpotent element
and $c \subset \widehat{W}$ is the corresponding two-sided cell. Recall
that the module $K^{\operatorname{PGL}_{n}}(\widetilde{U})$ has a canonical
basis $\bar{C}_{\nu}$ parametrized by $\nu $ such that
$w_{\nu }\in c \cup \{1\}$. It follows from Proposition~\ref{descr_cell_subreg} that the set of $w_{\nu}$ as above is given by

\begin{Lem}%
\label{descr_subreg_cell_type_A}
The elements $w_{\nu }\in c$ can be described as follows. The set of possible
$\nu $ is parametrized by ${\mathbb{Z}}$. For $i \in {\mathbb{Z}}$, the
corresponding element $w_{i}$ is equal to
\begin{equation*}
w_{i}=
\begin{cases}
s_{[i]}s_{[i-1]}\ldots s_{1}s_{0}&\text{for}~i > 0
\\
s_{0}&\text{for}~i=0,
\\
s_{[i]}s_{[i+1]}\ldots s_{[-1]}s_{0}&\text{for}~i < 0
\end{cases}
\end{equation*}
and $\nu _{i}$ is the image of $w_{i}$ in
$Q^{\vee }\simeq \widehat{W}/S_{n}$.
\end{Lem}

\subsection{Description of ${\cal B}_{e}$ and ${\mathbb{C}}^{\times}$-equivariant line bundles on it}
\label{subsec_str_B_e_A}

Let us now recall the explicit description of $\mathcal{B}_{e}$.

The element $e$ can be described as follows. Let
$V={\mathbb{C}}^{n}$ be the standard representation of
$\mathfrak{sl}_{n}$ and let $v_{1},\ldots ,v_{n}$ be the standard basis
of ${\mathbb{C}}^{n}$. Then the element $e$ is:
\begin{equation*}
e(v_{n-1})=e(v_{n})=0,\, e(v_{i})=v_{i+1},\, i=1,\ldots ,n-2.
\end{equation*}
For every (ordered) basis $b_{1},\ldots ,b_{n}$ of $V$ let
${\mathcal{F}}(b_{1},\ldots ,b_{n})$ be the flag
\begin{equation*}
\{0\} \subset \operatorname{Span}_{{\mathbb{C}}}(b_{1}) \subset
\operatorname{Span}_{{\mathbb{C}}}(b_{1},b_{2}) \subset \cdots
\subset \operatorname{Span}_{{\mathbb{C}}}(b_{1},\ldots ,b_{n-1})
\subset V.
\end{equation*}
To every
$(k,a) \in (\{1,\ldots ,n-1\} \times {\mathbb{C}}) \cup \{(0,0)\}$ we associate
the flag
\begin{equation*}
{\mathcal{F}}_{k,a}={\mathcal{F}}(v_{n-1},v_{n-2},\ldots ,v_{n-k+1},v_{n-k}+av_{n},v_{n},v_{n-k-1},v_{n-k-2},
\ldots ,v_{1}).
\end{equation*}
The irreducible components of $\mathcal{B}_{e}$ are parametrized by
$k \in \{1,2,\ldots ,n-1\}$:
\begin{equation*}
\Pi _{k}=\{{\mathcal{F}}_{k,a}\mid  a \in {\mathbb{C}}\} \cup \{{
\mathcal{F}}_{k+1,0}\}.
\end{equation*}

For $k=1,2,\ldots ,n$ we set
\begin{equation*}
p_{k-1,k}:={\mathcal{F}}_{k-1,0}={\mathcal{F}}(v_{n-1},\ldots ,v_{n-k+1},v_{n},v_{n-k},v_{n-k-1},
\ldots ,v_{1}).
\end{equation*}
For $k=2,\ldots ,n-1$ we have
$p_{k-1,k}={\mathcal{F}}_{n-k,0}=\Pi _{k} \cap \Pi _{k-1}$.

Our first goal is to describe the action of $\widehat{W}$ on
$K^{\operatorname{PGL}_{n}}(\widetilde{{\mathbb{O}}}_{e})$. It will be
more convenient to describe the action of
$\widehat{W}^{\mathrm{ext}}$ on
$K^{\operatorname{SL}_{n}}(\widetilde{{\mathbb{O}}}_{e})$ first. Recall that
\begin{equation*}
K^{\operatorname{PGL}_{n}}(\widetilde{{\mathbb{O}}}_{e})=K^{Z_{e,
\operatorname{PGL}_{n}}}({\cal B}_{e}),\, K^{\operatorname{SL}_{n}}(
\widetilde{{\mathbb{O}}}_{e})=K^{Z_{e,\operatorname{SL}_{n}}}({\cal B}_{e}),
\end{equation*}
where $Z_{e,\operatorname{PGL}_{n}} \subset \operatorname{PGL}_{n}$,
$Z_{e,\operatorname{SL}_{n}} \subset \operatorname{SL}_{n}$ are reductive
parts of the centralizers of $e$.

We have the identifications (compare the second identification with
\cite[Section 5.1]{lu_notes_aff})
\begin{gather*}
{\mathbb{C}}^{\times }\iso Z_{e,\operatorname{PGL}_{n}} \subset
\operatorname{PGL}_{n},\, t \mapsto \operatorname{diag}(1,1,\ldots ,1,t),%
\\
{\mathbb{C}}^{\times }\iso Z_{e,\operatorname{SL}_{n}} \subset
\operatorname{SL}_{n},\, t \mapsto \operatorname{diag}(t^{-1},t^{-1},
\ldots ,t^{-1},t^{n-1}).%
\end{gather*}

We obtain two actions of ${\mathbb{C}}^{\times}$ on
$\mathcal{B}_{e}$. The first action sends a flag
${\mathcal{F}}_{k,a}$ to the flag ${\mathcal{F}}_{k,ta}$ and the second
action sends ${\mathcal{F}}_{k,a}$ to the flag
${\mathcal{F}}_{k,t^{n}a}$.

Fix $1 \leqslant k \leqslant n-1$ and consider the following action of
${\mathbb{C}}^{\times}$ on ${\mathbb{P}}^{1}$.
%
\begin{equation}
\label{act_P_1}
t \cdot [x:y]=[t^{n}x:y].
\end{equation}
The following lemma is straightforward (compare with
\cite[Section 5.4]{lu_notes_aff}, \cite[Section 3.6]{gr}).

\begin{Lem}
Let ${\mathbb{C}}^{\times}$ act on ${\mathbb{P}}^{1}$ via (\ref{act_P_1}).
For every collection of integers $i,j$, satisfying $i-j=nm$ for some
$m \in {\mathbb{Z}}$, there exists the unique
${\mathbb{C}}^{\times}$-equivariant line bundle on
${\mathbb{P}}^{1}$ such that $t \in {\mathbb{C}}^{\times}$ acts via
$t^{j}$ at the fiber over $[1:0]$ and acts via $t^{i}$ at the fiber over
$[0:1]$. Every ${\mathbb{C}}^{\times}$-equivariant line bundle on
${\mathbb{P}}^{1}$ can be obtained in this way. We denote the line bundle
above by ${\mathcal{O}}^{j,i}$.
\end{Lem}

\begin{Rem}
The Euler characteristic of ${\mathcal{O}}^{j,i}$ is equal to $m+1$, so
${\mathcal{O}}^{j,i}$ is isomorphic to
${\mathcal{O}}_{{\mathbb{P}}^{1}}(m)$ as a line bundle.
\end{Rem}

Recall again that we have the action of
${\mathbb{C}}^{\times }\simeq Z_{e,\operatorname{SL}_{n}}$ on
$\mathcal{B}_{e}$, which acts on every $\Pi _{k}$ via (\ref{act_P_1}).
We identify
$K^{Z_{e,\operatorname{SL}_{n}}}(\operatorname{pt})={\mathbb{Z}}[\xi ^{
\pm 1}]$,
$K^{Z_{e,\operatorname{PGL}_{n}}}={\mathbb{Z}}[\xi ^{\pm n}]$. For
$k \in \{1,\ldots ,n-1\}$ and $a,b \in {\mathbb{Z}}$ such that
$a-b \in n{\mathbb{Z}}$ we denote by $O_{k}^{b,a}$ the line bundle on
$\Pi _{k}$ whose fiber over $p_{k-1,k}$ is ${\xi}^{b}$ and the fiber over
$p_{k,k+1}$ is ${\xi}^{a}$.

\begin{Lem}%
\label{eq_O_Euler}
For every $a,\, b \in {\mathbb{Z}}$ such that
$a-b \in n{\mathbb{Z}}$ we have
\begin{equation*}
O^{b,a}_{k}+O_{k}^{a,b}=({\xi}^{a}+{\xi}^{b})O^{0,0}_{k}.
\end{equation*}
\end{Lem}
\begin{proof}
Use the Euler sequence for $\Pi _{k} \simeq {\mathbb{P}}^{1}$ (see (\ref{seq_p_1_eul})).
\end{proof}

For $k = 1,\ldots ,n-1$ we set
\begin{equation*}
{\bf{O}}_{k}:=[O^{0,-n}_{k}].
\end{equation*}

\begin{Rem}
{\textit{Note that the line bundle $O^{0,-n}_{k}$ has degree $m=-1$, i.e. is
isomorphic to ${\mathcal{O}}_{\Pi _{k}}(-1)$.}}
\end{Rem}

Let ${\mathbb{C}}_{p_{0,1}}$ be the skyscraper sheaf of the point
$p_{0,1} \in {\cal B}_{e}$.

\begin{Lem}
We have
\begin{equation*}
[{\mathcal{O}}_{{\cal B}_{e}}]=[{\mathbb{C}}_{p_{01}}]+\sum _{k=1}^{n-1}
\xi ^{n}{\bf{O}}_{k}.
\end{equation*}
\end{Lem}
\begin{proof}
Recall that $O_{k}^{0,-n}$ is a line bundle whose fiber over
$p_{k-1,k}$ is $1$ and the fiber over $p_{k,k+1}$ is $\xi ^{-n}$. Clearly
we have an exact sequence
\begin{equation*}
0 \rightarrow \xi ^{n}O_{n-1}^{0,-n} \rightarrow {\mathcal{O}}_{{
\cal B}_{e}} \rightarrow {\mathcal{O}}_{\cup _{k=1}^{n-2}\Pi _{k}}
\rightarrow 0.
\end{equation*}
The claim follows by induction.
\end{proof}

We set
\begin{equation*}
{\bf{O}}_{0}:=-[{\mathcal{O}}_{{\cal B}_{e}}]=-[{\mathbb{C}}_{p_{01}}]-
\sum _{k=1}^{n-1}\xi ^{n}{\bf{O}}_{k}.
\end{equation*}

We finally extend ${\bf{O}}_{k}$ to every $k \in {\mathbb{Z}}$ in such a way
that ${\bf{O}}_{k}=\xi ^{n}{\bf{O}}_{k+n}$ for every $k \in {\mathbb{Z}}$.
The set $\{{\bf{O}}_{k}\mid  k \in {\mathbb{Z}}\}$ forms a basis of the
${\mathbb{Z}}$-module $K^{Z_{e,\operatorname{PGL}_{n}}}(\mathcal{B}_{e})$.

\subsection{Modules $\mathfrak{h}_{\infty ,{\mathbb{Z}}}$, $\widehat{\mathfrak{h}}_{\infty ,{\mathbb{Z}}}$ over $S_{n} \ltimes {\mathbb{Z}}^{n}$, $\widehat{W}$}

\subsubsection{Module $\mathfrak{h}_{\infty ,{\mathbb{Z}}}$}

Using the identification of ${\mathbb{Z}}^{n} \rtimes S_{n}$ with the permutations
${\mathbb{Z}}\rightarrow {\mathbb{Z}}$, we obtain the action of
${\mathbb{Z}}^{n} \rtimes S_{n} \curvearrowright {\mathbb{Z}}^{
\oplus {\mathbb{Z}}}$ that sends $\epsilon _{i}$ to
$\epsilon _{\sigma (i)}$. Consider now the action of
${\mathbb{Z}}[\xi ^{\pm n}]$ on ${\mathbb{Z}}^{\oplus {\mathbb{Z}}}$ given
by $\xi ^{n} \cdot \epsilon _{i}=\epsilon _{i-n}$.

\begin{Rem}
Note that ${\mathbb{Z}}^{n} \rtimes S_{n}$-action on
${\mathbb{Z}}^{\oplus {\mathbb{Z}}}$ commutes with the
${\mathbb{Z}}[\xi ^{\pm n}]$-action.
\end{Rem}

Let
$\mathfrak{h}_{\infty ,{\mathbb{Z}}} \subset {\mathbb{Z}}^{\oplus {
\mathbb{Z}}}$ be the submodule, consisting of elements
$(a_{i})_{i \in {\mathbb{Z}}}$ such that
$\sum _{i \in {\mathbb{Z}}}a_{i}=0$. We obtain the action
$\widehat{W} \curvearrowright \mathfrak{h}_{\infty ,{\mathbb{Z}}}$. Module
$\mathfrak{h}_{\infty ,{\mathbb{Z}}}$ has a ${\mathbb{Z}}$-basis
$\{\alpha _{i}^{\vee},\, i \in {\mathbb{Z}}\}$, where
$\alpha _{i}^{\vee}=\epsilon _{i}-\epsilon _{i+1}$.

\begin{Rem}
Recall that $\mathfrak{h}_{\infty ,{\mathbb{Z}}}$ is a $\widehat{W}$-module
over ${\mathbb{Z}}[\xi ^{\pm n}]$. We can consider the quotient
$\mathfrak{h}_{\infty ,{\mathbb{Z}}}/(\xi ^{n}-1)\mathfrak{h}_{
\infty ,{\mathbb{Z}}}$. It is easy to see that $\widehat{W}$-module
$\mathfrak{h}_{\infty ,{\mathbb{Z}}}/(\xi ^{n}-1)\mathfrak{h}_{
\infty ,{\mathbb{Z}}}$ is isomorphic to
$\mathfrak{h}_{{\mathbb{Z}}} \oplus {\mathbb{Z}}K$ via
$[\alpha ^{\vee}_{i}] \mapsto \alpha ^{\vee}_{i}$,
$i=0,1,\ldots ,n-1$.
\end{Rem}

\begin{Rem}
Let us define the symmetric bilinear form $(\,,\,)$ on
$\mathfrak{h}_{\infty ,{\mathbb{Z}}}$. Consider the symmetric bilinear
form $(\,,\,)$ on ${\mathbb{Z}}^{\oplus {\mathbb{Z}}}$ given by
$(\epsilon _{i},\epsilon _{j})=\delta _{[i],[j]}$. We denote by
$(\,,\,)$ its restriction to $\mathfrak{h}_{\infty ,{\mathbb{Z}}}$. It
is clear that $(\,,\,)$ is ${\mathbb{Z}}^{n} \rtimes S_{n}$-invariant.
It is also clear that
\begin{equation*}
(\alpha _{i}^{\vee},\alpha _{j}^{\vee})=
\begin{cases}
2&\text{if}~[i]=[j],
\\
-1&\text{if}~[i],[j] \in {\mathbb{Z}}/n{\mathbb{Z}}~\text{are adjacent},
\\
0&\text{otherwise}.
\end{cases}
\end{equation*}
\end{Rem}

\subsubsection{Module $\widehat{\mathfrak{h}}_{\infty ,{\mathbb{Z}}}$}
\label{ext_h_infty}

Recall the $\widehat{W}$-module
$\mathfrak{h}_{\infty ,{\mathbb{Z}}}$. Set
\begin{equation*}
\widehat{\mathfrak{h}}_{\infty ,{\mathbb{Z}}}:=\mathfrak{h}_{\infty ,{
\mathbb{Z}}} \oplus {\mathbb{Z}}d
\end{equation*}
and define the $\widehat{W}$-module structure on it by ($i=1,\ldots ,n-1$)
\begin{equation*}
t_{\alpha _{i}^{\vee}}(d)=d+\epsilon _{i}-\epsilon _{i+1-n},\, w(d)=d,
\, w \in S_{n}.
\end{equation*}

Clearly, we have an exact sequence of $\widehat{W}$-modules
\begin{equation*}
0 \rightarrow \mathfrak{h}_{\infty ,{\mathbb{Z}}} \rightarrow
\widehat{\mathfrak{h}}_{\infty ,{\mathbb{Z}}} \rightarrow {\mathbb{Z}}_{
\mathrm{triv}} \rightarrow 0.
\end{equation*}

\begin{Rem}
Recall the bilinear form $(\,,\,)$ on
$\mathfrak{h}_{\infty ,{\mathbb{Z}}}$. It can be extended to the bilinear
from on $\widehat{\mathfrak{h}}_{\infty ,{\mathbb{Z}}}$ by
$(d,\alpha _{i})=0$ for $i \notin n{\mathbb{Z}}$ and
$(d,\alpha _{i})=1$ for $i \in n{\mathbb{Z}}$, $(d,d)=0$. It is easy to
see that this form is $\widehat{W}$-invariant.
\end{Rem}

We can extend the action of $\widehat{W}$ on
$\widehat{\mathfrak{h}}_{\infty ,{\mathbb{Z}}}$ to the action of
${\mathbb{Z}}^{n} \rtimes S_{n}$ on
${\mathbb{Z}}^{\oplus {\mathbb{Z}}} \oplus {\mathbb{Z}}d$ via
\begin{equation*}
t_{\epsilon _{i}}(d)=d+\epsilon _{i}.
\end{equation*}
For $\gamma \in {\mathbb{Z}}^{n}$ let us describe explicitly the element
$t_{\gamma}(d)$. For $k \in {\mathbb{Z}}$ we set
\begin{equation*}
\overline{\Lambda }^{\infty}_{k}:=\sum _{l \leqslant k}\epsilon _{l}^{*}
\in \mathfrak{h}_{\infty}^{*}.
\end{equation*}
Clearly
\begin{equation*}
\langle \overline{\Lambda }^{\infty}_{i}, \alpha ^{\vee}_{k} \rangle =
\delta _{i,k}.
\end{equation*}

For $a \in {\mathbb{Z}}_{>0}$ we have
\begin{equation*}
t_{a\epsilon _{k}}(d)=d+\epsilon _{k}+\cdots +\epsilon _{k+(a-1)n},
\end{equation*}
so
\begin{equation*}
\langle \overline{\Lambda }^{\infty}_{i},t_{a\epsilon _{k}}(d)
\rangle =|{\mathbb{Z}}_{\leqslant i} \cap [k,k+{(a-1)n}] \cap (k+n{
\mathbb{Z}})|.
\end{equation*}

For $a < 0$ we have
\begin{equation*}
t_{a\epsilon _{k}}(d)=d-\epsilon _{k-n}-\cdots -\epsilon _{k+an},
\end{equation*}
so
\begin{equation*}
\langle \overline{\Lambda }^{\infty}_{i},t_{a\epsilon _{k}}(d)
\rangle =-|{\mathbb{Z}}_{\leqslant i} \cap [k+an,k-n] \cap (k+n{
\mathbb{Z}})|.
\end{equation*}

For $i \in {\mathbb{Z}}$, $k=1,\ldots ,n$, and $a \in {\mathbb{Z}}$ we
set
%
\begin{equation}
\label{z_def!}
z_{i}(a\epsilon _{k}):=
\begin{cases}
|{\mathbb{Z}}_{\leqslant i} \cap [k,k+{(a-1)n}] \cap (k+n{\mathbb{Z}})|&
\text{for}~ a \in {\mathbb{Z}}_{\geqslant 0},
\\
-|{\mathbb{Z}}_{\leqslant i} \cap [k+an,k-n] \cap (k+n{\mathbb{Z}})|&
\text{for}~a \in {\mathbb{Z}}_{\leqslant 0}.
\end{cases}
\end{equation}

We conclude that for
$\gamma =a_{1}\epsilon _{1}+a_{2}\epsilon _{2}+\cdots +a_{n}\epsilon _{n}
\in Q^{\vee}$ we have
\begin{equation*}
t_{\gamma}(d)=d+\sum _{i \in {\mathbb{Z}}}\Big(\sum _{k=1}^{n}z_{i}(a_{k}
\epsilon _{k})\Big)\alpha _{i}^{\vee}.
\end{equation*}

\subsection{Structure of modules $K^{Z_{e,\operatorname{SL}_{n}}}({\cal B}_{e})$, $K^{Z_{e,\operatorname{PGL}_{n}}}({\mathcal{B}}_{e})$}
\label{type_A_fib_module}

The main reference for this section is \cite{lu_notes_aff}. The goal of
this section is to construct an isomorphism of $\widehat{W}$-modules
$\mathfrak{h}_{\infty ,{\mathbb{Z}}} \simeq K^{Z_{e,
\operatorname{PGL}_{n}}}({\cal B}_{e})$. We describe the action of
$\widehat{W}^{\mathrm{ext}}$ on
$K^{Z_{e,\operatorname{SL}_{n}}}({\cal B}_{e})$ first.

The set
$\{{\bf{O}}_{k}\}_{k=1,\ldots ,n-1} \cup \{[{\mathbb{C}}_{p_{0,1}}]\}$
forms a basis of the ${\mathbb{Z}}[{\xi}^{\pm 1}]$-module
$K^{Z_{\operatorname{SL}_{n}}(e)}(\mathcal{B}_{e})$. The following lemma
holds by \cite[Section 5.11]{lu_notes_aff}.
%
\begin{Lem}
We have ($k=1,2,\ldots ,n-1$)
\begin{gather*}
s_{l}({\bf{O}}_{l-1})=-{\bf{O}}_{l-1}-{\bf{O}}_{l}~\text{for}~l=2,
\ldots ,n-1,%
\\
s_{l-1}({{\bf{O}}_{l}})=-{\bf{O}}_{l}-{\bf{O}}_{l-1}~\text{for}~l=2,
\ldots ,n-1,%
\\
s_{k}({\bf{O}}_{k})={\bf{O}}_{k} ,%
\\
s_{l}({\bf{O}}_{k})=-{\bf{O}}_{k}~\text{for}~l\neq k-1,k,k+1,%
\\
s_{l}([{\mathbb{C}}_{p_{0,1}}])=-[{\mathbb{C}}_{p_{0,1}}]~\text{for}~l=2,
\ldots ,n-1,%
\\
s_{1}([{\mathbb{C}}_{p_{0,1}}])=-[{\mathbb{C}}_{p_{0,1}}]+(1-{\xi}^{n}){
\bf{O}}_{1}.%
\end{gather*}
\end{Lem}

Let $L_{k-1,k}$ be the line bundle on ${\cal B}_{e}$ whose fiber over
$(\{0\} \subset D_{1} \subset \cdots \subset D_{n-1} \subset V)$ is
$D_{k}/D_{k-1}$ (see \cite[Section 5.1]{lu_notes_aff}). By definition,
the action of $t_{-\epsilon _{k}}$ on
$K^{Z_{\operatorname{SL}_{n}}(e)}({\cal B}_{e})$ is given by the tensor
product with the line bundle $L_{k-1,k}$. The fiber of $L_{k-1,k}$ over
$p_{l-1,l}$ is equal to ${\xi}^{-1}$ if $k \neq l$ and ${\xi}^{n-1}$ if
$k=l$. We conclude that
%
\begin{gather}
\label{eq_1!}
t_{-\epsilon _{k}}( [O^{b,a}_{l}])=[O^{b-1,a-1}_{l}]={\xi}^{-1}[O_{l}^{b,a}]~
\text{if}~l \neq k-1, k,%
\\
\label{eq_2!}
t_{-\epsilon _{k}}([O^{b,a}_{k-1}])= [O^{b-1,a+n-1}_{k-1}],%
\\
\label{eq_3!}
t_{-\epsilon _{k}}([O^{b,a}_{k}])=[O^{b+n-1,a-1}_{k}].%
\end{gather}

\begin{Lem}%
\label{id_sky_vs_str}
We have
\begin{gather*}
[{\mathbb{C}}_{p_{k-1,k}}]=[O^{0,0}_{k}]-[O^{n,0}_{k}],\, [{
\mathbb{C}}_{p_{k,k+1}}]=[O^{0,0}_{k}]-[O^{0,-n}_{k}],%
\\
[{\mathbb{C}}_{p_{k-1,k}}]=[{\mathbb{C}}_{p_{k,k+1}}]+(1-{\xi}^{n}){
\bf{O}}_{k},%
\\
[{\mathbb{C}}_{p_{01}}]=[{\mathbb{C}}_{p_{k-1,k}}]+\sum _{l=1}^{k-1}(1-{
\xi}^{n}){\bf{O}}_{l}.%
\end{gather*}
\end{Lem}
\begin{proof}
All of these equalities directly follow from the standard exact sequences
on ${\mathbb{P}}^{1}$.
\end{proof}

\begin{Cor}\label{O_00_via_p}
We have
\begin{equation*}
[O_{k}^{0,0}]=[{\mathbb{C}}_{p_{01}}]+{\bf{O}}_{k}+\sum _{l=1}^{k}({
\xi}^{n}-1){\bf{O}}_{l}=-\sum _{l=0}^{k-1}{\bf{O}}_{l}-\xi ^{n}\sum _{l=k+1}^{n-1}{
\bf{O}}_{l}.
\end{equation*}
\end{Cor}
\begin{proof}
By Lemma~\ref{id_sky_vs_str}
\begin{align*}
[O_{k}^{0,0}]
&=[{\mathbb{C}}_{p_{k,k+1}}]+[O_{k}^{0,-n}]=[{\mathbb{C}}_{p_{k,k+1}}]+{
\bf{O}}_{k}
\\
&=[{\mathbb{C}}_{p_{01}}]+{\bf{O}}_{k}+\sum _{l=1}^{k}({\xi}^{n}-1){
\bf{O}}_{l}=-\sum _{l=0}^{k-1}{\bf{O}}_{l}-\xi ^{n}\sum _{l=k+1}^{n-1}{
\bf{O}}_{l}.
\end{align*}\vskip-20pt
\end{proof}

\begin{Lem}
We have
\begin{gather*}
t_{-\epsilon _{k}} ([{\mathbb{C}}_{p_{0,1}}])={\xi}^{-1}[{\mathbb{C}}_{p_{0,1}}],
\, k=2,\ldots ,n-1,%
\\
t_{-\epsilon _{1}}([{\mathbb{C}}_{p_{0,1}}])={\xi}^{n-1}[{\mathbb{C}}_{p_{0,1}}],%
\\
t_{-\epsilon _{k}}({\bf{O}}_{l})={\xi}^{-1}{\bf{O}}_{l}~\text{if}~l
\neq k-1,k,%
\\
t_{-\epsilon _{k}}({\bf{O}}_{k-1})={\xi}^{-1}([{\mathbb{C}}_{p_{k-1,k}}]+{
\bf{O}}_{k-1})=-\xi ^{-1}\Big(\sum _{l=0}^{k-2}{\bf{O}}_{l}+\xi ^{n}
\sum _{l=k}^{n-1}{\bf{O}}_{l}\Big),%
\\
t_{-\epsilon _{k}}({\bf{O}}_{k})= {\xi}^{-1}(-[{\mathbb{C}}_{p_{k-1,k}}]+{
\bf{O}}_{k})=\xi ^{-1}\Big(\sum _{l=0}^{k} {\bf{O}}_{l} + \xi ^{n}
\sum _{l=k}^{n-1}{\bf{O}}_{l}\Big).%
\end{gather*}
\end{Lem}

\begin{proof}
The first three equalities follow from the fact that the fiber of
$L_{k-1,k}$ over $p_{l-1,l}$ is $\xi ^{-1}$ if $k \neq l$ and
$\xi ^{n-1}$ if $k=l$.

Since $[O_{k-1}^{-1,-1}]=\xi ^{-1}[O^{0,0}_{k-1}]$ and using Corollary~\ref{O_00_via_p}, we conclude that
\begin{equation*}
[O_{k-1}^{-1,-1}]={\xi}^{-1}[O^{0,0}_{k-1}]=-\xi ^{-1}\sum _{l=0}^{k-2}{
\bf{O}}_{l}-\xi ^{n-1}\sum _{l=k}^{n-1}{\bf{O}}_{l}.
\end{equation*}

We see that
\begin{equation*}
t_{-\epsilon _{k}}([O_{k-1}^{0,-n}])=[O_{k-1}^{-1,-1}]=-\xi ^{-1}
\Big(\sum _{l=0}^{k-2}{\bf{O}}_{l}+\xi ^{n}\sum _{l=k}^{n-1}{\bf{O}}_{l}
\Big).
\end{equation*}

By Lemma~\ref{eq_O_Euler}
$[O^{0,-n}_{k}]+[O_{k}^{-n,0}]=(1+{\xi}^{-n})[O_{k}^{0,0}]$, so
\begin{equation*}
[O_{k}^{0,-n}]= -[O_{k}^{-n,0}]+(1+{\xi}^{-n})[O_{k}^{0,0}].
\end{equation*}
Using (\ref{eq_3!}) we see that
\begin{align*}
t_{-\epsilon _{k}}([O_{k}^{0,-n}])
&=t_{-\epsilon _{k}}(-[O_{k}^{-n,0}]+(1+\xi ^{-n})[O_{k}^{0,0}])
\\
&=\xi ^{-1}(-[O_{k}^{0,0}]+(1+\xi ^{-n})[O_{k}^{n,0}])
\\
&=\xi ^{-1}({\bf{O}}_{k}+\xi ^{n}{\bf{O}}_{k}-[O_{k}^{0,0}])=\xi ^{-1}
\Big(\sum _{l=0}^{k}{\bf{O}}_{l}+\xi ^{n}\sum _{l=k}^{n-1}{\bf{O}}_{l}
\Big).
\end{align*}\vskip-20pt
\end{proof}

Combining all the relations, we get the $\widehat{W}^{\mathrm{ext}}$-representation
$K^{Z_{e,\operatorname{SL}_{n}}}({\cal B}_{e})$ with
${\mathbb{Z}}[\xi ^{\pm 1}]$-basis
${\bf{O}}_{1},\ldots ,{\bf{O}}_{n-1}, [{\mathbb{C}}_{p_{0,1}}]$ and the
following action of $\widehat{W}^{\mathrm{ext}}$ (here
$k \in \{1,\ldots ,n-1\}$):
\begin{gather*}
s_{k}({\bf{O}}_{k-1})=-{\bf{O}}_{k-1}-{\bf{O}}_{k},%
\\
s_{k-1}{\bf{O}}_{k}=-{\bf{O}}_{k}-{\bf{O}}_{k-1},%
\\
s_{k}({\bf{O}}_{k})={\bf{O}}_{k} ,%
\\
s_{l}({\bf{O}}_{k})=-{\bf{O}}_{k}~\text{for}~l\neq k-1,k,k+1,%
\\
s_{k}([{\mathbb{C}}_{p_{0,1}}])=-[{\mathbb{C}}_{p_{0,1}}],\, k=2,
\ldots ,n-1,%
\\
s_{1}([{\mathbb{C}}_{p_{0,1}}])=-[{\mathbb{C}}_{p_{0,1}}]+(1-{\xi}^{n}){
\bf{O}}_{1},%
\\
t_{-\epsilon _{k}}([{\mathbb{C}}_{p_{0,1}}])={\xi}^{-1}[{\mathbb{C}}_{p_{0,1}}],
\, k=2,\ldots ,n-1,%
\\
t_{-\epsilon _{1}}([{\mathbb{C}}_{p_{0,1}}])={\xi}^{n-1}[{\mathbb{C}}_{p_{0,1}}],%
\\
t_{-\epsilon _{k}}({\bf{O}}_{l})={\xi}^{-1}{\bf{O}}_{l}~\text{if}~l
\neq k-1,k,%
\\
t_{-\epsilon _{k}}({\bf{O}}_{k-1})=-\xi ^{-1}\Big(\sum _{l=0}^{k-2}{
\bf{O}}_{l}+\xi ^{n}\sum _{l=k}^{n-1}{\bf{O}}_{l}\Big)=-\xi ^{-1}
\sum _{l=k-n}^{k-2}{\bf{O}}_{l},%
\\
t_{-\epsilon _{k}}({\bf{O}}_{k})=\xi ^{-1}\Big(\sum _{l=0}^{k} {\bf{O}}_{l}
+ \xi ^{n}\sum _{l=k}^{n-1}{\bf{O}}_{l}\Big)=\xi ^{-1}\sum _{l=k-n}^{k}{
\bf{O}}_{l}.%
\end{gather*}

\begin{Prop}%
\label{iso_B_e_mod_type_A}
There is an isomorphism of $\widehat{W}$-modules over
${\mathbb{Z}}[\xi ^{\pm n}]$:
\begin{equation*}
\mathfrak{h}_{\infty ,{\mathbb{Z}}} \otimes {\mathbb{Z}}_{
\mathrm{sign}} \simeq K^{Z_{e,\operatorname{PGL}_{n}}}({\cal B}_{e}),
\end{equation*}
given by
\begin{equation*}
\alpha ^{\vee}_{k} \otimes 1 \mapsto -{\bf{O}}_{k},\, k \in {
\mathbb{Z}},\, [{\mathbb{C}}_{p_{0,1}}] \mapsto \alpha _{1-n}^{\vee}+ \cdots +\alpha _{-1}^{\vee}+\alpha _{0}^{\vee}=\epsilon _{1-n}-
\epsilon _{1}.
\end{equation*}
\end{Prop}
\begin{proof}
Directly follows from the formulas for the action of
$\widehat{W}^{\mathrm{ext}}$ on
$K^{Z_{e,\operatorname{SL}_{n}}}({\cal B}_{e})$ above together with the
fact that
$t_{\alpha _{k}^{\vee}}=t_{\epsilon _{k}} \circ t_{-\epsilon _{k+1}}$.
\end{proof}

\subsection{Structure of $K^{\operatorname{PGL}_{n}}(\widetilde{U})$}

Recall that by Theorem~\ref{cells} we have an exact sequence of
$\widehat{W}$-modules
\begin{equation*}
0 \rightarrow K^{Z_{e,\operatorname{PGL}_{n}}}(\mathcal{B}_{e})
\rightarrow K^{\operatorname{PGL}_{n}}(\widetilde{U}) \rightarrow {
\mathbb{Z}}_{\mathrm{sign}} \rightarrow 0.
\end{equation*}
We have already described $\widehat{W}$-module
$K^{Z_{e,\operatorname{PGL}_{n}}}(\mathcal{B}_{e})$ explicitly. So in order
to describe the action of $\widehat{W}$ on
$K^{\operatorname{PGL}_{n}}(\widetilde{U})$ we just need to compute the
action of $\widehat{W}$ on $[{\mathcal{O}}_{\widetilde{U}}]$. Since
$K^{\operatorname{PGL}_{n}}(\widetilde{U})$ is the quotient of
$K^{\operatorname{PGL}_{n}}(\widetilde{{\mathcal{N}}})$, we have
$w([{\mathcal{O}}_{\widetilde{U}}])=\varepsilon (w)[{\mathcal{O}}_{
\widetilde{U}}]$ for $w \in S_{n}$. It remains to determine the action
of $t_{\alpha _{k}^{\vee}}$, $k=1,\ldots ,n-1$, on
$[{\mathcal{O}}_{\widetilde{U}}]$.

\begin{Lem}%
\label{act_t_al_O}
We have
\begin{equation*}
t_{\alpha _{k}^{\vee}}([{\mathcal{O}}_{\widetilde{U}}])=[{\mathcal{O}}_{
\widetilde{U}}]+\sum _{l=k+1-n}^{k-1}{\bf{O}}_{l}.
\end{equation*}
\end{Lem}
\begin{proof}
Indeed, using Lemma~\ref{short_exact_div_on_tilde_N} and Corollary~\ref{O_00_via_p} we obtain
\begin{align*}
t_{\alpha ^{\vee}_{k}}([{\mathcal{O}}_{\widetilde{U}}])
&=[{\mathcal{O}}_{\widetilde{U}}(\alpha _{k}^{\vee})]=[{\mathcal{O}}_{\widetilde{U}}]-[\iota _{*}O_{k}^{0,0}]
\\
&=[{\mathcal{O}}_{\widetilde{U}}]-[{\mathbb{C}}_{p_{01}}]-{\bf{O}}_{k}+
\sum _{l=1}^{k} (1-\xi ^{n}){\bf{O}}_{l}=[{\mathcal{O}}_{\widetilde{U}}]+\sum _{l=k+1-n}^{k-1}{\bf{O}}_{l}.
\end{align*}
\end{proof}

\begin{Prop}%
\label{descr_K_U_A}
We have an isomorphism of $\widehat{W}$-modules
$K^{\operatorname{PGL}_{n}}(\widetilde{U}) \simeq
\widehat{\mathfrak{h}}_{\infty ,{\mathbb{Z}}} \otimes {\mathbb{Z}}_{
\mathrm{sign}}$. This isomorphism is given by
\begin{equation*}
\bar{C}_{0} \mapsto d \otimes 1,\, \bar{C}_{\nu _{i}} \mapsto -
\alpha _{i}^{\vee},\, i \in {\mathbb{Z}}.
\end{equation*}
\end{Prop}
\begin{proof}
Follows from Proposition~\ref{iso_B_e_mod_type_A}, Lemma~\ref{act_t_al_O} and Proposition~\ref{descr_canon}.
\end{proof}

\begin{Rem}
One can avoid the use of Proposition~\ref{descr_canon} in the proof of
Proposition~\ref{descr_K_U_A} and instead use results of
\cite{lu_notes_aff} on the canonical basis in
$K^{G^{\vee}}({\cal B}_{e})$ together with \cite[Theorem 5.3.5]{bm}, see
also \cite{gr}.
\end{Rem}

\subsection{Computation of ${\bf{m}}^{w_{\gamma}}_{w_{i}}$ and the proof of Theorem~\ref{main_th_formulation_A}}
\label{m_th_A}

Recall now that
\begin{equation*}
t_{\gamma }\cdot 1=T_{\gamma}=\sum _{\nu }{\bf{m}}^{w_{\gamma}}_{w_{
\nu}} C_{\nu}.
\end{equation*}

Taking the image of this equality in
$K^{\operatorname{PGL}_{n}}(\widetilde{U}) \simeq
\widehat{\mathfrak{h}}_{\infty ,{\mathbb{Z}}} \otimes {\mathbb{Z}}_{
\mathrm{sign}}$ and using Proposition~\ref{descr_K_U_A}, we see that
\begin{equation*}
t_{\gamma}(d)=d-\sum _{i \in {\mathbb{Z}}} {\bf{m}}^{w_{\gamma}}_{w_{i}}
\alpha _{i}^{\vee}.
\end{equation*}

So we can compute ${\bf{m}}^{w_{\gamma}}_{w_{i}}$. Indeed, this is just
the coefficient in front of $-\alpha _{i}^{\vee}$ in
\begin{equation*}
t_{\gamma}(d)-d=\sum _{i \in {\mathbb{Z}}}\Big(\sum _{k=1}^{n}z_{i}(
\langle \epsilon _{k},\gamma \rangle \epsilon _{k})\Big)\alpha _{i}^{
\vee}.
\end{equation*}
We conclude that
%
\begin{equation}
\label{comp_m_A}
{\bf{m}}^{w_{\gamma}}_{{w_{i}}}= -\sum _{k=1}^{n}z_{i}(\langle
\epsilon _{k},\gamma \rangle \epsilon _{k}).
\end{equation}
We are now ready to prove Theorem~\ref{main_th_formulation_A}.

\begin{proof}[Proof of Theorem~\ref{main_th_formulation_A}]
Let $i$, $\lambda $, $\Lambda $ be as in Theorem~\ref{main_th_formulation}. Combining
(\ref{our_kl_form_convenient}) and (\ref{comp_m_A}), we conclude that
\begin{align*}
\widehat{R} \operatorname{ch}L(\Lambda )
&=\sum _{\gamma \in Q^{\vee}}\sum _{u \in W}\varepsilon (uw_{i}) {\bf{m}}^{w_{-\gamma}}_{w_{i}} e^{ut_{\gamma}(\lambda +\widehat{\rho})}
\\
&=-\sum _{\gamma \in Q^{\vee}}\sum _{u \in W}\varepsilon (uw_{i})
\Big(\sum _{k=1}^{n} z_{i}(-\langle \epsilon _{k},\gamma \rangle
\epsilon _{k})\Big) e^{ut_{\gamma}w_{i}(\Lambda +\widehat{\rho})}
\end{align*}
that is precisely the statement of Theorem~\ref{main_th_formulation_A}.
\end{proof}

\section{Possible generalizations}
\label{poss_gener}

\subsection{Non-simply laced case}

Recall that in this paper we restrict ourselves to the simply laced case.
One can consider arbitrary simple Lie algebra $\mathfrak{g}$. Using an
approach similar to the one in this paper it should be possible to obtain
explicit formulas for characters of certain $\widehat{\mathfrak{g}}$-modules
$L(\Lambda )$ (``corresponding'' to the subregular cell in
$\widehat{W}$). We plan to return to this in the future.\footnote{This will
be done in the joint paper \cite{KS} of the third author and Kenta Suzuki.}
The relevant question here is the explicit description of the
$\widehat{W}$-module $K^{Z_{e}}({\cal B}_{e})$ ($e$ is a subregular nilpotent
of $\mathfrak{g}^{\vee}$). Let us describe the conjectural answer.

\begin{Rem}
Note that the $\widehat{W}$-module $K({\cal B}_{e})$ is described in
\cite[Section 6]{lu_notes_aff}.
\end{Rem}

Consider the affine Lie algebra $\widehat{\mathfrak{g}}$ and the corresponding
(affine) Weyl group $\widehat{W}$ (see Section~\ref{aff_Weyl_grp_g} or
\cite{ka}, \cite[Section 1.6]{lu_monod}). Our goal is to describe the
$\widehat{W}$-module structure on $K^{Z_{e}}({\cal B}_{e})$. The Dynkin
diagram of $(\widehat{\mathfrak{g}})^{\vee}$ can be obtained from a simply
laced affine Dynkin diagram by folding (see for example
\cite[Section 14.1.5]{lu_quant_book}). We denote the simply laced affine
Lie algebra above by $\mathfrak{k}$ and denote by $W(\mathfrak{k})$ its
Weyl group.

\begin{Rem}
Note that $(\widehat{\mathfrak{g}})^{\vee}$ is a {\textit{twisted}} affine Lie
algebra (see Example~\ref{ex_g_to_k}).
\end{Rem}

Assume for simplicity that $\mathfrak{g}$ is $B_{n}$ or $F_{4}$. By
\cite[Corollary 3.3]{lu_subreg_crit} there is an embedding
$\widehat{W} \subset W(\mathfrak{k})$ that sends a simple reflection of
$\widehat{W}$ to the product of simple reflections over the corresponding
orbit of folding. Let
$\widetilde{\mathfrak{t}}_{{\mathbb{Z}}} \subset \mathfrak{k}$ be the (integral
form of the) ``reflection'' representation of $W(\mathfrak{k})$ and
$\widehat{\mathfrak{t}}_{{\mathbb{Z}}}=\widetilde{\mathfrak{t}}
\oplus {\mathbb{Z}}d$ be the ``Cartan'' representation. Using the embedding
$\widehat{W} \subset W(\mathfrak{k})$ we obtain the action of
$\widehat{W}$ on $\widetilde{\mathfrak{t}}_{{\mathbb{Z}}}$,
$\widehat{\mathfrak{t}}_{{\mathbb{Z}}}$. The following conjecture will
be proven (and generalized to other types) in \cite{KS}.

\begin{Conj}
Assume that $\mathfrak{g}$ is $B_{n}$ $(n \geqslant 3)$, $F_{4}$. We have
isomorphisms of $\widehat{W}$-modules
\begin{equation*}
K^{Z_{e}}({\cal B}_{e}) \simeq \widetilde{\mathfrak{t}}_{{\mathbb{Z}}}
\otimes {\mathbb{Z}}_{\mathrm{sign}},\, K^{G^{\vee}}(\widetilde{U})
\simeq \widehat{\mathfrak{t}}_{{\mathbb{Z}}} \otimes {\mathbb{Z}}_{
\mathrm{sign}}.
\end{equation*}
\end{Conj}

\begin{Example}%
\label{ex_g_to_k}
Assume for example that $\mathfrak{g}=B_{n}$. Then
$\widehat{\mathfrak{g}}=\widetilde{B}_{n}$, hence,
$(\widehat{\mathfrak{g}})^{\vee}=A^{(2)}_{2n+1}$. We conclude that
$\mathfrak{k}=\widetilde{D}_{2n}$. For $\mathfrak{g}=F_{4}$ we have
$(\widehat{\mathfrak{g}})^{\vee}=E_{6}^{(2)}$ so
$\mathfrak{k}=\widetilde{E}_{7}$.
\end{Example}

\begin{Rem}
For $\mathfrak{g}=B_{n}, F_{4}$ it is easy to see that the rank of
$K^{Z_{e}}({\cal B}_{e})$ (over ${\mathbb{Z}}$) is equal to the rank of
$\widetilde{\mathfrak{t}}_{{\mathbb{Z}}}$. Indeed, recall that by
\cite[Section 6.2]{sl} variety ${\cal B}_{e}$ can be identified with the
fiber of the Springer resolution over a subregular nilpotent element of
the unfolding of $\mathfrak{g}$. It follows that $K({\cal B}_{e})$ has
a basis, consisting of $[{\mathcal{O}}_{\Pi _{i}}(-1)]$,
$[{\mathcal{O}}_{{\cal B}_{e}}]$, where $i$ runs through the set of simple
roots of the unfolding of $\mathfrak{g}$ (recall that to every {\textit{long}}
simple root $\alpha _{i}$ of $\mathfrak{g}$ corresponds the unique simple
root of the unfolding of $\mathfrak{g}$ and to every {\textit{short}} root
$\alpha _{i}$ correspond two simple roots $\alpha _{i^{\pm}}$ of the unfolding
of $\mathfrak{g}$). Recall also that $Z_{e} \simeq S_{2}$. Using the same
argument as in \cite[Section 1.25]{lu_subreg}, one can show that the basis
of $K^{Z_{e}}({\cal B}_{e})$ can be described as follows: for every long
simple root $\alpha _{i}$ of $\mathfrak{g}$ we consider
$[{\mathcal{O}}_{\Pi _{i}}(-1)^{+}]$,
$[{\mathcal{O}}_{\Pi _{i}}(-1)^{-}]$, where $\pm $ correspond to two different
$Z_{e}$-equivariant structures on ${\mathcal{O}}_{\Pi _{i}}(-1)$, we also
have $[{\mathcal{O}}_{{\cal B}_{e}}^{\pm}]$ and finally for every short
simple root $\alpha _{i}$ (of $\mathfrak{g}$) we have
$[{\mathcal{O}}_{\Pi _{i^{+}}}(-1) \sqcup {\mathcal{O}}_{\Pi _{i^{-}}}(-1)]$
(this is the class supported on $\Pi _{i^{+}} \sqcup \Pi _{i^{-}}$). So
we see that to every {\textit{long}} root of $\widehat{\mathfrak{g}}$ (in particular,
to the affine root) we associate two basis elements of
$K^{Z_{e}}({\cal B}_{e})$ and every {\textit{short}} root corresponds to the
unique basis element of $K^{Z_{e}}({\cal B}_{e})$. It follows that the
rank of $K^{Z_{e}}({\cal B}_{e})$ is equal to the number of vertices of
the unfolding of the Dynkin diagram of
$(\widehat{\mathfrak{g}})^{\vee}$ i.e. the number of simple roots of
$\mathfrak{k}$ (the rank of
$\widetilde{\mathfrak{t}}_{{\mathbb{Z}}}$).
\end{Rem}

\subsection{Case of arbitrary nilpotent $e$}

One can try to use a similar approach to the one in this paper to compute
characters of more general $\widehat{\mathfrak{g}}$-modules
$L(\Lambda )$ such that the level of $\Lambda $ is greater than
$-h^{\vee}$ and $\Lambda +\widehat{\rho}$ is integral quasi-dominant. Let
$w \in \widehat{W}$ be the longest element such that
$w(\lambda +\widehat{\rho})$ is dominant. Let
$c \subset \widehat{W}$ be the two-sided cell that contains $w$. Let
$e \in \mathfrak{g}^{\vee}$ be the corresponding nilpotent element (not
necessarily subregular). For $e' \in {\mathcal{N}}$ we say that {\textit{$e'$
is over $e$}} if $e$ is contained in the closure of the orbit
${\mathbb{O}}_{e'}=G^{\vee }\cdot e'$. Let $U \subset {\mathcal{N}}$ be the union of the
${\mathbb{O}}_{e'}$ such that $e'$ is over $e$; this is an open subset
of ${\mathcal{N}}$. Set $\widetilde{U}:=\pi ^{-1}(U)$.

It follows from the above that the character of $L(\Lambda )$ can be extracted
from the $\widehat{W}$-module $K^{G^{\vee}}(\widetilde{U})$ and the canonical
basis in it. Recall that $\widetilde{U}$ was constructed starting from
a nilpotent element $e \in \mathfrak{g}^{\vee}$.

Recall that an element $e\in {\cal N}$ is called {\textit{distinguished}} if
it is not contained in a proper Levi subalgebra. Apparently the simplest
case to consider is the case when $e$ is {\textit{very distinguished}} i.e.
if every element $e'\in {\cal N}$ over $e$ is distinguished. If this is
the case then the module $K^{G^{\vee}}(\widetilde{U})$ is clearly finite
dimensional. It follows that the function
$\gamma \mapsto {\bf{m}}_{w_{\nu}}^{w_{\gamma}}$ is a \emph{quasi-polynomial}
in this case. Degrees and periods of these quasi-polynomials will be estimated
in the Appendix in \cite{KS} (written by first and third authors joint
with Kenta Suzuki), the main technical tool is the localization theorem
in equivariant algebraic $K$-theory.

\begin{Example}%
\label{vdist}
A regular element is always very distinguished. A subregular element is
very distinguished except in types $A_{n}$ and $B_{n}$, when it is not
distinguished. There are also other examples: one in $F_{4}$, two in
$E_{7}$, three in $E_{8}$, etc., see e.g. \cite{A}.
\end{Example}

\begin{Rem}
Recall (see Proposition~\ref{descr_cell_subreg}) that the cell $c$ such
that ${\mathbb{O}}_{c}$ is subregular consists of elements $w\ne 1$ with
a unique minimal decomposition (see Proposition~\ref{descr_cell_subreg} below). A similar (but more complicated) description
of the next case, which includes most very distinguished examples, should
follow from \cite{GX}.
\end{Rem}

Another interesting case to consider is the case of
$\mathfrak{g}=\mathfrak{sl}_{n}$ and $e$ being the two-block nilpotent,
see \cite{AV} for the parametrization and description of the irreducible
objects in the heart of the exotic $t$-structure in this case.

\appendix
\section{Basic facts about Kazhdan-Lusztig polynomials}
\label{basic_KL}

\subsection{Canonical and standard bases in Hecke algebra ${\mathcal{H}}_{q}(\widehat{W})$}

The Hecke algebra ${\mathcal{H}}_{q}={\mathcal{H}}_{q}(\widehat{W})$ over
${\mathbb{Z}}[q^{\pm 1}]$ is an ${\mathbb{Z}}[q^{\pm 1}]$-algebra with
free ${\mathbb{Z}}[q^{\pm 1}]$-basis $\{H_{w}\}_{w \in \widehat{W}}$ whose
multiplication is determined by the following:
\begin{gather*}
H_{w}H_{v}=H_{wv}~\text{if}~\ell (wv)=\ell (w)+\ell (v),%
\\
(H_{s}+1)(H_{s}-q)=0~\text{for}~s \in \{s_{0},s_{1},\ldots ,s_{r}\}.%
\end{gather*}
We can define the Hecke algebra ${\mathcal{H}}_{q}(W)$ similarly.

We define an involutive ring endomorphism
${\mathcal{H}}_{q}(\widehat{W}) \ni h \mapsto \overline{h} \in {
\mathcal{H}}_{q}(\widehat{W})$ by
\begin{equation*}
\overline{\sum _{w \in \widehat{W}} a_{w} H_{w}}=\sum _{w \in
\widehat{W}} \overline{a}_{w} H_{w^{-1}}^{-1},
\end{equation*}
where $\overline{q}=q^{-1}$. Let $\preccurlyeq $ be the Bruhat order on
$\widehat{W}$.

\begin{Prop}[\cite{kl0}]\label{canon_in_Hecke}
For any $v \in \widehat{W}$ there exists a unique
$C_{v} \in {\mathcal{H}}_{q}(\widehat{W})$, satisfying the following conditions:
\begin{equation*}
C_{v}=\sum _{w \preccurlyeq v} P_{w,v}(q)H_{w}
\end{equation*}
with $P_{v,v}(q)=1$ and $P_{w,v}(q) \in {\mathbb{Z}}[q]$ of degree
$\leqslant (\ell (v)-\ell (w)-1)/2$ for $w \prec v$,
\begin{equation*}
\overline{C}_{v}=q^{-\ell (v)}C_{v}.
\end{equation*}
\end{Prop}
The polynomials $P_{w,v}(q)$ are called Kazhdan-Lusztig polynomials. Let
us now introduce {\textit{inverse}} Kazhdan-Lusztig polynomials
${\bf{m}}^{w}_{v}(q)$ (see \cite[Section~2]{kl01} where they are denoted
by $Q_{v,w}(q)$). These polynomials are determined by:
\begin{equation*}
H_{w}=\sum _{v \preccurlyeq w}\varepsilon (wv^{-1}){\bf{m}}^{w}_{v}(q)
C_{v}.
\end{equation*}

\begin{Rem}
Note that in \cite{kl01} the polynomials ${\bf{m}}^{w}_{v}(q)$ are denoted
by $Q_{v,w}(q)$. We change this notation since we have already reserved
``$Q$'' for the coroot lattice $Q^{\vee}$.
\end{Rem}

\begin{Lem}%
\label{symm_m_inv}
For $v, w \in \widehat{W}$ we have
${\bf{m}}^{w}_{v}(q)={\bf{m}}^{w^{-1}}_{v^{-1}}(q)$.
\end{Lem}
\begin{proof}
Consider the anti-automorphism $i$ of ${\mathcal{H}}_{q}(\widehat{W})$ given
by $i(H_{x})=H_{x^{-1}}$, $i(q)=q$. Map $i$ commutes with the involution
$\overline{\bullet }$. The claim follows.
\end{proof}

The following Theorem holds by \cite[Section 0.3]{KT} together with Lem\-ma~\ref{symm_m_inv}:

\begin{Thm}%
\label{char_dom_via_inverse}
For a dominant integral weight
$\lambda \in \widehat{\mathfrak{h}}^{*}$ we have
\begin{equation*}
\operatorname{ch} L(v^{-1} \circ \lambda )=\sum _{w \in \widehat{W}}
\varepsilon (wv^{-1}){\bf{m}}^{w}_{v}(1)\operatorname{ch}M(w^{-1}
\circ \lambda ).
\end{equation*}
\end{Thm}

\subsection{Canonical and standard bases in the anti-spherical module}

Define the algebra homomorphism
$\chi \colon {\mathcal{H}}_{q}(W) \rightarrow {\mathbb{Z}}[q]$ by
$\chi (H_{w})=\varepsilon (w)$. We define the induced module
${\mathcal{M}}$ (anti-spherical module over
${\mathcal{H}}_{q}(\widehat{W})$) by
\begin{equation*}
{\mathcal{M}}:= {\mathcal{H}}_{q}(\widehat{W}) \otimes _{{\mathcal{H}}_{q}(W)}
{\mathbb{Z}}[q]
\end{equation*}
and define
$\varphi \colon {\mathcal{H}}_{q} \twoheadrightarrow {\mathcal{M}}$ by
$\varphi (h)=h \otimes 1$.

It is easily checked that
${\mathcal{M}}\ni m \mapsto \overline{m} \in {\mathcal{M}}$ is well defined
by
\begin{equation*}
\overline{\varphi (m)}=\varphi (\overline{m}).
\end{equation*}

For $\gamma \in Q^{\vee}$ set
$H'_{w_{\gamma}}:=\varphi (H_{w_{\gamma}})$. It is easily seen that
${\mathcal{M}}$ is a free ${\mathbb{Z}}[q^{\pm 1}]$-module with basis
$\{H'_{w_{\gamma}}\}_{\gamma \in Q^{\vee}}$.

\begin{Prop}[\cite{deo}]\label{canon_in_antisph}
For any $\nu \in Q^{\vee}$ there exists a unique
$C_{w_{\nu}}' \in {\mathcal{M}}$, satisfying the following conditions.
\begin{equation*}
C'_{w_{\nu}}=\sum _{\gamma \in Q^{\vee},\, w_{\gamma }\preccurlyeq w_{
\nu}} \tilde{P}_{w_{\gamma},w_{\nu}}(q) H'_{w_{\gamma}}
\end{equation*}
with $\tilde{P}_{w_{\nu},w_{\nu}}(q)=1$ and
$\tilde{P}_{w_{\gamma},w_{\nu}}(q) \in {\mathbb{Z}}[q]$ of degree
$\leqslant (\ell (w_{\nu})-\ell (w_{\gamma})-1)/2$ for
$w_{\gamma }\prec w_{\nu}$.
\begin{equation*}
\overline{C'_{w_{\nu }}}=q^{-\ell (w_{\nu})}C'_{w_{\nu}}.
\end{equation*}
\end{Prop}
It easily follows from Proposition~\ref{canon_in_antisph} and Theorem~\ref{canon_in_Hecke} that
\begin{equation}
\label{prop_varphi_canon}
C'_{w_{\nu}}=\varphi (C_{w_{\nu}})~\text{and}~\varphi (C_{w})=0~
\text{if}~w \neq w_{\nu}~\text{for any}~\nu \in Q^{\vee},
\end{equation}
so
\begin{equation*}
{\mathcal{M}}={\mathcal{H}}_{q}/\langle C_{w}\mid  w \notin \{w_{\nu}
\mid  \nu \in Q^{\vee}\}\rangle .
\end{equation*}
Polynomials $\tilde{P}_{w_{\gamma},w_{\nu}}(q)$ are called parabolic Kazhdan-Lusztig
polynomials. Let us now define the parabolic inverse Kazhdan-Lusztig polynomials
${\tilde{{\bf{m}}}^{w_{\gamma }}_{w_{\nu}}}(q)$. Following
\cite[Equation (2.40)]{KT}, we define them by
\begin{equation*}
H_{w_{\gamma}}'=\sum _{\nu \in Q^{\vee},\, w_{\nu }\preccurlyeq w_{
\gamma}} \varepsilon (w_{\gamma }w_{\nu}^{-1}) {{\tilde{\bf{m}}}^{w_{
\gamma}}_{w_{\nu}}}(q) C'_{w_{\nu}}.
\end{equation*}

The following proposition holds by \cite{S0} (see also
\cite[Proposition 2.7]{KT}).

\begin{Prop}\label{eq_m_m_prime}
For $\nu , \gamma \in Q^{\vee}$ we have
${\tilde{\bf{m}}^{w_{\gamma}}_{w_{\nu}}}(q)={\bf{m}}^{w_{\gamma}}_{w_{
\nu}}(q)$.
\end{Prop}
\begin{proof}
Easily follows from the fact that
$\varphi (H_{w_{\gamma }u})=\varepsilon (u)H_{w_{\gamma}}',\, u \in W$
together with (\ref{prop_varphi_canon}).
\end{proof}

So we conclude that
\begin{equation*}
H'_{w_{\gamma}}=\sum _{\nu \in Q^{\vee}} \varepsilon (w_{\gamma }w_{
\nu}^{-1}){\bf{m}}^{w_{\gamma}}_{w_{\nu}}(q)C'_{w_{\nu}}.
\end{equation*}
Let us now modify bases $H'_{w_{\gamma}}$, $C'_{w_{\nu}}$ as follows:
\begin{equation*}
T_{\gamma}:=\varepsilon (w_{\gamma})H'_{w_{\gamma}},~C_{\nu}:=
\varepsilon (w_{\nu})C'_{w_{\nu}}.
\end{equation*}

We have
\begin{equation*}
T_{\gamma}=\sum _{\nu \in Q^{\vee}} {\bf{m}}^{w_{\gamma}}_{w_{\nu}}(q)
C_{\nu}.
\end{equation*}

\begin{Rem}
The numbers ${\tilde{\bf{m}}^{w_{\gamma}}_{w_{\nu}}}(1)={\bf{m}}^{w_{\gamma}}_{w_{
\nu}}(1)={\bf{m}}^{w_{\gamma}}_{w_{\nu}}$ are matrix coefficients of the
transition matrix from classes of standard sheaves on the affine Grassmannian
of $G$ to classes of irreducible objects ($IC$ sheaves) (see \cite[Corollary~5.5]{KT} for details). Set $M={\mathcal{M}}/(q-1)={\mathbb{Z}}\widehat{W}
\otimes _{{\mathbb{Z}}W} {\mathbb{Z}}_{\mathrm{sign}}$. After the
identification $K^{G^{\vee}}(\widetilde{{\mathcal{N}}}) \simeq M$ elements
$T_{\gamma}$ become $[{\mathcal{O}}_{\widetilde{{\mathcal{N}}}}(\gamma )]$
and $C_{\gamma}$ are classes of irreducible objects in the heart of the ``exotic'' $t$-structure on
$D^{b}(\operatorname{Coh}^{G^\vee}(\widetilde{{\mathcal{N}}}))$ (see \cite{AB,BHum,BHe} for details).
\end{Rem}

\end{document}